\documentclass[10pt,leqno,a4paper]{article}

\usepackage{amsmath} 
\usepackage{amsthm}
\usepackage{amssymb}

\usepackage{yhmath} %%% crea problemi con l'espressione $\what{ \what{G} / \what{K} }$. 
\usepackage{color}

\usepackage{hyperref} %%%%%%%%%%%%%%% DANGER 
\hypersetup{
    colorlinks=true,       % false: boxed links; true: colored links
    linkcolor=blue,          % color of internal links (change box color with linkbordercolor)
    citecolor=green,        % color of links to bibliography
    filecolor=magenta,      % color of file links
    urlcolor=cyan           % color of external links
} 

\usepackage{tikz}
\usepackage{tikz-cd}
\usetikzlibrary{matrix,arrows,decorations.pathmorphing}

\usepackage{pdfsync}
\usepackage[english]{babel}
\usepackage[totalheight=22 true cm, totalwidth=14 true cm]{geometry}

\usepackage{setspace}
\usepackage{mathrsfs}

\newtheorem{thm}{Theorem}[section]
\newtheorem{cor}[thm]{Corollary}
\newtheorem{lemma}[thm]{Lemma}

\newtheorem{prop}[thm]{Proposition}

\newtheorem{conv}[thm]{Convention}

\newtheorem{defn}[thm]{Definition}

\theoremstyle{remark}

\theoremstyle{definition}

\newtheorem{rmk}[thm]{Remark}

\newtheorem{notation}[thm]{Notation}

\numberwithin{equation}{thm}
\def\beq{\begin{equation}}
\def\eeq{\end{equation}}
\def\beqa{\begin{equation*}}
\def\eeqa{\end{equation*}}
\def\ben{\begin{enumerate}}
\def\een{\end{enumerate}}
\def\besp{\begin{split}}
\def\eesp{\end{split}}

\def\crash#1{}
\def\N{{\mathbb N}}

\def\Z{{\mathbb Z}}
\def\0{{\mathbb O}}

\def\P{{\mathbb P}}
\def\Q{{\mathbb Q}}
\def\R{{\mathbb R}}
\def\C{{\mathbb C}}

\def\A{{\mathbb A}}

\def\F{{\mathbb F}}

\def\G{{\mathbb G}}

\def\P{{\mathbb P}}

\def\V{{\mathbb V}}

\def\1{{\mathbf 1}}

\def\l{\left}
\def\r{\right}
\def\[[{\l[\l[}
\def\]]{\r]\r]}

\def\padic{{p-{\rm adic}}\,}

\def\neat{{\rm neat}}

\def\ie{\emph{i.e.}\,}
\def\lc{\emph{loc.cit.\,}}
\def\eg{\emph{e.g.}\,}

\def\cA{{\mathcal A}}

\def\cC{{\mathcal C}}
\def\cD{{\mathcal D}}

\def\cF{{\mathcal F}}
\def\cG{{\mathcal G}}

\def\cM{{\mathcal M}}

\def\cO{{\mathcal O}}

\def\cK{{\mathcal K}}
\def\cL{{\mathcal L}}

\def\cP{{\mathcal P}}

\def\cS{{\mathcal S}}
\def\cT{{\mathcal T}}

\def\cU{{\mathcal U}}

\def\g"{``}
\def\g'{`}

\def\sB{{\mathscr B}}
\def\sC{{\mathscr C}}
\def\sD{{\mathscr D}}

\def\sF{{\mathscr F}}
\def\sH{{\mathscr H}}
\def\sJ{{\mathscr J}}
\def\sK{{\mathscr K}}

\def\sP{{\mathscr P}}
\def\sQ{{\mathscr Q}}
\def\sR{{\mathscr R}}

\def\sU{{\mathscr U}}

\def\fP{{\mathfrak P}}

\def\fX{{\mathfrak X}}

\def\wtilde{\widetilde}
\def\what{\widehat}
\def\nwhat{\wideparen}

\def\veps{\varepsilon}

\def\Loc{{\rm Loc}}

\def\inc{{\rm inc}}

\def\pfg{{\rm pfg}}
\def\dis{{\rm dis}}

\def\vect{{\rm W}}

\def\weak{{\rm weak}}
\def\strong{{\rm strong}}

\def\unif{{\rm unif}}
\def\univ{{\rm univ}}

\def\acs{{\rm acs}}

\def\cs{{\rm cs}}

\def\id{{\rm id}}
\def\Ker{{\rm Ker}}
\def\Coker{{\rm Coker}}

\def\dis{{\rm dis\,}}

\def\ol{\overline}
\def\iso{\xrightarrow{\ \sim\ }}
\def\map#1{\xrightarrow{\ #1\ }}
\def\til#1{\widetilde{#1}}

\def\limit{{\rm lim}}
\def\colimit{{\rm colim}}
\def\PRODcan{{\prod}^{\can}}
\def\PRODsqu{\prod^{\square,\un}}

\def\Mod{{\cM}od}

\def\Top{{\cT}op}

\def\cLMu{{\cL\cM}^\un}
\def\cSLMu{{\cS\cL\cM}^\un}

\def\cCLMcan{{\cC\cL\cM}^{\rm can}}

\def\SUMU{ \mathop{{\bigoplus}^\un}\limits}

\def\cCLMou{\cC\cL\cM^{\omega,\rm u}}
\def\cCLMu{\cC\cL\cM^{\rm u}}

\def\un{{\rm u}}
\def\tu{{\bf t}}

\def\for{{\rm for}}

\def\Inv{{\rm Inv}}

\def\can{{\rm can}}
\def\pfg{{\rm pfg}}
\def\Bil{{\rm Bil}}
\def\hom{{\rm Hom}}
\def\End{{\rm End}}
\def\Inv{{\rm Inv}}

\def\wt{\what{\otimes}}

\author{Francesco Baldassarri
\thanks{Universit\`{a} di Padova,
Dipartimento di Matematica, Via Trieste, 63, 35121 Padova, Italy.}
 }
\title{Non-archimedean integration on totally disconnected spaces.}

\begin{document}
%%%%%%%%%%%%%%%%%%%%%%%%%%%%%%%%%%%%%
%%\begin{dedication}
%%\hspace{4cm}
%%\vspace*{6ex}{dedicated to Philippe Robba, with gratitude.}
%%\end{dedication}
%%%%%%%%%%%%%%%%%%%%%%%%%%%%%%%%%%%%%%%%
\date{\today}

\maketitle

\begin{abstract}
\smallskip
We work in the category $\cCLMu_k$ of \cite{closed} of separated complete bounded  $k$-linearly topologized modules over a complete linearly topologized ring $k$ 
 and 
discuss  duality on certain exact subcategories. We study topological and 
uniform structures on locally compact paracompact $0$-dimensional topological spaces $X$, named \emph{$td$-spaces} in \cite{cartier} and \cite{ngo}, and the corresponding algebras  $\sC_?(X,k)$ of  continuous   $k$-valued 
functions, with a choice of 
support and 
uniformity conditions. We apply the previous duality theory to define and study the dual coalgebras $\sD_?(X,k)$  of  $k$-valued measures on $X$.  
 We then complete the picture by providing a direct definition of the various types of measures.
   In the case of $X$ a commutative $td$-group $G$ the integration pairing provides perfect dualities of Hopf $k$-algebras 
   between 
   $$
 \sC_\unif(G,k)   \longrightarrow 
\sC(G,k) \;
\;\;\mbox{and}\;\;\; \sD_\acs(G,k) 
  \longrightarrow   \sD_\unif(G,k)   \;.
$$ 
We conclude the paper with 
 the remarkable example 
 of $G= \G_a(\Q_p)$ and $k = \Z_p$,  leading to the basic Fontaine ring   
 $${\bf A}_{\inf}   =   \vect \l(\what{\F_p[[t^{1/p^\infty}]]}\r) = \sD_\unif(\Q_p,\Z_p) \;.$$
 %the subring  $\sD_\unif(\Q_p,\Z_p) = \vect (\sD_\unif(\Q_p,\F_p))$ of 
% of Fontaine's ring ${\bf A}_{\inf} = \vect \l(\C_p^{\flat \circ}\r)$, \bmau where $t=?$\emau,
% $\sD_\unif(\Q_p,\Z_p) = \vect (\sD_\unif(\Q_p,\F_p)) = {\bf A}_{\inf} =   \vect (C^\circ)$, where $C$ is any topological field isomorphic to $\what{\Q_p(\zeta_{p^\infty})}$
 %$\C^\flat_p$, 
We discuss Fourier duality between ${\bf A}_{\inf}$ and $\sC_\unif(\Q_p,\Z_p)$ and exhibit a remarkable Fr\'echet basis of 
$\sC_\unif(\Q_p,\Z_p)$ related to the classical binomial coefficients. 
% the remarkable examples
% of $X= \Z_p$ or $=\Q_p$, and $k =$ the $p$-adic ring  $\Z_p$ or the finite field $\F_p$ and sketch  the corresponding Fourier theory.
\end{abstract}

\tableofcontents

\bigskip
\setcounter{section}{-1}
\begin{section}{Conventions} \label{conventions}
\par
A prime number $p$ is fixed throughout this paper although it will only   play a role in the applications of the last sections~\ref{Integrations} to \ref{fourier}.
We fix a universe $\cU$ and assume  every object we deal with, except $\cU$ itself,  is ``small'', \ie is an object of $\cU$. In particular limits and colimits are understood to be ``small''. Unless otherwise specified, a  \emph{ring} is meant to be commutative with 1. For any ring $R$, $\Mod_R$ will denote the category of $R$-modules. It is a bicomplete abelian category. For $M \in \Mod_R$ and any small set $A$, $M^A$ (resp. $M^{(A)}$) will denote the product (resp. sum) in $\Mod_R$ of the family $(M_\alpha)_{\alpha \in A}$ with $M_\alpha = M$ for any $\alpha \in A$. 
 A topological ring is \emph{linearly topologized} if it admits a basis of neighborhoods of $0$ consisting of ideals. If $R$ is a linearly topologized ring, a topological $R$-module $M$ is said to be \emph{linearly topologized} if $M$ admits a basis of open neighborhoods of $0$ consisting of $R$-submodules. 
 The ring underlying the topological ring $R$ will be denoted  by $R^\for$; similarly, the $R^\for$-module underlying the topological  $R$-module $M$, will be denoted  by $M^\for$. Such care will often be waived unless there is a risk of confusion.  
 \par
 If $K$ is a non-archimedean field, both the ring of integers $K^\circ$ and the residue field $\til{K} = K^\circ/K^{\circ \circ}$ are linearly topologized rings, while $K$, unless its valuation is trivial,  is only linearly topologized as a topological $K^\circ$-module  \par 
  Let $R$ be a  linearly topologized ring and let $M$ be a linearly topologized $R$-module. A formal series 
$\sum_{\alpha \in A} m_\alpha$ of elements $m_\alpha \in M$, indexed by any set $A$, is said to be (an $A$-series) \emph{unconditionally convergent to $m \in M$} if, for any open submodule $N$ of $M$, there is a finite subset $A_N$ of $A$
such that  
 $\sum_{\alpha \in F} m_\alpha \in m + N$  for  any finite set $F$, $A_N \subset F \subset A$. Then  $m_\alpha \in N$ for any $\alpha \notin A_N$ and we write $m = \sum_{\alpha \in A} m_\alpha$.  If $M$ is separated and complete and $\{m_\alpha\}_{\alpha \in A}$ is a family of elements of $M$ such that, for any open $R$-submodule $P$ of $M$, $m_\alpha \in P$, for almost all $\alpha \in A$, then 
 $m = \sum_{\alpha \in A} m_\alpha$, for a unique $m \in M$.
In particular, if $M$ is discrete, an $A$-series $\sum_{\alpha \in A} m_\alpha$ of elements $m_\alpha \in M$ is unconditionally convergent if and only if $m_\alpha = 0$ for all but a finite set of $\alpha \in A$. In that case the $A$-series converges to the finite sum $\sum'_{\alpha \in A} m_\alpha$ of non-zero $m_\alpha$'s. 
\par
  For any set $X$, $\fP(X)$ (resp. $\sF(X)$) will be the family of (resp. finite) subsets of $X$. We use the term \emph{clopen} to indicate a subset of a topological space $X$ which is both open and closed: the set of clopens of $X$ is denoted $\Sigma(X)$. We adopt the convention that ``compact''  or ``locally compact'' does not imply Hausdorff. The family of compact open subsets of a topological space $X$ is denoted by $\cK(X)$. 
  \par
 For any  subset $U \subset X$, we denote by $\chi_U^{(X)} = \chi_U^{(X,k)}: X \to \{0,1\} \subset k$, or just $\chi_U$,  the characteristic function of $U \subset X$. If $U = \{x\}$ we also write $\chi_x^{(X)}$, or simply $\chi_x$,  for $\chi_{\{x\}}^{(X)}$. If $X$ is a topological space and $U \in \Sigma(X)$,  then 
$\chi_U^{(X)}$ is a continuous function $X \to k$.  
\par
We denote by $\Top$ the category of topological spaces and continuous maps.
\par
 For generalities on quasi-abelian categories we refer to \cite{schneiders}. In a quasi-abelian category $\sC$, a \emph{strict subobject} (resp. a \emph{strict quotient object}) of $X \in \sC$ is the kernel (resp. cokernel) of a $\sC$-morphism $X \to Y$ (resp. $Y \to X$). In the category $\Top$ an \emph{open} (resp. \emph{closed}) \emph{embedding} $X \to Y$ is a homeomorphism of $X$ to an open (resp. closed) subspace $X'$ of $Y$. In the categories of complete topological modules considered in this paper a strict subobject is not in general a closed embedding of the underlying topological spaces. See Remark~\ref{compunif2} below.
% \par A piece of unusual notation is that, for any set $X$ and ring $k$,  we  denote by $\Delta_x$ (resp. $\delta_x$)  a $k$-valued Dirac measure centered at $x \in X$ if $k$ has characteristic 0 (resp. $\neq 0$). 
 
 %%%%%%%%%%%%%%%%%%%%%%%%%%%%%%%%%%%%%%%%%%%%%%%%%%%%%%%%%%%%%%%%%%%%%%%%%%%%%%%%%%%%%%%%%%%%%%%

\section{Introduction} \label{intro} 
The origin of the present discussion lies in a problem   my Teacher Iacopo Barsotti had suggested 
in view of my 1974 Padova Thesis. Namely, the function-theoretic interpretation of  self-duality for  
$p$-divisible groups attached to an ordinary abelian variety in characteristic $p>0$. 
In particular, that had to include a functional interpretation of some of the Barsotti-Witt rings of bivectors attached by Barsotti 
to any $p$-divisible group. After partial results obtained at that time, the problem was left dormient for 
over two score because of  difficulties related to topological algebra, and in particular to duality,  over a $p$-adic 
ring and for lack of an immediate application. 
More recently, I came back to this question~: the motivation was  to show that a certain entire 
$p$-adic analytic function I had encountered in my thesis \cite{psi}, \cite{psi2} could be useful in the construction of an analytic Fourier theory for bounded uniformly continuous functions on $\Q_p$ with $p$-adically entire test functions  very similar to the Fourier expansion of a function $\R \map{} \C$ in terms of $\{\exp{{2 i \pi z}\over n}\}_{n \in \N}$; this will appear in \cite{psi3}. Moreover, we had realized over time  that the weak dual of the Hopf $\Z_p$-algebra of uniformy continuous functions $\G_a(\Q_p) \map{} \Z_p$, namely the Hopf $\Z_p$-algebra $\sD_\unif(\Q_p,\Z_p)$ of ``uniform'' $\Z_p$-valued measures on $\G_a(\Q_p)$, equipped with the topology of uniform convergence on  balls of one same radius
 is nothing but Fontaine's ring ${\bf A}_{\inf} = \vect (\cO_{C})$ on the perfectoid field $C$ = $t$-adic  completion of 
  $\F_p((t^{1/p^\infty}))$, where $t^{1/p^n} = \delta_{1/p^n} - \delta_0$ for $n \in \N$, and $ \vect (\cO_{C})$ is equipped with the product topology. 
Of course, the realization of a non trivial intersection with Fontaine's theory \cite{FFcurve} was a major motivation for completing this research. 
\par
Foundations of functional analysis over a very general linearly topologized ring is separately discussed in \cite{closed}, but we have summarized the essentials in 
sections~\ref{lintop} and \ref{duality} so that, with the precious reference to \cite[Chap. 8]{GR}, this paper stays self-contained. 

\par \medskip We fix  a separated and complete linearly topologized ring $k$ satisfying properties
 \eqref{kaxiom} below (\eg \, $k = \F_p$ or $\Z_p$, but not $k =\Q_p$, where $p$ is a prime number).
  We develop the basics of $k$-valued integration theory over  locally compact paracompact $0$-dimensional  topological (or uniform) spaces~:  we call them  ``$td$-spaces'' as in \cite{ngo} and \cite{cartier}. 
  This category of spaces is sufficiently general to include  non-archimedean locally analytic varieties (and in particular   locally analytic groups) 
  of the type considered by Schneider and Teitelbaum \cite{ST1}, \cite{ST2} which arise from finite field extensions $L \supset K \supset \Q_p$, as the set   $V = \V(L)$(resp. $G =\G(L)$) of $L$-valued points of an algebraic variety  $\V$ (resp. algebraic group $\G$) defined over $K$.  
  In that case, for any further completely valued  extension $E \supset K$, our methods permit to study continuous  (resp. uniformly continuous) functions on $V$ and $G$ with values 
 in $k=\wtilde{E}$ or $k = E^\circ$, and the dual notions of  measures on $V$ and $G$ which will be introduced below. 
% The main application of  the theory of integration on  $td$-groups we have in mind refers to the study of $p$-adic Mellin transforms and   
% analyticity of the corresponding $p$-adic zeta and L-functions. 
%  on absolute galois groups of finite extensions of $\Q_p$ (compact case) but also on the correponding Weil groups (locally compact $td$ case).  
  Notice that we do not consider explicitly here \emph{$E$-valued unbounded} functions and measures on $V$ or $G$, nor do we discuss differentiability or local analyticity of functions and the corresponding spaces of distributions. 
 We content ourselves with an  abstract  duality theory for  (resp. uniformly)  continuous  bounded functions and measures on (resp. uniform)   $td$-spaces and $td$-groups.
% \bmau  \par Important examples  of $td$-spaces also 
% arise from  an algebraic  variety $\cV$  defined over a global field $K$ and a finite extension $L$ of $K$. Here the role of $k$ is played by the compact ring $\prod_{v \in M^f_K} \cO_{K,v}$, where $M_K^f$ is the set of non-archimedean places of $K$. Let $A_L^f$ denote  the non-archimedean part of the adele ring $A_L$ of $L$. Then the set $V= \cV(A_L^f)$ also carries a natural structure of $td$-space. In the same vein the absolute galois group of $K$ as well as its Weil group are examples of $td$-groups to which our study of $k$-valued continuous functions and measures applies. \emau
  \par 
  %There are at least two reasons why we think it is useful to revisit these topics. 
  We decided  to revisit these topics because  topological algebra over a linearly topologized ring presents remarkable difficulties, related to the description of limits and colimits, and in particular of kernels and cokernels.  
 Such  categories of topological modules are in general not abelian~: in some cases of interest we get however quasi-abelian categories \cite{schneiders}, namely exact categories with kernels and cokernels. We are especially interested in the case when those categories carry a closed monoidal structure compatible with the additive structure. 
 We deal  with these problems systematically in \cite{closed},  and summarize in section~\ref{lintop}  below the   results which are relevant for our present purposes.  Our approach to topological algebra borrows from the one of  Gabber and Ramero \cite{GR}, and we refer to their very informative book project whenever helpful. 
  \par  
 The   difficulties arising in  topological algebra recently led Clausen and Scholze to base their theory of analytic stacks on the more robust notion of \emph{condensed sets}, namely sheaves on a Grothendieck site of profinite sets \cite{condensed}. More recently still,  they chose to restrict  their foundational universe to \emph{light} condensed sets so obtaining a more straightforward theory \cite{PSanalytic}. We were tempted to follow  their approach, but we were facing two difficulties. On the one hand, the most interesting examples of integration pairing we had in mind involve   \emph{uniform} $td$-spaces. Such notion however does not appear, as far as we understand, in the present state  of art of Condensed Mathematics. For viewing uniform  $td$-spaces as condensed sets, we might have adopted Andr\'e's notion of a \emph{uniform sheaf}, nicely explained in  \cite{yves}; 
but that looked like a long term project. 
\par On the other hand, part of the charm  of Fourier theory lies  in the possibility of exhibiting explicit  dual topological bases for spaces of functions and of measures, which could hardly be done in a general framework of triangulated categories of 
quasi-coherent sheaves. Certainly, in the setting  of topological $k$-modules duality is not well-behaved in general~: we did our best in some particular cases of perfect duality. For example, in section~\ref{duality} we define weak and strong duals and establish the yoga for their application  to limits, colimits, tensor products in the category $\cCLMu_k$. 
Those results seem to be new and are used all over the present paper.   
  \par \smallskip
Section~\ref{STS} is completely topological. It provides a fairly complete description of  $td$-spaces and of a certain type of uniform structures, called \emph{$td$-uniformities}, on them. As mentioned above, a $td$-space may naturally  be seen as a 
condensed set, and the restriction to light condensed sets would be the proper setting to discuss \emph{metrizable} $td$-spaces of which we give an \emph{ad hoc} characterization. We dedicate some effort to describe  the family of  $td$-uniformities $(X,\Theta_\sJ)$  (see section~\ref{uniftd} below) on a given $td$-space $X$. For us $\sJ$ is a cofiltered projective sub-system of the system $\sQ(X)$ of all discrete quotient spaces $\pi_D:X \longrightarrow D$ of $X$, where the quotient map $\pi_D$ is proper. It turns out that in the category of topological spaces (resp. of uniform spaces) for any $td$-space $X$ (resp. for any $td$-uniform space $(X,\Theta_\sJ)$)
\beq \label{limitexp} 
X = \colimit_{C \in \cK(X)} C = \limit_{D \in \sQ(X)} D \;\;\mbox{(resp.}\; (X,\Theta_\sJ) = \limit_{D \in \sJ} (D,\Theta_D) \;\mbox{)} 
\eeq
(where the uniformity $\Theta_D$ on the discrete space $D$ is meant to be the discrete uniformity). 
\par \smallskip 
In section~\ref{STScontinuity} we analyze the structure of the topological $k$-module  $\sC(X,k)$ of  continuous  functions on 
a $td$-space  $X$ equipped with the topology of uniform convergence on compact subsets of $X$. For a $td$-uniform space $(X,\Theta_\sJ)$  we similarly analyze the space $\sC_\unif((X,\Theta_\sJ),k)$
of  uniformly continuous  functions,    with the topology of uniform convergence on $X$.  
%In section~\ref{STScontinuity} we analyze the structure of the topological $k$-module  $\sC(X,k)$ (resp. $\sC_\unif((X,\Theta_\sJ),k)$) of (resp. uniformly) continuous  functions on 
%(resp. $td$-uniform) $td$-spaces $X$ (resp. $(X,\Theta_\sJ)$), equipped with the topology of uniform convergence on compact subsets of $X$ (resp. of uniform convergence on $X$). 
We also introduce the  space  $\sC_\acs(X,k)$, of functions with  \emph{almost compact support} (equivalently, \emph{vanishing at infinity}),   with the topology of uniform convergence on $X$. For any uniform structure $(X,\Theta_\sJ)$ on $X$, 
$\sC_\acs(X,k)$ is a closed subspace of $\sC_\unif((X,\Theta_\sJ),k)$, while it is a dense subset of $\sC(X,k)$. 
Our strategy for understanding any of the topological $k$-modules of continuous functions  above, for a general $td$-space $X$,  is based on a previous discussion of the simpler special  cases of $X=C$  a Stone space, and $X=D$ a discrete space. We then use the description~\ref{limitexp} to clarify the structure of the previous topological $k$-modules in the general case in terms of limits or colimits of simpler structures. 
\par \smallskip
In the following section~\ref{measdual} we discuss 
the dual structures of spaces of measures; their topology comes in the classical variants of ``strong'' and ``weak''. We concentrate on the choices that will be more useful to us in the future.  We use intensively the results of section~\ref{duality}. Duality considerations  appear both in  \cite{ST1}, \cite{ST2}, and in \cite{colmez2}, \cite{colmez3}, but they are there finalized to the study of $p$-adic representations and  appeal to  sophisticated methods of  $p$-adic functional analysis \emph{over a  field}. 
We work instead on first principles over a  linearly topologized ring  $k$ with topological $k$-modules of continuous functions and measures. Most of these results are quite subtle and, to the best of our knowledge,  new. 
\par  \smallskip
Our conceptually   naive approach based on the general description of  certain subcategories of  $\cCLMu_k$  \cite{closed}
and on some basic duality results, still permits to distinguish important patterns. 
After the abstract definition in section~\ref{measdual}  of dual spaces of measures for various spaces of continuous functions, section~\ref{meastheory} proceeds with an  independent  and explicit   description of  $k$-valued measures on $X$ as functions on measurable sets. \par \smallskip
 It turns out that the weak dual $\sD(X,k)$ (resp. $\sD_\acs(X,k)$)  of the space $\sC_\acs(X,k)$ (resp. $\sC(X,k)$)
%equipped with the topology of uniform convergence on compact subsets of $X$,  
 admits a natural direct description (see  Proposition~\ref{measinter}).
Moreover, the weak dual of $\sD(X,k)$ (resp. $\sD_\acs(X,k)$) identifies with $\sC_\acs(X,k)$ (resp. $\sC(X,k)$) \eqref{bidual1} (resp. \eqref{generalbidual}). 
%The same instance of weak biduality holds  for $\sC_\acs(X,k)$ and $\sD(X,k)$. 
In partial contrast, if $X$ is equipped with a $td$-uniform structure $\Theta_\sJ$,  the space of \emph{uniform measures on $(X,\Theta_\sJ)$}  is the weak dual $\sD_\unif((X,\Theta_\sJ),k)$ of $\sC_\unif((X,\Theta_\sJ),k)$ \eqref{funcunifdual}, while the latter is the \emph{strong} dual of the former \eqref{measunifdefst}.  The weak dual of $\sD_\unif((X,\Theta_\sJ),k)$ is instead  the space $\sC(X,k)$ of continuous functions on $X$ \eqref{funcunifdef}. 
%The space of uniform measures is defined independently of duality in section~\ref{meastheory}. 
% We refer to Corollary~\ref{uniffact}, which is our main result on non-archimedean integration pairing on $td$-spaces, for a precise statement. 
  \par \smallskip
  %Due to our lack of competence in non-commutative geometry, 
 After the very general theory developed so far,  section~\ref{td-groups} is regrettably restricted to commutative $td$-groups and is still perhaps too short.  We have summarized here without proofs some well-known constructions for commutative formal groups which will come out handy in the next sections. 
  We hope to be 
  able to come back to this topic in the future seriously enough to be able to illustrate some instance of Fourier theory on a classical non-compact non-commutative $td$-group. The $td$-groups studied by Wesolek \cite{wesolek1}, \cite{wesolek2} seem 
  to offer a rich area of application.
\par \smallskip
We then come to the example we had in mind, namely Fourier theory for bounded (uniformly) continuous $p$-adic functions on $\Q_p$. 
\par \smallskip In section~\ref{MAss} we  present  in a slightly unusual way the Mahler-Amice theory for $\Z_p$-valued  functions and measures on $\Z_p$. 
 We explain our view of the Iwasawa identification of the formal torus $\Z_p[[T]]$ (where $\P(T) = T \wt 1 + 1 \wt T + T \wt T$)   with the Hopf algebra $\sD(\Z_p,\Z_p) = \sC(\Z_p,\Z_p)'_\weak$ of $\Z_p$-valued  measures on $\Z_p$, via the map $T \longmapsto \Delta_1- \Delta_0$ (see $\mathit 3$  of Proposition~\ref{mahleramice} below). 
 An observation  which does not seem to appear explicitly in the literature is the fact 
  that the  $(p,T)$-adic topology  on $\Z_p[[T]]$ corresponds to  the topology 
 of uniform convergence on any family of balls $a + p^h \Z_p \subset \Z_p$ of the same radius $p^h$  on $\sD(\Z_p,\Z_p)$. We name it the \emph{natural topology} of $\sD(\Z_p,\Z_p)$. The proof is actually not immediate (see Proposition~\ref{Iwasawa3} below). 
 \par \smallskip
  In section~\ref{Intpairingss} we extend the results of section~\ref{MAss}  to the case of the Hopf algebra $\sD_\unif(\Q_p,\Z_p)$ of $\Z_p$-valued uniform measures on $\Q_p$. We show that the $p$-adic filtration makes $\sD_\unif(\Q_p,\Z_p)$ into a strict $p$-adic ring with residue ring $\sD_\unif(\Q_p,\F_p) \cong \what{\F_p[[t^{1/p^\infty}]]}$. Here $t = \delta_1-\delta_0$, where $\delta_a$ (resp. $\Delta_a$) is the $\F_p$-valued (resp. $\Z_p$-valued) Dirac mass concentrated in $a \in \Q_p$. So, for the $p$-adic topology
\beq
\label{Tlift}
\wtilde{T} = \limit_n (\Delta_{1/p^n} -\Delta_0)^{p^n} =[\delta_1-\delta_0] \in \sD_\unif(\Q_p,\Z_p) = \vect \l(\what{\F_p[[t^{1/p^\infty}]]}\r) \;.
\eeq
The reader will then recognize the appearance of Fontaine's ring  ${\bf A}_{\inf}$ in the guise of the ring of uniform $\Z_p$-valued measures on $\Q_p$.  
Again, it is more convenient to endow $\sD_\unif(\Q_p,\Z_p)$ with the (weaker) \emph{natural topology}, namely the topology 
 of uniform convergence on balls of a given radius of $\Q_p$. Our main result here is the identification 
 of $\sD_\unif(\Q_p,\Z_p)$ equipped with the natural topology, with the completion of its subring $\Z_p[\wtilde{T}^{1/p^\infty}]$, for $\wtilde{T}$ as in \eqref{Tlift}, in the $(p,\wtilde{T})$-adic topology (see $\mathit 4$ of Proposition~\ref{unifconvballs}). 
The reason why it is reasonable  to endow $\sD_\unif(\Q_p,\Z_p)$ with the natural topology  is twofold. First of all,  
it is the topology used in the general theory of section~\ref{measdual} and it yields in particular the left-weak   and right-strong perfect duality pairing \eqref{dualex1}
\beq
\label{fourierunif}
\sD_\unif(\Q_p,\Z_p) \times \sC_\unif(\Q_p,\Z_p) \longrightarrow \Z_p\;\;,\;\; (\mu , f) \longmapsto \int_{\Q_p}f(x) d\mu(x) \;,
\eeq
 in which the choice of topology plays a crucial role. 
Secondly, the natural topology is sufficiently weak to allow the interesting changes of variable of section~\ref{ChgVarss}. In particular the \emph{canonical measure}  $\mu_\can \in \sD_\unif(\Q_p,\Z_p) = {\bf A}_{\inf}$ (see Proposition~\ref{repr1}) is related to $T$ via 
an analytic procedure, based on the Artin-Hasse exponential, which would not converge $p$-adically.     The canonical measure will play a central  role in the analytic Fourier theory of \cite{psi3}. 
\par \smallskip
In section~\ref{fourier} we develop the ``uniformly continuous'' Fourier theory for the pairing \eqref{fourierunif}.  This is a novelty of this paper.    For $q \in S := \Z[1/p] \cap \R_{\geq 0}$ we define  a uniformly continuous function ${x \choose q}^{(p)}  : \Q_p \longrightarrow \Z_p$ (see 
\eqref{binQp}). We prove orthogonality (see  \eqref{orthrel1})  of $\{{x \choose q}^{(p)} \}_{q \in S}$ and $\{\wtilde{T}^q\}_{q\in S}$ 
w.r.t. the pairing \eqref{fourierunif}. In Proposition~\ref{binQp2}, we compute the Fourier expansion of the  additive character (see \eqref{addchar})
\beq
\label{charexp}
\Q_p  \longrightarrow 1+\C_p^{\circ \circ} \;\;,\;\;
s  \longmapsto (\gamma^s)^\sharp = \sum_{q \in S} {s \choose q}^{(p)}  (\varpi^q)^\sharp \;, 
\eeq
where $\gamma = 1 + \varpi \in 1+\C_p^{\flat \circ \circ}$, where we have used the notation of \cite{scholze}.  
In the paper  \cite{psi3} the present uniformly continuous Fourier theory for $\sD_\unif(\Q_p,\Z_p)$ will be improved into  a $p$-adically entire Fourier theory in which the dual basis of $\{\mu_\can^q\}_{q\in S}$
will be a family $\{G_q(x)\}_{q \in S}$  of $p$-adically entire functions,  obtained from the entire function $\Psi$ of \cite{psi2},  such that $(G_q)_{|\Q_p} \in \sC_\unif(\Q_p,\Z_p)$ for any $q \in S$.   In \lc
 formula \eqref{charexp} will be ameliorated into the analytic formula
\beq
\label{charanal}
(\gamma^x)^\sharp = \sum_{q \in S} G_q(x) \int_{\Q_p} (\gamma^s)^\sharp \, d\mu_\can^q(s) 
\eeq
 which has applications to the theory of $p$-adic Mellin transforms. The existence and properties of $\mu_\can \in  {\bf A}_{\inf}$ have gone unnoticed till now. 
 
\begin{subsection}{Aknowledgements} The main ideas of this paper had already appeared in my thesis, and were largely inspired by my Teacher Iacopo Barsotti.   I wish to express my gratitude to Maurizio Cailotto who patiently listened to my ideas and often corrected them.  
 I profited from  discussions with  Umberto Marconi who helped me to understand the topology of locally compact $0$-dimensional spaces and the uniformities on them. 
 \par
 Finally, I want to express my gratitude to Adrian Iovita for his constant encouragement. 
% Finally, the possible applications of \cite{psi3} to an analytic extension of the $p$-adic zeta function emerged in conversations with Adrian Iovita: I thank him heartily for that. 
\end{subsection}  
\end{section}

\begin{section}{Generalities on linear topologies} \label{lintop}
 \begin{subsection}{Linear topologies} \label{lincat}
Let $k$ be a  separated and complete linearly topologized  ring.  We require that $k$ satisfies  the following properties
\begin{itemize}\label{kaxiom}
\item ($\omega$-admissibility) $k$ has a countable basis of open ideals;
\item (the ``finiteness'' condition) any open ideal of  $k$ contains a finitely generated open ideal;
\item (the ``op'' condition) 
for any open ideals $I,J$ of $k$,  the ideal $IJ$ is open in $k$. 
\end{itemize}
 We will also denote by $\cP(k)$ the family of open ideals of $k$.  
\begin{rmk}\label{ophodge} By \cite[Rmk. 8.3.3 (iii)]{GR} for any linearly topologized ring $k$ (satisfying $\omega$-admissibility and finiteness) the op condition is equivalent to the following condition
\begin{itemize}
\item There exists a basis $\sB$ of open ideals of $k$ such that, for any $I \in \sB$, the ideal $I^2$ is open. 
\end{itemize}
\end{rmk}

We will work in the category $\cLMu_k$   of  linearly topologized  $k$-modules $M$ such that 
the map multiplication by scalars
$$
k \times M \longrightarrow M\;\;,\;\; (r,m) \longmapsto rm
$$
is \emph{uniformly continuous} for the product uniformity of $k \times M$. This condition means that the topology of $M$ is weaker than the \emph{naive canonical topology} of $M$, the latter being the topology with fundamental system of open submodules the family $\{IM\}_{I \in \cP(k)}$. The same condition is also equivalent to  the fact that the  $k$-linearly topologized module $M$ be bounded \cite{closed}.  Morphisms of $\cLMu_k$ are continuous $k$-linear maps.  We are interested in the full subcategory $\cSLMu_k$ (resp. $\cCLMu_k$) of separated (resp. separated and complete) modules.  
The  category $\cCLMu_k$ is the one of \cite[\S 15.1.2]{GR} (implicitly described in \cite[Chap. 0, \S 7.7.1]{EGA}). Its full subcategory of $\cCLMu_k$ consisting of objects with a countable basis of open $k$-submodules is denoted by $\cCLMou_k$. In  \cite{closed} we prove that 
$\cCLMou_k$ is a quasi-abelian category with enough injectives \cite[Defn. 1.3.18]{schneiders}. 
%The  natural forgetful  functor 
%$\cCLMu_k \to \Mod_{k^\for}$ will be denoted  by $M \mapsto M^\for$. 
For any object $M$ of $\cCLMu_k$, $\cP(M)$ will denote the family of open $k$-submodules of $M$. Notice that if the topology of the object $M$ of $\cCLMu_k$ is discrete, then $IM = (0)$ for some $I \in \cP(k)$. 
%Conversely, 
%if $M \in \cCLMu_k$ and $IM = (0)$ for some $I \in \cP(k)$, then $M$ is discrete. 
For any   $k/I$-module $M$ we will denote by $M^\dis$ the object of $\cCLMu_k$ with underlying $k^\for$-module $M^\for$, equipped with the discrete topology.
% Let $M$ be an object   of $\cCLMu_k$ such that $IM =(0)$ for some ideal $I \in \cP(k)$. 
%Then, we will denote by $M^\dis$ the object of $\cCLMu_k$ with underlying $k^\for$-module $M^\for$, equipped with the discrete topology. 
\par
The category $\cCLMu_k$ admits  all limits and colimits. 
The former are calculated in the category of $k$-modules and carry the weak topology of the canonical projections, but the latter need some explanation (to be given below).  So, a limit will be denoted by $\limit$ while a colimit will carry an apex $(-)^\un$ as in  $\colimit^\un$. 
%By $\colimit$ we will denote a colimit in $\Mod_k$. 
For an inductive system $\{M_\alpha\}_{\alpha \in A}$ in $\cCLMu_k$, let 
$$(M := \colimit_\alpha M_\alpha, \{\iota_\alpha : M_\alpha \to M\}_\alpha)$$
 be the familiar colimit in $\Mod_k$. 
Then 
$\colimit^\un_\alpha M_\alpha$ is the (separated) completion of $M$ equipped with the $k$-linear topology for which a basis of open $k$-submodules consists of the submodules $N \subset M$ such that $\iota_\alpha^{-1}(N) \in \cP(M_\alpha)$ for any $\alpha$, and \emph{such that there exists $I \in \cP(k)$ with  $IM \subset N$.} If $k$ and all $M_\alpha$'s are discrete, then 
$\colimit^\un_\alpha M_\alpha$ is discrete, as well. 
\par
The cokernel of a morphism $f: N \to M$ in $\cCLMu_k$ is 
$$
\Coker^\un (f) = \limit_{P \in \cP(M)} M/(f(N) + P)
$$
where $M/(f(N) + P)$ is calculated in $\Mod_k$ and is endowed with the discrete topology. In general, $\Coker^\un (f)$  contains as a subspace $M/\ol{f(N)}$,  which is the cokernel of $f$ in $\cSLMu_k$.
If $M \in \cCLMou_k$, then $\Coker^\un (f) = M/\ol{f(N)}$; in particular, if $M \in \cCLMou_k$ and $f$ is closed, then 
$(\Coker^\un (f))^\for$ coincides with $\Coker (f^\for)$.
% If $f$  is closed and  $M \in \cCLMou_k$, then 
%$(\Coker^\un (f))^\for$ coincides with $\Coker (f^\for)$.
For any injective $\cCLMu_k$-morphism $i: N \hookrightarrow M$, we write $(M/N)^\un$ for $\Coker^\un (i)$.
For an inductive system $\{M_\alpha\}_{\alpha \in A}$ in $\cCLMu_k$, 
\beq \label{indu}
\colimit^\un_{\alpha \in A} M_\alpha = \limit_{I \in \cP(k)}  \colimit^\un_{\alpha \in A} (M_\alpha / I M_\alpha )^\un \;. 
\eeq  
%(If 
%$M_\alpha \in \cCLMou_k$ and  $I$ finitely generated, $\ol{I M_\alpha} =  I M_\alpha$, see \cite[Rmk. 8.3.3 (iv)]{GR}.)
An object $M$ of $\cCLMu_k$ whose topology is the naive canonical topology is called \emph{canonical}, or $k$-\emph{canonical} if necessary. The full subcategory of $\cCLMu_k$ consisting of $k$-canonical modules is a quasi-abelian category with enough projectives \cite[Defn. 1.3.18]{schneiders}, denoted $\cCLMcan_k$. 
See \cite{closed}  for more detail on canonical modules, and in particular for the fact that the fully faithful functor $\cCLMcan_k \subset \cCLMu_k$ admits a right adjoint $(-)^\can$, where, for $M \in \cCLMu_k$, $M^\can$ is simply $M$ endowed with the naive canonical topology.  
 \begin{rmk} \label{8.3.3}  
 The crucial consequence of our assumptions on $k$ is that any $M \in \cCLMu_k$ is complete in its naive canonical topology. This is proven in \cite{closed}. A weaker result states that this is the case if $k$ is a $I$-adic ring for a finitely generated ideal $I$. This appears as Lemma 8.3.12 (b) of \cite{GR}. In the special case when  $M \in \cCLMou_k$, this also appears 
as Proposition 0.7.2.5 of \cite{FK}. 
\end{rmk}
\begin{rmk} \label{complfop} For any $M \in \Mod_k$, let $\nwhat{M}$ be the completion of $M$ equipped with
the naive canonical topology. It follows from \cite[Rmk. 8.3.3]{GR} that the topology of the $k$-module $\nwhat{M}$ is again the 
naive canonical topology. 
%In particular, let  $M$ be the limit  in $\cCLMu_k$ of a projective system $\{M_I\}_{I \in \cP(k)}$ where $M_I$ is a discrete $k/I$-module and for $I \subset J$, $M_I \longrightarrow M_J$ is a  morphism over $k/I \longrightarrow k/J$ with kernel $JM_I$. Then $M$ is 
%equipped with the naive canonical topology and 
%$$M/JM = \bigcap_{I \subset J} \Im (M_I \longrightarrow M_J)\;,
%$$ where $I \in \cP(k)$. 
\end{rmk}
\begin{rmk} \label{limcan}
Assume  in \eqref{indu} $M_\alpha \in \cCLMcan_k$, $\forall \alpha \in A$, and let  $I \in \cP(k)$ be finitely generated. Then, by 
\cite[Rmk. 8.3.3 (iv)]{GR}, 
$\ol{I M_\alpha} =  I M_\alpha$, $\forall \alpha \in A$. Therefore in $\Mod_k$, for $M =  \colimit_{\alpha \in A} M_\alpha$  we have 
$$
 M/IM = \colimit_{\alpha \in A} (M_\alpha/IM_\alpha) \;,
 $$
 so that
\eqref{indu} becomes
\beq \label{indu1}
\colimit^\un_{\alpha \in A} M_\alpha = \limit_{I \in \cP(k)} (M/IM)^\dis = \nwhat{M}  \;.\eeq
It follows from Remark~\ref{complfop} that $\nwhat{M}$ is canonical, hence that 
$\cCLMcan_k$ is closed under $\colimit^\un$ (\ie under colimits   in $\cCLMu_k$). 
\par On the other hand, since $(-)^\can$ is a right adjoint, $\cCLMcan_k$ admits projective limits, calculated by applying $(-)^\can$ to $\limit$\,: the latter are  denoted by $\limit^\can$ and will be described below.
\end{rmk}
%An interesting result proven in \cite{closed} is the following. For any object  $M$ of $\cCLMu_k$ and any bounded closed 
%$k$-submodule $B \subset M$
%which is the closure in $M$ of a finitely generated $k$-submodule (we say that $B$ is a
%closed topologically finitely generated $k$-submodule $B \subset M$), 
%let $B^\can$ denote the $k$-module $B$ equipped with the naive canonical topology.  Then $B^\can$ is separated and complete and the natural map 
%$B^\can \to M$ induces a bijection $B^\can \to B$. We obtain a  natural morphism 
%\beq 
%M^\ind := \limINDU_{B} B^\can  \longrightarrow M \;,
%\eeq 
%for $B$ any closed finitely generated $k$-submodule of $M$. We have
%\begin{prop} \label{indcan}  For any $M$ in $\cCLMu_k$, $M^\ind$ is $k$-canonical. The canonical morphism $M^\ind \to M$ is an isomorphism if and only if $M$ is $k$-canonical. In that case
%the  underlying $k$-module map is the bijection
%$$
%\bigcup_B B^\for  \longrightarrow M^\for \;.
%$$
%\end{prop} 
\end{subsection}
\begin{subsection}{Sums and products}
For any family $M_\alpha$, $\alpha \in A$, of objects of $\cCLMu_k$, the direct sum and direct product exist and will be denoted by 
$$
\SUMU_{\alpha \in A} M_\alpha \;\;,\;\; \prod_{\alpha \in A} M_\alpha \;,
$$
respectively. 
We explicitly notice that $\SUMU_{\alpha \in A} M_\alpha$, called the \emph{uniform direct sum} of the $M_\alpha$'s,  is the completion of the algebraic direct sum 
$\bigoplus_{\alpha \in A} M_\alpha$ of the algebraic $k$-modules $M_\alpha$'s, equipped with the $k$-linear topology for which a fundamental system of open $k$-modules consists of the $k$-submodules 
$$
\bigoplus_{\alpha \in A} (U_\alpha + I M_\alpha)\;\;\mbox{such that} \;U_\alpha \in \cP(M_\alpha)\;\; \forall \alpha\;, \; \mbox{and}\;I \in \cP(k)\; \mbox{is independent of} \, \alpha \; .
$$
 Then the  $k$-module underlying $\SUMU_{\alpha \in A} M_\alpha$ differs in general  from the algebraic direct sum $\bigoplus_{\alpha \in A} M_\alpha$, even if  all $M_\alpha$ are discrete. 
 %See remark~\ref{sumsqprod2} below. 
 But if $k$ is discrete as well, then 
 $\SUMU_{\alpha \in A} M_\alpha$ coincides with $\bigoplus_{\alpha \in A} M_\alpha$ equipped with the discrete topology. 
 \par 
 For $M_\alpha$ a copy of $M \in \cCLMu_k$ for any $\alpha \in A$, we denote  by $M^{(A,\un)}$ the uniform direct sum of the $M_\alpha$'s. 
 When all $M_\alpha$'s are canonical,  ${\bigoplus}^\un_{\alpha \in A} M_\alpha$ is the completion of $\bigoplus_{\alpha \in A} M_\alpha$ in its naive canonical topology. 
 In particular, for any small set $A$, we have defined 
 $k^{(A,\un)}$. The following result is proven in \cite{closed}.
\begin{thm} \label{canmod} $\cCLMcan_k$ is the full subcategory of $\cCLMu_k$ consisting of 
 quotients of objects of the form $k^{(A,\un)}$, for small sets $A$.
%cokernels of morphisms of the form $k^{(A,\un)} \to k^{(B,\un)}$ for small sets $A,B$. 
 \end{thm}
 \begin{rmk}  \label{canmodpres}  It follows from Theorem~\ref{canmod}  that an object of $\cCLMcan_k$ is the 
 cokernel  of a morphism  of the form $k^{(A,\un)} \to k^{(B,\un)}$ for small sets $A,B$. 
 \end{rmk}
For any family $M_\alpha$, $\alpha \in A$, of objects of $\cCLMu_k$, as before,  it will  be useful to introduce the \emph{uniform box product} of the   $M_\alpha$'s~:
\beq
\label{square}
{\prod}_{\alpha \in A}^{\square,\un} M_\alpha    \;.
\eeq
Set-theoretically,  \eqref{square} coincides with $\prod_{\alpha \in A} M_\alpha$, but a basis of open submodules of  \eqref{square} consists 
of all $\prod_{\alpha \in A} U_\alpha$, with $U_\alpha \in \cP(M_\alpha)$, for which there exists $I \in \cP(k)$ such that 
$I M_\alpha \subset U_\alpha$, for any $\alpha \in A$.   
Notice that  
\beq \label{limsq}
{\prod}_{\alpha  \in A}^{\square,\un} M_\alpha = 
 \limit_{I \in \cP(k)}  
{\prod} _{\alpha  \in A}^{\square,\un}
(M_\alpha/IM_\alpha)^\un \;.
\eeq
There are natural continuous injections  
\beq \label{sumsqprod} {\bigoplus}^\un_{\alpha \in A} M_\alpha \subset {\prod}_{\alpha \in A}^{\square,u} M_\alpha \subset  {\prod}_{\alpha \in A} M_\alpha
 \;,
 \eeq
 the first of which is a closed subspace embedding while the second is a bijection. On the other hand, the inclusion 
  $$ {\bigoplus}^\un_{\alpha \in A} M_\alpha \subset  {\prod}_{\alpha \in A} M_\alpha
 $$
 is dense. 
% \begin{rmk} \label{sumsqprod2} If the $M_\alpha$'s are all discrete, the first embedding in \eqref{sumsqprod} is also a bijection.
% For example,  in $\cCLMu_{\Z_p}$, 
% $$\SUMU_{n \in \N} \Z/p^n\Z = \limit_{m \in \N} \bigoplus_{n=0}^m  \Z/p^n\Z  \oplus \bigoplus_{n=m+1}^\infty  \Z/p^m\Z 
% $$
% $$ \PRODSCU{n \in \N} \Z/p^n\Z \;.$$ 
% \end{rmk}
More generally, for a projective system $(M_\alpha)_{\alpha  \in A}$, indexed by the cofiltered poset $A$, in $\cCLMu_k$,    the $k$-submodule 
${\limit_{\alpha \in A}} M_\alpha^\for$ of  $\prod_{\alpha \in A}^{\square,u} M_\alpha$ is closed. We then define the 
 \emph{uniform box limit} of the previous projective system 
\beq \label{sqlimit}
{\limit^{\square,\un}_{\alpha \in A}} M_\alpha 
\eeq
as this closed submodule equipped with the subspace topology.   In particular,
\beq \label{sqlimit1}
({\limit^{\square, \un}_{\alpha \in A}} M_\alpha )^\for = {\limit_{\alpha \in A}} M_\alpha^\for \;.
\eeq
%\begin{prop}\label{limsquare}
%If the preordered set $A$ is cofiltered, then 
%$${\limit^{\square, \un}_{\alpha \in A}} M_\alpha  = \limit_{\alpha \in A} M_\alpha \;.$$
%\end{prop}
From \eqref{limsq} we deduce that 
\beq  \label{boxdefn}
{\limit^{\square,\un}_{\alpha \in A}} M_\alpha  = \limit_{I \in \cP(k)} {\limit^{\square,\un}_{\alpha \in A}} 
(M_\alpha/IM_\alpha)^\un \;.
\eeq
%\begin{lemma} \label{boxdefn} For any projective system $(M_\alpha)_{\alpha  \in A}$ in $\cCLMu_k$, we have 
%$$
%{\limit^{\square,\un}_{\alpha \in A}} M_\alpha  = \limit_{I \in \cP(k)} {\limit^{\square,\un}_{\alpha \in A}} 
%(M_\alpha/IM_\alpha)^\un \;.
%$$
%In particular, for any   family $(M_\alpha)_{\alpha  \in A}$ in $\cCLMu_k$,
%$$
%{\prod}_{\alpha  \in A}^{\square,\un} M_\alpha = 
% \limit_{I \in \cP(k)}  
%{\prod} _{\alpha  \in A}^{\square,\un}
%(M_\alpha/IM_\alpha)^\un \;.
%$$
%\end{lemma}
\begin{cor} \label{canlim}
In the particular case when every $M_\alpha$   is $k$-canonical, we have
\beq \label{canlim1}
\begin{split}
{\prod}_{\alpha \in A}^{\square,\un} M_\alpha &= \limit_{I \in \cP(k)} (\prod_{\alpha \in A} M_\alpha/{I M_\alpha})^\dis \\
\mbox{(resp.}\;{\limit^{\square,\un}_{\alpha \in A}} M_\alpha  &= \limit_{I \in \cP(k)} (\limit_{\alpha \in A} M_\alpha/{I M_\alpha})^\dis 
\;\;\mbox{)}\;.
\end{split}
\eeq
%(Here $I$ runs over the cofinal subfamily of \emph{finitely generated} ideals in $\cP(k)$.)
So, in this special case, $\PRODsqu_{\alpha \in A}  M_\alpha$ (risp. ${\limit^{\square,\un}_{\alpha \in A}} M_\alpha$) coincides with   $\prod_{\alpha \in A}  M_\alpha^\for$ (risp. ${\limit_{\alpha \in A}}M_\alpha^\for$) 
equipped with the naive canonical topology; it then coincides with the
direct product $\PRODcan_{\alpha \in A} M_\alpha$ (risp. the limit ${\limit^\can_{\alpha \in A}} M_\alpha$) in the category $\cCLMcan_k$. If, for any $\alpha \in A$, $M_\alpha = M \in \cCLMcan_k$, then $\prod^\can_{\alpha \in A} M_\alpha$ is denoted $M^{A,\can}$ and it is simply the product $k$-module $M^A$ equipped with the naive $k$-canonical topology. In particular, for any $I \in \cP(k)$ we have 
\beq \label{canlim11}
\begin{split}
{\prod}_{\alpha \in A}^{\square,\un} M_\alpha/ I {\prod}_{\alpha \in A}^{\square,\un} M_\alpha&=  (\prod_{\alpha \in A} M_\alpha/{I M_\alpha})^\dis \\
\mbox{(resp.}\;{\limit^{\square,\un}_{\alpha \in A}} M_\alpha/ I {\limit^{\square,\un}_{\alpha \in A}} M_\alpha &= (\limit_{\alpha \in A} M_\alpha/{I M_\alpha})^\dis \;\;\mbox{)}\;.
\end{split}
\eeq
 \end{cor} 
%Dually, for any inductive system $(M_\alpha)_{\alpha  \in A}$, indexed by the filtered poset $A$, in $\cCLMu_k$, $\colimit^\un_\alpha M_\alpha$ is the quotient of ${\bigoplus}^\un_{\alpha \in A} M_\alpha$ modulo the closure of the kernel of
%$\bigoplus_\alpha M_\alpha^\for \to \colimit_\alpha M_\alpha^\for$, where the closure is taken in $ {\bigoplus}^\un_{\alpha \in A} M_\alpha$ or, equivalently, in $\prod^{\square,u}_{\alpha \in A} M_\alpha$. If $k$ and all $M_\alpha$'s are discrete, $\colimit^\un_\alpha M_\alpha = \colimit_\alpha M_\alpha^\for$, equipped with the discrete topology. 
%Moreover, 
%\beq \label{cancolim} \begin{split}
% {\bigoplus}^\un_{\alpha \in A} M_\alpha &= \limit_{I \in \cP(k)} {\bigoplus}^\un_{\alpha \in A} (M_\alpha/{I M_\alpha})^\un \;, \\
%\mbox{(resp.}\; \colimit^\un_{\alpha \in A} M_\alpha  &= \limit_{I \in \cP(k)} \colimit^\un_{\alpha \in A} (M_\alpha/{I M_\alpha})^\un
% \;\;\mbox{)}\;.
%\end{split}
%\eeq 
%So, if all $M_\alpha$'s are canonical, $(M_\alpha/I M_\alpha)^\un = M_\alpha/I M_\alpha$ is discrete and the objects in \eqref{cancolim}
%are
% the completions of  ${\bigoplus}_{\alpha \in A} M_\alpha^\for$ and $\colimit_{\alpha \in A} M_\alpha^\for$, respectively, in their 
%naive canonical topology. 
\end{subsection}
\begin{subsection}{Tensor product}
For two objects $M$ and $N$ of $\cLMu_k$, we define  $M \otimes^\un_k N \in \cLMu_k$ as $M \otimes_k N$ equipped with the topology with basis of open $k$-submodules $\{P \otimes_k N + M \otimes_k Q\}$ for $P \in \cP(M)$ and $Q \in \cP(N)$. 
\begin{rmk} \label{corep} For $M,N, X \in \cLMu_k$, let $\Bil^\un_k(M \times N,X)$ denote the $k$-module of $k$-bilinear functions $M \times N \longrightarrow X$ which are uniformly continuous for the product uniformity of $M \times N$. Then 
$\otimes^\un_k: M \times N \longrightarrow M \otimes^\un_k N$  belongs to $\Bil^\un_k(M \times N,M \otimes^\un_k N)$ and 
the pair $(M \otimes^\un_k N, \otimes^\un_k )$ corepresents the functor 
\beq \label{corep1}
\begin{split}
\cLMu_k &\longrightarrow \Mod_k \\
X & \longmapsto \Bil^\un_k(M \times N,X) \;.
\end{split}
\eeq
\end{rmk} 

For two objects $M$ and $N$ of $\cCLMu_k$, we set 
\beq \label{tensproj0}
M \wt^\un_k N = \widehat{M \otimes^\un_k N} =\limit^\un_{P \in \cP(M),Q \in \cP(N)} ( M/P \otimes_k N/Q ) \;.
\eeq
%In particular, if $I \in \cP(k)$  
%and $N$ is canonical, so that $IN$ is an  open submodule  of $N$ and with the subspace topology is also canonical. 
%We claim that the image of $M \wt^\un_k  IN$ in $M \wt^\un_k N$ coincides with 
%$$M \wt^\un_k  IN   
%= \limit_{P \in \cP(M),J \in \cP(k)} (M/P) \otimes (IN/JIN) = \limit_{P \in \cP(M),J \in \cP(k)} 
%I (M\otimes_k N)/(M \otimes_k JN + P \otimes_k IN)
%$$
%where we assume  that $JM \subset P$, and this coincides with the closure in $M \wt^\un_k  N$ of 
%$$
%\ol{ I (M \wt^\un_k  N)} = I (M \wt^\un_k  N)
%$$
%$$M \wt^\un_k (N/IN) = M \wt^\un_k  N / M \wt^\un_k  IN = M \wt^\un_k  N / I(M \wt^\un_k  N)
%$$
The category $\cCLMu_k$, equipped with the tensor product $ \wt^\un_k$, which coincides with the one 
 of \cite[{$\bf 0$}.7.7]{EGA} (see also \cite[Chap. III, \S 2, Exer. 28]{algebracomm}),  is a symmetric monoidal category.   
For $M \in \cCLMou_k$, the functor 
$$ 
\cCLMu_k \longrightarrow \cCLMu_k  \;\;,\;\;
X  \longmapsto M \wt^\un_k X\;,
$$
restricts to a functor
\beq \label{rightex}
 \begin{split}
M \wt^\un_k -:\cCLMou_k &\longrightarrow \cCLMou_k \\
X &\longmapsto M \wt^\un_k X \;.
\end{split} \eeq
It follows from \cite[Lemma 15.1.27 (i)]{GR} that the functor \eqref{rightex} is  \emph{strongly right exact} in the sense that it preserves  cokernels of arbitrary morphisms in $\cCLMou_k$ and \emph{left exact} in the sense that it preserves  kernels of strict morphisms in $\cCLMou_k$ \cite[Defn. 1.1.12]{schneiders}. In other words, the functor \eqref{rightex} preserves cokernels and images in $\cCLMou_k$. 
\begin{rmk} \label{corepcompl} For $M,N \in \cCLMu_k$, 
%let $\Bil^\un_k(M \times N,X)$ denote the $k$-module of $k$-bilinear functions $M \times N \longrightarrow X$ which are uniformly continuous for the product uniformity of $M \times N$. Then 
$\wt^\un_k: M \times N \longrightarrow M \wt^\un_k N$  belongs to $\Bil^\un_k(M \times N,M \wt^\un_k N)$ and 
the pair $(M \wt^\un_k N, \wt^\un_k )$ corepresents the restriction to $\cCLMu_k$ of the
functor \eqref{corep1}.
%\beq \label{corep1}
%\begin{split}
%\cCLMu_k &\longrightarrow \Mod_k \\
%X & \longmapsto \Bil^\un_k(M \times N,X) \;.
%\end{split}
%\eeq
\end{rmk} 
\begin{cor}\label{tenscan}  
If $M$, $N$ are canonical $k$-modules then 
 $M \wt_k^\un N$
 is canonical. \end{cor}
\begin{proof}  Let $\phi:k^{(A,\un)} \to M$ and $\psi:k^{(B,\un)} \to N$ be strict epimorphisms. Then 
$k^{(A,\un)} \wt^\un_k k^{(B,\un)} = k^{(A \times B,\un)}$, and the morphism $\phi  \wt^\un_k \psi : k^{(A \times B,\un)} \to 
M \wt_k^\un N$ is a strict epimorphism, as well. In fact, $\phi  \wt^\un_k \psi$ decomposes into the product
$$ k^{(A,\un)} \wt^\un_k k^{(B,\un)} \map{\id_{k^{(A,\un)}} \wt^\un_k \psi} k^{(A,\un)} \wt^\un_k N 
\map{\phi  \wt^\un_k  \id_N}
M \wt_k^\un N
$$
The functor $\cCLMou_k \longrightarrow \cCLMou_k$,  $X \longmapsto k^{(A,\un)} \wt^\un_k X$ (resp. $Y \longmapsto Y \wt^\un_k  N$)  is  right exact, so that both  $\id_{k^{(A,\un)}} \wt^\un_k \psi$ and $\phi  \wt^\un_k  \id_N$ are strict epimorphisms, and therefore so is their composition $\phi  \wt^\un_k \psi$ \cite[\S 1.1.3]{schneiders}. 
\end{proof}
\begin{rmk}\label{tensnaive} For $M, N$ canonical, $M \wt_k^\un N = M \otimes^\un_k N$ is simply $M \otimes_k N$ equipped with the naive canonical topology. 
\end{rmk}
%In particular, if $I \in \cP(k)$ and $N$ is canonical (so that $IN$ is itself canonical)
%$$M \wt^\un_k  IN
%= \limit 
%$$
%$$M \wt^\un_k (N/IN) = M \wt^\un_k  N / M \wt^\un_k  IN = M \wt^\un_k  N / I(M \wt^\un_k  N)
%$$
%and, if both
%$M,N$ are canonical,
%\beq \label{tenscan} \begin{split}
%M \wt^\un_k N = \limit_{I \in \cP(k)} &M/IM \otimes_{k/I} N/IN  =  \limit_{I \in \cP(k)} M \otimes_k N/I(M \otimes_k N) = 
%\\ &\limit_{I \in \cP(k)} M \wt^\un_k N/I(M \wt^\un_k N)
% \end{split}
%\eeq
%is its own completion in the naive $k$-canonical topology. Since $M \wt^\un_k N \in \cCLMou_k$ is already complete in that topology, $M \wt^\un_k N$ is canonical, as well. 
\begin{defn}\label{proflat} An object $P \in \cCLMou_k$ is \emph{pro-flat} if, for any $I \in \cP(k)$, $(P/IP)^\un$ is flat in $\Mod_{k/I}$. 
\end{defn}
According to \cite[Lemma 15.1.27 (ii)]{GR}, if $P \in \cCLMou_k$ is pro-flat, the functor 
$$P \wt^\un_k - : \cCLMou_k \map{} \cCLMou_k
$$ 
is left exact in the above sense \ie it preserves kernels of strict morphisms.  
  \begin{rmk} \label{cantens} The functor $(-)^\can : \cCLMu_k \longrightarrow \cCLMcan_k$ does not commute in general  with $\wt^\un_k$.
\end{rmk}
\begin{lemma} \label{tensfunct} \ben
\item Let $\{M_\alpha\}_{\alpha \in A}$ (resp.  $\{N_\beta\}_{\beta \in B}$) be an inductive system in $\cCLMu_k$ and let 
 $$
 M = \colimit_{\alpha \in A}^\un M_\alpha \;\;,\;\;  N = \colimit_{\beta \in B}^\un N_\beta
 $$
 in $\cCLMu_k$. Then we have a natural morphism
 \beq  \label{tensfunct1}
 \colimit_{\alpha,\beta}^\un (M_\alpha \wt^\un_k N_\beta) \map{} M \wt^\un_k N \;.
 \eeq
 \item Let $\{M_\alpha\}_{\alpha \in A}$ (resp.  $\{N_\beta\}_{\beta \in B}$) be a  projective system in $\cCLMu_k$ and let 
 $$
 M = \limit_{\alpha \in A} M_\alpha \;\;,\;\;  N = \limit_{\beta \in B} N_\beta
 $$
 in $\cCLMu_k$. Then we have a natural morphism
 \beq  \label{tensfunct2}
 M \wt^\un_k N \map{}   \limit_{\alpha,\beta} (M_\alpha \wt^\un_k N_\beta)\;.
 \eeq
 \een
 \end{lemma}
 \begin{proof} Clear. 
 \end{proof}
 \begin{prop} \label{tensind}    $\wt^\un_k$ commutes with colimits in $\cCLMcan_k$. 
\end{prop}
\begin{proof}
 Let $\{M_\alpha\}_{\alpha \in A}$ (resp.  $\{N_\beta\}_{\beta \in B}$) be an inductive system in $\cCLMcan_k$ and let 
 $$
 M = \colimit_{\alpha \in A}^\un M_\alpha \;\;,\;\;  N = \colimit_{\beta \in B}^\un N_\beta
 $$
 in $\cCLMcan_k$. 
From Remark~\ref{limcan} we know that 
$$
 M/IM = \colimit_{\alpha \in A} (M_\alpha/IM_\alpha) \;\;,\;\; N/IN = \colimit_{\beta \in B} (N_\beta/IN_\beta) \;.
 $$
Now, for any $X \in \cCLMcan_k$, 
$$
X =  \limit_{I\in \cP(k)} X/IX \;.
$$
Then we have
$$
M \wt^\un_k N = \limit_{I\in \cP(k)} (M \wt^\un_k N)/I(M \wt^\un_k N) =  \limit_{I\in \cP(k)} (M \otimes_k N)/I(M \otimes_k N) = $$
$$\limit_{I\in \cP(k)} ( M/IM \otimes_{k/I} N/IN )= 
\limit_{I\in \cP(k)}  (\colimit_{\alpha \in A} (M_\alpha/IM_\alpha) \otimes_{k/I}  \colimit_{\beta \in B}  (N_\beta/IN_\beta)) =
$$
$$
\limit_{I\in \cP(k)}  \colimit_{\alpha \in A, \beta \in B} (M_\alpha/IM_\alpha \otimes_{k/I} N_\beta/IN_\beta) =
\limit_{I\in \cP(k)}  \colimit_{\alpha \in A, \beta \in B} (M_\alpha \otimes_k N_\beta)/I (M_\alpha \otimes_k N_\beta) =
$$
$$
\limit_{I\in \cP(k)}  \colimit_{\alpha \in A, \beta \in B} (M_\alpha \wt^\un_k N_\beta)/I(M_\alpha \wt^\un_k N_\beta) = \colimit^\un_{\alpha \in A, \beta \in B} M_\alpha \wt^\un_k N_\beta \;.
$$
%  We apply  \eqref{indu} and the fact that, if $M$ is canonical, then $IM$ is open, hence closed, in $M$, and 
%  $(M/I M)^\un$ is the quotient $k$-module $M/I M$ equipped with the discrete topology. So,  for an inductive system  $\{M_\alpha\}_{\alpha \in A}$ (resp.  $\{N_\beta\}_{\beta \in B}$) in $\cCLMcan_k$  
%we have 
% $$
%M_\alpha = \limit_{I \in \cP(k)}  M_\alpha/IM_\alpha \;\;,\;\; N_\beta =  \limit_{I \in \cP(k)} N_\beta/IN_\beta  \;,
%$$
%$$
% \colimit^\un_{\alpha \in A} M_\alpha = \limit_{I \in \cP(k)}  \colimit_{\alpha \in A} M_\alpha/IM_\alpha \;\;,\;\;
%  \colimit^\un_{\beta \in B} N_\beta = \limit_{I \in \cP(k)}  \colimit_{\beta \in B} N_\beta/IN_\beta \; ,
%$$
%and  both
%$$
% \colimit_{\alpha \in A} M_\alpha/IM_\alpha \;\;,\;\;
%\colimit_{\beta \in B}  N_\beta/IN_\beta 
%$$
%are discrete $k/I$-modules. So
%$$
%(\colimit^\un_{\alpha \in A} M_\alpha )  \wt_k^\un (\colimit^\un_{\beta \in B}  N_\beta ) =
%$$
%$$
%\limit_{I \in \cP(k)}  ((\colimit_\alpha (M_\alpha/IM_\alpha) \otimes_{k/I} (\colimit_\beta N_\beta/IN_\beta)) =
%$$
%$$ 
%\limit_{I \in \cP(k)}  \colimit_{\alpha , \beta} (M_\alpha/IM_\alpha \otimes_{k/I} N_\beta/IN_\beta) =
% \colimit^\un_{\alpha , \beta}  (M_\alpha \wt_k^\un N_\beta) \;.
%$$
\end{proof}
\end{subsection}
\end{section}
\begin{section}{Duality} \label{duality}
\begin{subsection}{Weak and strong duals}  \label{strongduality}
 We now define two internal $\hom$ functors in $\cCLMu_k$. Their difference depends  on the fact that 
 for two objects $M,N$ of $\Mod_k$, there are two natural topologies on the $k$-module $\hom_k(M,N) = \limit_F \hom_k(F,N)$, where $F$ runs over the finitely generated submodules of $M$. Namely, $\hom_k(M,N)$ may be equipped with  the discrete topology or  with  the weak  topology of the natural projections  $\hom_k(M,N) \to \hom_k(F,N)$ where, for $F$ as before, $\hom_k(F,N)$ is given  the discrete topology. We denote by $\hom_k(M,N)^\dis$ (resp. $\hom_k(M,N)^\pfg$, where ``pfg" stands for ``pro-finitely generated") the corresponding objects of $\cCLMu_k$. 
 We first observe that   there is a canonical isomorphism  in  $\Mod_k$
 \beq \label{cathom}
\hom_{\cCLMu_k}(M,N) = \limit_{Q \in \cP(N)}  \colimit_{P \in \cP(M)}\hom_k(M/P,N/Q) \;.
\eeq
%An object $M$ of $\cCLMu_k$ which carries the discrete topology is necessarily annihilated by an open ideal $I$ of $k$ and is therefore canonical both as a $k$-module and as a $k/I$-module. Conversely, for any $I \in \cP(k)$ and  any $k/I$-module $X$  we let $X^\dis$ denote the same $k$-module equipped with the discrete topology, which is then canonical. 
We  set the following 
%(\cf Proposition~\ref{limsquare} for the second equality in both \eqref{weakhom} and \eqref{stronghom})
\begin{defn}  \label{homtypes} For  any objects $M,N$  of $\cCLMu_k$, we set  
\beq \label{weakhom} 
%\begin{split}
\hom_\weak(M,N) = \limit_{Q \in \cP(N)}  \colimit^\un_{P \in \cP(M)} \hom_k(M/P,N/Q)^\pfg 
%\\
%= \limit^{\square,\un}_{Q \in \cP(N)}  \colimit^\un_{P \in \cP(M)} \hom_k(M/P,N/Q)^\pfg \in \cCLMu_k
%\end{split}
\eeq
\beq \label{stronghom} 
%\begin{split}
\hom_\strong(M,N) = \limit^{\square,\un}_{Q \in \cP(N)}  \colimit^\un_{P \in \cP(M)} \hom_k(M/P,N/Q)^\dis    
%\\=  \limit_{Q \in \cP(N)}&  \colimit^\un_{P \in \cP(M)} \hom_k(M/P,N/Q)^\dis \in \cCLMu_k \;.
%\end{split}
\eeq 
In particular, we set $M'_\strong := \hom_\strong(M,k)$ (resp. $M'_\weak := \hom_\weak(M,k)$) and call it the \emph{strong} (resp. \emph{weak}) dual of $M$. 
\end{defn}

\begin{prop} \label{unifcompl} 
For  any   $M,N \in \cCLMu_k$, there is a  natural  $k$-linear isomorphism
\beq \label{unifcomplform2} \hom_{\cCLMu_k}(M,N) \iso \hom_\strong(M,N)^\for \;.
\eeq
 %In particular, $\hom_{\cCLMu_k}(M,N)$ is complete in its naive $k$-canonical topology. 
 \end{prop}
\begin{proof}\hfill \ben
\item
 In formula \eqref{stronghom} which defines 
 $\hom_\strong(M,N)$, the topological $k$-module 
  $$C(Q):= \colimit^\un_{P \in \cP(M)}\hom_k(M/P,N/Q)^\dis$$
 is a discrete $k/I$-module for any $I \in \cP(k)$ such that $IN \subset Q$
 %So, it is an object of $\cCLMcan_k$ 
 and therefore identifies $k$-linearly with 
 the colimit 
 $$\colimit_{P \in \cP(M)}\hom_k(M/P,N/Q)\;,
 $$ 
 taken  in $\Mod_k$.  
 \item 
 %The calculation of  $\limit^\can_{Q \in \cP(N)} C(Q)$ can be performed in 2 steps. 
Then 
 $$\limit_{Q \in \cP(N)} C(Q) \in \cCLMu_k$$ 
  identifies $k$-linearly with the r.h.s. of \eqref{cathom}. That is, there is a canonical $k$-linear isomorphism 
\beq \label{unifcomplform1}\hom_{\cCLMu_k}(M,N) \iso \hom_\strong(M,N)^\for \;.
%(\limit^\can_{Q \in \cP(N)} C(Q))^\for \;.
\eeq
%Then, we apply the functor $(-)^\can$ which simply equips  both sides with with their naive $k$-canonical topology. 
%By Remark~\ref{8.3.3}
%the r.h.s. is already complete in that topology, hence the same holds for $ \hom_{\cCLMu_k}(M,N)$.   
\een
\end{proof}
%%in which they are already complete. 
%%produces the completion of $\limit_{Q \in \cP(N)} C(Q)$ equipped with the naive $k$-canonical topology. 
% \item If $N \in \cCLMcan_k$, then 
% $$\limit_{Q \in \cP(N)} C(Q)  = \limit_{I \in \cP(k)} C(IN)$$
% is an object of 
%$\cCLMou_k$. 
%By Remark~\ref{8.3.3}
%it is  complete in its naive $k$-canonical topology and therefore $\hom_\strong(M,N)^\for$  identifies with the   r.h.s. of \eqref{cathom}. We conclude that, if $M \in \cCLMu_k$ and $N \in \cCLMcan_k$, 
%\begin{proof} For any $Q \in \cP(N)$,  $N/Q$ is a $k/I$-module for some $I \in \cP(k)$. For the given $Q$, the $k$-modules $\hom_k(M/P,N/Q)$ which appear in  \eqref{stronghom} are then of $I$-torsion  and therefore  $\hom_k(M/P,N/Q)^\dis$ is $k$-canonical. Then 
%$$\colimit^\un_{P \in \cP(M)}\hom_k(M/P,N/Q)^\dis = \colimit_{P \in \cP(M)}\hom_k(M/P,N/Q)^\dis$$
%is also $k$-canonical.  
%\end{proof}
\begin{rmk} \label{weakstrong} 
\hfill\ben 
\item For any $M,N \in \cCLMu_k$,  the topology induced on $\hom_{\cCLMu_k}(M,N)$ by the isomorphism \eqref{unifcomplform2} is usually called the 
\emph{strong topology} of $\hom_{\cCLMu_k}(M,N)$. We conclude that $\hom_{\cCLMu_k}(M,N)$ is separated and complete in its strong topology.  The strong topology of $\hom_{\cCLMu_k}(M,N)$ coincides with the  topology of uniform convergence on $M$. 

\item We have a natural injective  morphism with dense image
$$\hom_\strong(M,N)  \longrightarrow \hom_\weak(M,N)
$$ 
in $\cCLMu_k$.   
%$$
% \hom_{\cCLMu_k}(M,N) \longrightarrow \hom_\strong(M,N)^\for \longrightarrow  \hom_\weak(M,N)^\for  \;.
%$$
%Notice that it may very well be that $\hom_{\cLMu_k}(M,N) =0$, while   $\hom_\weak(M,N) \neq 0$. For example, let $k$ be the $p$-adic ring $\Z_p$ and $M = \F_p = \Z_p/p\Z_p$. Then 
%$\hom_{\cLMu_{\Z_p}}(\F_p,\Z_p) =0$ while 
%$$\hom_\weak(\F_p,\Z_p) =   \limPRO_{n \in \N} \hom_k(\F_p,\Z_p/p^n\Z_p)^\pfg
%$$
The topology induced on $\hom_{\cCLMu_k}(M,N)$ by the injective map
$$
 \hom_{\cCLMu_k}(M,N) \hookrightarrow \hom_\weak(M,N)^\for 
$$
 is called the 
\emph{weak topology} of $\hom_{\cCLMu_k}(M,N)$. 
Then   $\hom_\weak(M,N)$   is the completion of 
$\hom_{\cCLMu_k}(M,N)$ for its weak topology. The  weak topology of $\hom_{\cCLMu_k}(M,N)$ coincides with the  topology of simple convergence on $M$.  
\een
\end{rmk}
%\bmau NON FUNZIONA 
%\begin{rmk}\label{homcan}
%\ben \item  If  $N$ is canonical, then $\hom_\strong(M,N)$ is canonical as well. In fact, the expression \eqref{stronghom} becomes  
%\beq \label{canhom}
%\hom_\strong(M,N) = \limit_{I \in \cP(k)}  \colimit^\un_{P \in \cP(M)} \hom_k(M/P,N/IN)^\dis   
%\eeq
%where $P \supset IM$. so that the strong topology on $\hom_{\cCLMu_k}(M,N)$ is simply the $k$-canonical topology. From $\mathit 1$ of Remark~\ref{weakstrong} $\hom_{\cCLMu_k}(M,N)$ is complete in its strong topology, hence  by Remark~\ref{complfop} $\hom_\strong(M,N) \in \cCLMcan_k$ and is simply $\hom_{\cCLMu_k}(M,N)$ equipped with the naive $k$-canonical topology. 
%\item
%In particular, for  $M \in \cCLMcan_k$, 
%$$M'_\strong  = \limit_{I \in \cP(k)} \hom_k(M/IM,k/I)^\dis \in \cCLMcan_k  
%$$ 
%%We conclude from \eqref{unifcomplform2} that for any
%% $M \in \cCLMu_k$,   $M'_\strong  \in \cCLMcan_k$ 
%   is simply $\hom_{\cCLMu_k}(M,k)$ equipped with the naive $k$-canonical topology. 
%\een
%\end{rmk}
%\emau
%\begin{prop} \hfill \ben \item
%The strong topology of $\hom_{\cCLMu_k}(M,N)$
%coincides with the  topology of uniform convergence on $M$. The completion is then 
%$\hom_\strong(M,N)$. 
%\item If $N$ is canonical  then $\hom_{\cCLMu_k}(M,N)$ is separated and complete in its strong topology, which is simply the naive $k$-canonical topology.  It coincides with 
%$\hom_\strong(M,N)$.
%\een
%\end{prop}
%\begin{proof}
%Follows from Remark~\ref{unifcompl}.
%\end{proof}
%\begin{cor} \label{dualcan}  \end{cor}
\begin{cor} \label{dualcompl}  Let $M \in \cCLMcan_k$. Then 
\beq \label{dualcompl1} M'_\strong = \limit_{I \in \cP(k)} (M/IM)'_\strong \eeq
and
\beq \label{dualcompl2}  M'_\weak = \limit_{I \in \cP(k)} (M/IM)'_\weak \eeq
 \end{cor}
\begin{proof}
%\beq \label{weakhom}
%\hom_\weak(M,N) = \limit_{Q \in \cP(N)}  \colimit^\un_{P \in \cP(M)} \hom_k(M/P,N/Q)^\pfg \in \cCLMu_k
%\eeq
The first assertion holds because 
$$M'_\strong  = \limit_{I \in \cP(k)} \hom_k(M/IM,k/I)^\dis =  \limit_{I \in \cP(k)} (M/IM)'_\strong  \;.$$  
For the second it follows from the definition \eqref{weakhom} that
$$M'_\weak  =  \limit_{I \in \cP(k)}  \hom_k(M/IM,k/I)^\pfg =  \limit_{I \in \cP(k)} (M/IM)'_\weak \;.
$$
%$$M'_\strong =  \limit_{I \in \cP(k)}   \hom_{\cCLMu_k}(M,k) /I\hom_{\cCLMu_k}(M,k)  =
% \limit_{I \in \cP(k)}  \hom_k(M/IM,k/I)^\dis \;.$$  
 \end{proof}
%\bmau 
%\begin{rmk} \label{dualcompl} If $M \in \cCLMcan_k$, \eqref{stronghom} gives
%%We point out that, for any $M \in \cCLMu_k$,
%%$$M'_\strong = \limit_{I \in \cP(k)} ((M/IM)^\un)'_\strong   
%%\;,
%%$$
%%Assume $M$ is canonical. Then
%$$
%M'_\strong =  \limit_{I \in \cP(k)}  \colimit^\un_{I \supset J \in \cP(k)} \hom_k(M/JM,k/I)^\dis  = \limit_{I \in \cP(k)}  \hom_k(M/IM,k/I)^\dis  \;.
%$$
%%where the dual of $(M/IM)^\un$ is taken, indifferently, in $\cCLMu_k$ or in $\cCLMu_{k/I}$.  
%%If, in particular, $M$ is canonical, 
%So, 
%$$M'_\strong = \limit_{I \in \cP(k)} (M/IM)'_\strong  
%\;,
%$$
%where $(M/IM)'_\strong$ is simply $\hom_k(M/IM,k/I)$ equipped with the discrete topology. 
%So, $M'_\strong$ is the completion of $\hom_k(M,k)$ equipped with the naive canonical topology.  We conclude that  
%$M'_\strong \in \cCLMcan_k$.
%\end{rmk}
%\emau
 \begin{defn} \label{transpose}  
For a morphism $\varphi:M \longrightarrow N$ in $\cCLMu_k$, we get   \emph{transposed morphisms} 
$\varphi^\tu: N'_\strong \longrightarrow M'_\strong$ and $\varphi^\tu: N'_\weak \longrightarrow M'_\weak$. 
\end{defn}
   \end{subsection}

\begin{subsection}{Duals of limits and colimits}
\begin{lemma} \label{limduals1} Let  $(M_\alpha, j_{\alpha,\beta})_{\alpha \leq \beta}$ (resp. $(M_\alpha, \pi_{\alpha,\beta})_{\alpha \leq \beta}$) be an inductive 
(resp.  projective) 
system in $\cCLMu_k$, indexed by the filtered  
poset $(A,\leq)$, and let $N \in \cCLMu_k$. 
Then we have in $\cCLMu_k$ 
  \beq  \label{weakhomlim}
\hom_\weak(\colimit^\un_\alpha M_\alpha,N) = \limit_\alpha \hom_\weak( M_\alpha,N)\;.
\eeq
\beq  \label{boxhomlim}
\hom_\strong(\colimit^\un_\alpha M_\alpha,N) =  \limit^{\square,\un}_\alpha \hom_\strong( M_\alpha,N)   \;;
\eeq
\beq  \label{weakhomprolim}
\hom_? (\limit_\alpha M_\alpha,N) = \colimit^\un_\alpha \hom_?( M_\alpha,N)    
\eeq
where $? = \weak,\strong$.
%\beq  \label{stronghomlim}
%\hom_\strong(\limit_\alpha M_\alpha,N) = \colimit^\un_\alpha \hom_\strong( M_\alpha,N)    \;.
%\eeq
If $M_\alpha \in \cCLMcan_k$ for all $\alpha$ and $N \in \cCLMcan_k$, as well, we also have 
\beq  \label{weakhomprolimsq}
\hom_? (\limit^{\square,\un}_\alpha M_\alpha,N) = \colimit^\un_\alpha \hom_?( M_\alpha,N)   \;,
\eeq
for $? = \weak,\strong$.
%\beq  \label{stronghomlimsq}
%\hom_\strong(\limit^{\square,\un}_\alpha M_\alpha,N) = \colimit^\un_\alpha \hom_\strong( M_\alpha,N)    \;.
%\eeq
\end{lemma}
\begin{proof} Formula  \eqref{weakhomlim} is  obtained by completion of the 
tautological $k$-linear isomorphism
\beq  \label{weakhomlim1}
\hom_{\cCLMu_k}(\colimit^\un_\alpha M_\alpha,N) = \limit_\alpha \, \hom_{\cCLMu_k} ( M_\alpha,N) 
\eeq
for the weak topology of the \emph{l.h.s.} and the limit topology of the weak topologies of each term of the \emph{r.h.s.}.  Let 
$j_\alpha : M_\alpha \to \colimit^\un_\alpha M_\alpha$ denote the canonical morphism, for any $\alpha \in A$. 
The topology of simple convergence on 
$$\hom_{\cCLMu_k}(\colimit^\un_\alpha M_\alpha,N)
$$ has a fundamental system of open submodules 
$$U(\alpha_1,\dots,\alpha_n; P_{\alpha_1},\dots,P_{\alpha_n};Q) \;,$$ 
consisting of the $k$-linear functions $f: \colimit_\alpha M_\alpha \to N$ such that $f(j_{\alpha_i}(P_{\alpha_i})) \subset Q$, for $i=1,2,\dots,n$,  where $P_{\alpha_i}$ is a finitely generated  $k$-submodule of $M_{\alpha_i}$ and $Q$ an open $k$-submodule of $N$. 
It is clear that this topology coincides with the weak topology of the projections 
$\limit_\alpha  \hom_{\cCLMu_k}(M_\alpha,N) \to \hom_\weak(M_\alpha,N)$, for $\alpha \in A$. 
\par \medskip 
%\bmau We now prove \eqref{boxhomlim}. We recall \eqref{stronghom}
%\beq \label{stronghom2}
%\hom_\strong(M,N) = \limit_{Q \in \cP(N)}  \colimit^\un_{P \in \cP(M)} \hom_k(M/P,N/Q)^\dis \in \cCLMu_k \;.
%\eeq  \emau
 To prove \eqref{boxhomlim} we refer to the equality \eqref{weakhomlim1} and use the definition \eqref{stronghom}.  We pick $Q \in \cP(N)$. 
Then
\beqa \begin{split}
\hom_\strong(\colimit^\un_\alpha M_\alpha,N/Q) &= \limit^{\square,\un}_{P \in \cP(\colimit^\un_\alpha M_\alpha)}\, \hom_k((\colimit^\un_\alpha M_\alpha)/P, N/Q)^\dis =  \\
 \limit^{\square,\un}_{\beta \in B, P_\beta \in \cP(M_\beta)} & \, \hom_k((\colimit_\alpha M_\alpha)/\pi_\beta^{-1}(P_\beta), N/Q)^\dis= \\
 \limit^{\square,\un}_{\alpha \in A, P_\alpha \in \cP(M_\alpha)} &\, \hom_k(M_\alpha/P_\alpha, N/Q)^\dis = \limit^{\square,\un}_\alpha \, \hom_\strong( M_\alpha,N/Q) \;.
\end{split}
\eeqa
Now, for \eqref{weakhomprolim}  and $? = \weak,\strong$,  we pick $Q \in \cP(N)$.   Let $\pi_\alpha :  \limit_\alpha M_\alpha \to  M_\alpha$ denote the canonical morphism, for any $\alpha \in A$.  We set $! = \pfg$ if $? = \weak$ and $! = \dis$ if $? = \strong$,  and compute  
\beq \label{indsystem} \begin{split}
\hom_?(\limit_\alpha M_\alpha,N/Q) &= \colimit^\un_{P \in \cP(\limit_\alpha M_\alpha)} \hom_k((\limit_\alpha M_\alpha)/P, N/Q)^! =  \\
 \colimit^\un_{\beta \in B,P_\beta \in \cP(M_\beta)} & \hom_k((\limit_\alpha M_\alpha)/\pi_\beta^{-1}(P_\beta), N/Q)^!= \\
  \colimit^\un_{\alpha  \in A, P_\alpha \in \cP(M_\alpha)} & \hom_k(M_\alpha/P_\alpha, N/Q)^! \;,
  \end{split}
\eeq 
where  in the last term the cofiltered projective system $\ol{\pi}_{\alpha,\beta}:M_\beta/P_\beta \rightarrow M_\alpha/P_\alpha$ is induced by 
$\pi_{\alpha,\beta}: M_\beta \to M_\alpha$ and $P_\beta \subset \pi_{\alpha,\beta}^{-1}(P_\alpha)$  if $\beta \geq \alpha$, but $P_\beta = \pi_{\alpha,\beta}^{-1}(P_\alpha)$, for $\beta >> \alpha$. For two such systems  $P= (\ol{\pi}_{\alpha,\beta})$, $\ol{\pi}_{\alpha,\beta}:M_\beta/P_\beta \rightarrow M_\alpha/P_\alpha$ and  
$P' = (\ol{\pi}'_{\alpha,\beta})$, $\ol{\pi}'_{\alpha,\beta}:M_\beta/P'_\beta \rightarrow M_\alpha/P'_\alpha$ such that $P'_\alpha \subset P_\alpha$, for any $\alpha \in A$, we have a morphism 
  $\varphi_{P,P'} = (\varphi_\alpha)_\alpha : (\ol{\pi}'_{\alpha,\beta}) \to (\ol{\pi}_{\alpha,\beta})$, \ie a family of commutative diagrams
\beq \label{morfproj}
 \;\;\;  \begin{tikzcd}[column sep=2.5em, row sep=2.5em] 
M_\beta/P'_\beta  \arrow{r}{\varphi_\beta} \arrow{d}{\ol{\pi}'_{\alpha,\beta}}  &  M_\beta/P_\beta   \arrow{d}{\ol{\pi}_{\alpha,\beta}}
\\ 
M_\alpha/P'_\alpha    \arrow{r}{\varphi_\alpha} &M_\alpha/P_\alpha \;\;.
\end{tikzcd}
\eeq
For a further system $P'' = (\ol{\pi}''_{\alpha,\beta})$ and a morphism  $\varphi_{P',P''} : P'' \to P'$, we have $\varphi_{P,P'} \circ \varphi_{P',P''} = \varphi_{P,P''}$. The inductive system in \eqref{indsystem}  then splits into 
$$
\colimit^\un_{\alpha  \in A} \colimit^\un_{(P,\varphi_{P,P'})} \hom_k(M_\alpha/P_\alpha, N/Q)^! =   \colimit^\un_\alpha \hom_?( M_\alpha,N/Q) \;.
$$
%We similarly prove \eqref{stronghomlim}. 
\par \smallskip
The proof  of \eqref{weakhomprolimsq}    follows the scheme of \eqref{indsystem}.  Using \eqref{canlim11} we obtain
%We first observe that  
%for any $I \in \cP(k)$
%$$
%(\limit^{\square,\un}_\alpha M_\alpha)/I (\limit^{\square,\un}_\alpha M_\alpha) = (\limit_\alpha (M_\alpha/I M_\alpha))^\dis \;.
%$$
\beq \label{indsystemsq} \begin{split}
\hom_?(\limit^{\square,\un}_\alpha &M_\alpha,N)  = \limit_{I \in \cP(k)} \hom_k(\limit_\alpha (M_\alpha/I M_\alpha), N/IN)^! =  \\
% \colimit^\un_{\beta \in B,P_\beta \in \cP(M_\beta)} & \hom_k((\limit_\alpha M_\alpha)/\pi_\beta^{-1}(P_\beta), N/Q)^\pfg= \\
  \limit_{I \in \cP(k)}  \colimit_{\alpha  \in A}  & \hom_k(M_\alpha/IM_\alpha,   N/IN)^! =\colimit^\un_{\alpha  \in A}\hom_?(M_\alpha,N)  
  \;.
  \end{split}
\eeq  
%The proof of  \eqref{stronghomlimsq}  is the same up to replacement of the index $(-)_\weak$ by $(-)_\strong$ and of the apex mark $(-)^\pfg$ by $(-)^\dis$. 
%\beqa \begin{split}
%\hom_\strong(\limit_\alpha M_\alpha,N/Q) &= \colimit^\un_{P \in \cP(\limit_\alpha M_\alpha)} \hom_k((\limit_\alpha M_\alpha)/P, N/Q)^\dis =  \\
% \colimit^\un_{\beta \in B, P_\beta \in \cP(M_\beta)} & \hom_k((\limit_\alpha M_\alpha)/\pi_\beta^{-1}(P_\beta), N/Q)^\dis= \\
% \colimit^\un_{\alpha \in A, P_\alpha \in \cP(M_\alpha)} & \hom_k(M_\alpha/P_\alpha, N/Q)^\dis = \colimit^\un_\alpha \hom_\strong( M_\alpha,N/Q) \;.
%\end{split}
%\eeqa
\end{proof}

%\begin{lemma} \label{limduals2} Let  $(M_\alpha)_{\alpha  \in A}$ be an inductive 
%%(resp.  projective) 
%system in $\cCLMcan_k$, indexed by the filtered 
%%(resp. cofiltered) 
%set $A$. 
%Let $N$ be an object of $\cCLMcan_k$. 
%Then we have in $\cCLMcan_k$
%\beq  \label{boxhomlim}
%\hom_\strong(\colimit^\un_\alpha M_\alpha,N) =  \limit_\alpha^\can \hom_\strong( M_\alpha,N)  \;.
%\eeq
%\bmau NON SONO SICURO (resp. 
%\beq  \label{stronghomlim}
%\hom_\strong(\limit_\alpha M_\alpha,N) = \colimit^\un_\alpha \hom_\strong( M_\alpha,N) \;\mbox{)}   \;.
%\eeq \emau
%\beq  \label{stronghomlim}
%\hom_\strong(\limit_\alpha M_\alpha,N) = \colimit^\un_\alpha \hom_\strong( M_\alpha,N)    \;.
%\eeq
%\end{lemma}
%\begin{lemma} \label{limduals} Let  $(M_\alpha)_{\alpha  \in A}$ be an inductive (resp.  projective) system in $\cCLMu_k$, indexed by the preordered (resp. cofiltered) set $A$. 
%Let $N$ be a discrete (hence of open torsion) object of $\cCLMu_k$. 
%Then we have in $\cCLMu_k$
%\beq  \label{boxhomlim}
%\hom_\strong(\colimit^\un_\alpha M_\alpha,N) =  \limit^{\square,\un}_\alpha \hom_\strong( M_\alpha,N) \; ,
%\eeq
%\beq  \label{stronghomlim}   
%\hom_\strong(\limit_\alpha M_\alpha,N) = \colimit^\un_\alpha \hom_\strong( M_\alpha,N)    \;.
%\eeq
%\end{lemma}
   \begin{cor} \label{biduality}  For any family $(M_\alpha)_{\alpha \in A}$ of  $\cCLMu_k$
  we have 
    $$ ({\bigoplus}^\un_{\alpha \in A} M_\alpha )'_\weak= {\prod}_{\alpha \in A}  (M_\alpha)'_\weak \;\;,\;\;  
   ({\bigoplus}^\un_{\alpha \in A} M_\alpha )'_\strong = {\prod}^{\square, \rm u} _{\alpha \in A} (M_\alpha)'_\strong  \; ,$$ 
    $$
  ({\prod}_{\alpha \in A}  M_\alpha)_?'   = {\bigoplus}^\un_{\alpha \in A} (M_\alpha )_?'   \;\;,\;\;\mbox{for $? = \weak,\strong$}
 % \;\;,\;\; ({\prod}_{\alpha \in A}  M_\alpha)_\strong'   = {\bigoplus}^\un_{\alpha \in A} (M_\alpha )_\strong'  
   \; .
   $$
   If, for any  $\alpha \in A$,  $M_\alpha \in \cCLMcan_k$,  we also have 
   $$
  \l( {\prod}^{\square, \rm u} _{\alpha \in A}  M_\alpha \r)'_? =  {\bigoplus}^\un_{\alpha \in A} (M_\alpha )_?' 
 % \;\;,\;\;   \l( {\prod}^{\square, \rm u} _{\alpha \in A}  M_\alpha \r)'_\strong =  {\bigoplus}^\un_{\alpha \in A} (M_\alpha )_\strong' 
  \;\;,\;\;\mbox{for $? = \weak,\strong$} \;.
   $$
   \end{cor}
 \begin{rmk} \label{closedcan} In \cite{closed} we define an internal $\hom$, called $\hom_\strong^\can$ in $\cCLMcan_k$, 
 by simply setting $\hom_\strong^\can(M,N) = \hom_\strong(M,N)^\can$, for any $M,N \in \cCLMcan_k$.  We  prove that $(\cCLMcan_k, \wt^\un_k,\hom^\can_\strong,k)$
is a quasi-abelian closed symmetric monoidal category with all limits and colimits.\footnote{Moreover, $\cCLMcan_k$ has enough projectives in the sense of \cite[\S 1.3.4]{schneiders} and for every projective $P$, the functor 
$$
\cCLMcan_k \longrightarrow \cCLMcan_k \;\;,\;\; X \longmapsto P  \wt_k^\un X
$$
is left-exact (and therefore exact) in the sense of  \cite[Defn.  1.1.12]{schneiders}.} 
So, for $M \in  \cCLMcan_k$ the dual in $\cCLMcan_k$,   is $(M'_\strong)^\can$ hence differs in general from $M'_\strong$. 
 \end{rmk}
 \begin{prop}\label{canmoddual}
 Let $M \in \cCLMcan_k$ and let 
$$ M = \Coker^\un (k^{(A,\un)} \map{\varphi}  k^{(B,\un)})
 $$
 be a presentation of $M$ as in Remark~\ref{canmodpres}. Then 
\beq \label{canmoddual1} M'_\strong = \Ker (k^{B,\can} \map{\varphi^\tu} k^{A,\can}) \;.
\eeq
 \end{prop}
 \begin{proof}
 By Corollary~\ref{dualcompl}, we may assume that $k$ and $M$ are discrete so that 
 $$ M = \Coker(k^{(A)} \map{\varphi}  k^{(B)})
 $$
 in $\Mod_k$. The functor $(-)^\tu = \hom_k(-,k)$ is then left exact and 
 $$
\hom_k(M,k) = \Ker ( k^{B} \map{\varphi^\tu} k^{A} )\;.
 $$
 Going back to the situation of the statement, 
  $$ M'_\strong = \limit_{I \in \cP(k)}  \Ker  (k^{B}/I k^{B}  \map{\varphi^\tu \mod I} k^{A}/Ik^A) 
  $$
  which coincides with the r.h.s. of \eqref{canmoddual1}.
 \end{proof} 
  \begin{defn}   \label{perfpair} Let $M$, $N$ be objects of $\cCLMu_k$ and let 
  \beq  
\begin{split}
\circ  : M  \times N  & \longrightarrow k  \\ 
(m\; , \; n) & \longmapsto m \circ n
  \end{split}
  \eeq
  be  a $k$-bilinear map  which is 
  continuous in the two variables separately. For $? =\weak,\strong$, we say that $\circ$  
 is a   \emph{$?$-left perfect pairing}  
 %\emph{weak} (resp. \emph{strong}) \emph{left-perfect pairing} 
  if  the map 
\beqa \begin{split}
M &\longrightarrow \hom_{\cCLMu_k} (N,k) \\ m &\longmapsto (y \longmapsto m \circ y)
\end{split}
\eeqa
identifies $M$ with the $?$-dual of $N$. We say that $\circ$  
 is a   \emph{$?$-right perfect pairing}  
if
the map 
\beqa \begin{split}
N &\longrightarrow \hom_{\cCLMu_k} (M,k) \\ n &\longmapsto (x \longmapsto x \circ n)
\end{split}
\eeqa
identifies $N$ with the $?$-dual of $M$. If $\circ$ is both  $?$-left perfect and $?$-right perfect, we say that $\circ$ is a \emph{$?$ perfect pairing}.
    \end{defn}  
\end{subsection}  

\begin{subsection}{Relevant subcategories of $\cCLMu_k$} 
\par
We recall that $\cCLMcan_k \subset \cCLMou_k \subset \cCLMu_k$ is a sequence of full embeddings of categories, where the two smaller ones are quasi-abelian. 
Let  us denote by $\cF_k$ the full subcategory of   $\cCLMcan_k$  consisting of 
finitely and freely 
generated $k$-modules equipped with their naive 
$k$-canonical topology. 
For any full subcategory $\cA$ of $\cCLMu_k$,
we consider the following  full subcategories of $\cCLMu_k$~:
\ben
\item  The full subcategory $\colimit^\un \cA$ of $\cCLMu_k$ consisting of 
objects of the form $\colimit^\un_{\alpha \in A} M_\alpha$, for a filtered inductive system $\{M_\alpha\}_{\alpha \in A}$ of objects of $\cA$. Notice that  $\colimit^\un \cF_k \subset \cCLMcan_k$ and that 
objects of  $\cCLMu_k$ of the form
$$
{\bigoplus}^\un_{\alpha \in A} k e_\alpha
$$
where, for any index set $A$, $ k e_\alpha$ denotes a copy of the object $k$ of $\cCLMu_k$, are  in 
$\colimit^\un \cF_k$. 
 \item  The full subcategory $\limit \, \cA$ of $\cCLMu_k$ consisting of 
objects of the form $\limit_{\alpha \in A} M_\alpha$, for a cofiltered  projective system $\{M_\alpha\}_{\alpha \in A}$ of objects of $\cA$. 
Objects of  $\cCLMu_k$ of the form
$$
\prod_{\alpha \in A} k e_\alpha
$$
where, for any index set $A$, $ k e_\alpha$ denotes a copy of the object $k$ of $\cCLMu_k$, are  in 
$\limit \, \cF_k$.   
\item The full subcategory $\limit^\can \, \cF_k$ of $\cCLMcan_k$ consisting of 
objects of the form $\limit^\can_{\alpha \in A} F_\alpha$, for a cofiltered  projective system $\{F_\alpha\}_{\alpha \in A}$ of objects of $\cF_k$. 
\item The full subcategory $\limit \, \cCLMcan_k$ of $\cCLMu_k$ consisting of 
objects of the form $\limit_{\alpha \in A} M_\alpha$, for a cofiltered  projective system $\{M_\alpha\}_{\alpha \in A}$ of objects of $ \cCLMcan_k$. 
\een 
 \end{subsection} 
    \begin{subsection}{Duals of tensor products} 
  \begin{prop} \label{tensdual}
  Let $M,N$ be objects of 
  $\cCLMcan_k$.  
  Then,  for $? = \weak, \strong$, we have a natural morphism
  \beq \label{tensdual?}
 M'_? \wt^\un_k N'_?  \map{}  (M \wt^\un_k N)'_?  
   \eeq  
    in $\cCLMu_k$. 
%  \bmau  \item  NON VA BENE   Let $M,N$ be objects of $\sC$.    Then  $M \wt^\un_k N$ is also an object of $\sC$ and
%     $$
%   (M \wt^\un_k N)'_\strong  
%   = M'_\strong \wt^\un_k N'_\strong $$
%   in $\sD_\acs$. \emau
%   \item   Let $M,N$ be objects of $\sD_\acs$.    
%    Then  $M \wt^\un_k N$ is also an object of $\sD_\acs$ and
%     $$
%   (M \wt_k^\un N)'_\weak 
%   = M'_\weak \wt^\un_k N'_\weak $$ 
%   in $\sC$. 
     %\een
   \end{prop} 
   \begin{proof}    
%   For any $M \in \Mod_k$ let $M^\naive \in \cLMu_k$ be $M$ equipped with its naive $k$-canonical topology.   By Remarks~\ref{homcan}  and \ref{corepcompl}
%   $$(M \wt^\un_k N)'_\strong  =  \hom_{\cCLMu_k}(M \wt^\un_k N,k)^\naive = (\hom_{\cCLMu_k}(M,k) \otimes_k 
%   \hom_{\cCLMu_k}(N,k))^\naive =$$ 
%   $$
%= \hom_{\cCLMu_k}(M,k)^\naive \otimes^\un_k 
%   \hom_{\cCLMu_k}(N,k)^\naive =  \hom_{\cCLMu_k}(M,k)^\naive \wt^\un_k 
%   \hom_{\cCLMu_k}(N,k)^\naive  $$
%   $$ = M'_\strong \wt^\un_k N'_\strong \;.
%$$
For any ring $R$ and any $M,N \in \Mod_R$, we have a natural $R$-linear map
\beq \label{tens}
\hom_R(M,R) \otimes_R \hom_R(N,R) \map{} \hom_R(M\otimes_RN,R) \;.
\eeq
 From Corollary~\ref{dualcompl}~:
 %and Proposition~\ref{tensproj}~:
%    $$(M \wt^\un_k N)'_? = \limit_{I \in \cP(k)} (M \wt^\un_k N/I(M \wt^\un_k N))'_? = \limit_{I \in \cP(k)} (M/IM \otimes_{k/I} N/IN)'_? = $$
%    $$
%    \limit_{I \in \cP(k)} ((M/IM)'_? \otimes_{k/I} (N/IN)'_? ) =  \limit_{I \in \cP(k)} (M/IM)'_? \wt^\un_k  \limit_{I \in \cP(k)} (N/IN)'_?  = M'_? \wt^\un_k N'_? \;.$$
       $$ M'_? \wt^\un_k N'_?=  \limit_{I \in \cP(k)} (M/IM)'_? \otimes_{k/I}  \limit_{I \in \cP(k)} (N/IN)'_? \map{\eqref{tensfunct2}}
    \limit_{I \in \cP(k)} ((M/IM)'_? \otimes_{k/I} (N/IN)'_? )  $$
       $$\map{\ref{tens}}  \limit_{I \in \cP(k)} (M/IM \otimes_{k/I} N/IN)'_?   = \limit_{I \in \cP(k)} (M \wt^\un_k N/I(M \wt^\un_k N))'_? = (M \wt^\un_k N)'_? \;,  $$
       where the first arrow follows from functoriality while the second is deduced from \eqref{tens}.
 \end{proof}
 \begin{lemma} \label{ringtens}
 Let $A,B$ be ring objects of $\cCLMu_k$. Then there are natural morphisms $A \map{} A  \wt^\un_k B$ (resp. $B \map{} A  \wt^\un_k B$),  $a \map{} a \otimes 1_B$ (resp.  $b \map{} 1_A \otimes b$). Then,  for $? = \weak, \strong$, we have a natural morphism
  \beq \label{tensdual1?}
   (A \wt^\un_k B)'_?   \map{}   A'_? \wt^\un_k B'_?
   \eeq  
    in $\cCLMu_k$. 
    \end{lemma}
    \begin{cor}  \label{ringtens1} Let $A,B$ be (commutative unital) $k$-algebra objects of $\cCLMcan_k$. Then the morphism \eqref{tensdual1?} is an isomorphism of (co-commutative co-unital) $k$-coalgebra objects of $\cCLMu_k$. 
    \end{cor}
\end{subsection}
  \end{section}
 \begin{section}{The category  of  $td$-spaces} \label{STS}
\begin{subsection}{Topological generalities}  
 Let $\Top$ be the category of topological spaces  and continuous maps. For a topological space $X$ we denote by $\tau = \tau_X$ its topology, \ie  the family of open subsets of $X$. The category $\Top$ is bicomplete. For a small family $\{X_\alpha\}_{\alpha \in A}$, the coproduct $\coprod_{\alpha \in A} X_\alpha$ is the disjoint union of open subspaces $X'_\alpha$, for $\alpha \in A$, and is equipped with continuous maps $\iota_\alpha: X_\alpha \longrightarrow  \coprod_{\alpha \in A} X_\alpha$ which induce homeomorphisms $X_\alpha \iso X'_\alpha$, for any $\alpha$. We informally call \emph{sum} the coproduct in $\Top$. 
 \par
We recall that a continuous map $f:X \to Y$ of topological  spaces, with $Y$ locally compact Hausdorff,  is said to be  \ben
\item
\emph{proper} if, equivalently, 
\ben \item  the preimage of every compact subset of $Y$ is a compact in 
$X$,
\item   
$f$  is closed and has compact fibers.
\een
\item   \emph{locally finite} (resp. \emph{finite})  if it is closed and has finite fibers 
(resp. and there exists an integer $n$ such that the cardinality of any fiber of $f$ is bounded by $n$; we then say that  \emph{the degree of $f$ is $\leq n$}).  
\een
\begin{notation} \label{disconnectdef}
For any topological space $X$, and for $x \in X$, the \emph{component} of $x$ is the union of all connected subsets of $X$ 
containing $x$. The \emph{quasi-component} of $x$ is the intersection of clopen (= closed and open) subsets of $X$ containing $x$. 
So the quasi-component of $x$ in $X$ contains the component of $x$. In a compact Hausdorff  topological space $X$ and for any $x \in X$, the quasi-component and the component  of $x$ in $X$ coincide. A topological space $X$ is \emph{hereditarily disconnected} (resp.  \emph{totally disconnected}) if  for any $x \in X$ the component (resp. the quasi-component) of $x$ is $\{x\}$.  A non-empty Hausdorff space $X$ is \emph{zero dimensional} if  its topology admits a basis
consisting of clopen subsets. A zero dimensional space is totally disconnected hence hereditarily disconnected. A Hausdorff topological space is \emph{paracompact} (resp. a \emph{Lindel\"of space}) if (resp. it is regular and) every open cover has a locally finite (resp. countable) open refinement. In a Lindel\"of space every open cover admits a countable  subcover.  
\begin{thm} \label{stonechar}  (\cite[Thm. 6.2.10]{eng}) For a non empty locally compact paracompact  space $X$ the three notions are equivalent~:
\ben
\item $X$ is hereditarily disconnected;
\item $X$ is totally disconnected;
\item $X$ is zero-dimensional.
\een
\end{thm}

Any (resp. non empty) subspace of a hereditarily/totally disconnected (resp. $0$-dimensional) topological space is hereditarily/totally disconnected (resp. $0$-dimensional) \cite[6.2.11]{eng}. \par
%Moreover, $\cK(X)$ will denote  the set of open compact subsets of $X$. 
\end{notation}
\begin{lemma} \label{addprop}
Let $X = \coprod_{\alpha \in A} X_\alpha$ be the sum of a non empty family of non empty topological spaces. Then 
\ben
\item 
$X$ is hereditarily/totally disconnected (resp. $0$-dimensional) if and only if every $X_\alpha$ is  \cite[6.2.13]{eng}.
\item $X$ is paracompact (resp. locally compact) if and only if every $X_\alpha$ is~: see  \cite[5.1.30]{eng} (resp. this is clear).
\een
\end{lemma}
We are indebted to Umberto Marconi for pointing out to us the following theorem in Engelking \cite[5.1.27]{eng}
\begin{thm} \label{Lind} Every locally compact paracompact space $X$ is the sum of a family of Lindel\"of spaces, and conversely. 
\end{thm}
\end{subsection}
\begin{subsection}{Stone spaces and $td$-spaces}
\begin{defn} \label{Q(X)}
For a topological space $X$, we let $\sQ(X)$ denote the set of its discrete proper quotients, \ie of continuous surjective maps with compact fibers $\pi_D:X \longrightarrow D$ where $D$ is a discrete topological space. 
\end{defn}
\begin{defn} \label{Stonedef}  A compact Hausdorff space $X$ satisfying the three equivalent conditions of Proposition~\ref{stonechar} is called a \emph{Stone space}.   
\end{defn}
The following statement is well-known. 
\begin{lemma} \label{stonespace} The following properties of a  topological space $X$ are equivalent:
\ben
\item $X$ is a Stone space; 
\item Any open cover of $X$ admits a  refinement which is a finite partition of $X$ in a family of open compact subspaces;
\item $X$ is the limit of the projective system of its finite discrete quotients;
\item $X$ is a small cofiltered limit of finite discrete spaces and surjective maps, \ie $X$ is a profinite set.
\een
\end{lemma}
\begin{prop} \label{Stonemetr} Let $X$ be a Stone space. The following are equivalent~: 
\ben
\item $X$ is  second countable;
\item $X$ is metrizable;
\item the family of clopen subsets of $X$ is countable;
\item $X$ is an inverse limit of a sequence of finite (discrete) spaces.
\een
\end{prop}
\begin{defn} \label{STSdef} 
A   \emph{$td$-space} is a locally compact paracompact $0$-dimensional topological  space.
We let  $td-\cS$ denote the full subcategory of $\Top$ consisting of $td$-spaces. 
\end{defn}  
Clearly, any Stone space is a $td$-space.  
\begin{lemma} \label{LindStone} A locally compact $0$-dimensional Lindel\"of space is a countable sum of  Stone subspaces, and conversely. 
\end{lemma}
\begin{proof}
In a locally compact $0$-dimensional space  a basis for the topology is given by all open compact subsets. So, a locally compact $0$-dimensional space $X$ admits an open cover $\sU$ by Stone subspaces. If moreover $X$ is 
Lindel\"of  we may assume that the cover $\sU$ is countable, $\sU = \{U_j\}_{j=0,1,\dots}$. Then the sequence of open Stone subspaces $V_j = U_j - \bigcup_{i <j} U_i$, $j=0,1,\dots$  gives  a partition of $X$,  as required. 
\end{proof}
\begin{cor} \label{LindStone2} \hfill
\ben
\item
Let $X = \coprod_{\alpha \in A} X_\alpha$ be the sum of a non empty family of non empty topological spaces. Then 
$X$ is a $td$-space  if and only if every $X_\alpha$ is one.
\item
A   $td$-space is a sum of Stone spaces and conversely. 
\een
\end{cor}
\begin{proof} $\mathit 1$ follows from the definition and Lemma~\ref{addprop}. \par
$\mathit 2$ follows from Theorem~\ref{Lind} and Lemma~\ref{LindStone}. 
\end{proof}
\end{subsection}
\begin{subsection}{Limit and colimit description}
\begin{thm} \label{STSspaces}  The following properties of a topological space $X$ are equivalent:
\ben
\item $X$ is the limit 
  \beq \label{prodis}
X= \limit_{D \in \sQ(X)}  \,D \;.
\eeq
\item $X$ is  the limit  
    \beq \label{prodis2}
X= \limit_{D \in \sJ}  \, D 
\eeq
of a small cofiltered projective system $\sJ$ of discrete spaces and locally finite surjective maps. 
(In this case, $\sJ$ turns out to be a cofinal projective sub-system of $\sQ(X)$ see Remark~\ref{limdiscr} below). 
\item $X$ is the sum 
\beq \label{sumstone}
X = \coprod_{\alpha \in A} X_\alpha 
\eeq
of a small family of Stone spaces.
\item $X$ is  paracompact and is the colimit of the inductive system of its open Stone subspaces.
\item $X$ is  paracompact and is  the colimit  of a small  filtered inductive system   of Stone spaces and open embeddings. 
\item $X$ is Hausdorff and any open cover of $X$ can be refined by a partition of $X$ consisting of open compact subspaces.
\item $X$ is a $td$-space.
\een
\end{thm}
\begin{proof} $\mathit 1 \Rightarrow \mathit 2$ is obvious since $\sQ(X)$ is a $\sJ$ as required.  \par
To prove $\mathit 2 \Rightarrow \mathit 3$ we pick any discrete quotient of $X$, $A \in \sJ$. Any $\pi_{A,D}:D \to A$ in $\sJ/A = \sJ/{\pi_A}$, 
is a direct sum of finite surjective maps $\pi_{A,D}^{-1}(\alpha) \to \{\alpha\}$. Then 
$X_\alpha := \limit_{D \in \sJ/A} \pi_{A,D}^{-1}(\alpha)$ is a Stone space, and \eqref{sumstone} holds. 
%we write $X = \limit_{\beta \in B} D_\beta$, where $(B,\leq)$ is a directed set, $D_\beta$ 
%is a discrete space,  and 
%the maps $\pi_{\beta,\beta'} : D_{\beta'} \to D_\beta$, for $\beta \leq \beta'$, are locally finite surjective. Then, for any fixed $\beta \in B$, the projection $\pi_\beta: X \to D_\beta$ 
%is proper and surjective, so that $X$ is the sum of the open compact subsets $C_y := \pi_\beta^{-1}(y)$, for $y \in D_\beta$. Any $C_y$ is the projective limit 
%of the cofinal subsystem of finite discrete sets $\pi_{\beta,\beta''}^{-1}(y) \to \pi_{\beta,\beta'}^{-1}(y)$, for $\beta \leq \beta' \leq \beta''$, and therefore is a Stone space. 
\par
To prove $\mathit 3 \Rightarrow \mathit 1$ we 
%may reduce to a single $X_\alpha$, that is to the case of $X$ a Stone space as follows. We 
give to the set $A$ the discrete topology. Then the projection $\pi_A: X \to A$ may be viewed as an object  of $\sQ(X)$. 
The morphisms $\pi_D:X \to D$ in $\sQ(X)$ for which there exists $\pi_{A,D}: D \to A$ such that $\pi_A =  \pi_{A,D} \circ \pi_D$ form a cofinal  projective subsystem $\sQ(X)/A$  in $\sQ(X)$. On the other hand $\sQ(X)/A$ is the product category $\prod_{\alpha \in A} \sQ(X_\alpha)$. For any topological space $Y$, a continuous map $f:Y \to X$ determines a  continuous map $\sigma = \pi_A \circ f:Y \to A$ and a partition of $Y$ in clopen fibers $\sigma^{-1}(\alpha)$, for $\alpha \in A$. So, 
%\beqa \begin{split}
%\hom_{\Top} (Y, X) =& \prod_{\sigma \in \hom_{\Top} (Y, A)} \hom_{\Top} ( \sigma^{-1}(\alpha), X_\alpha) = \prod_{\sigma \in \hom_{\Top} (Y, A)} \limit_{D_\alpha \in \sQ(X_\alpha)} \hom_{\Top} ( \sigma^{-1}(\alpha), D_\alpha)  =  \\
%&\limit_{D \in \sQ(X)} \hom_{\Top} (Y,D) = \hom_{\Top} (Y,\colimit_{D \in \sQ(X)} D) \;.
%\end{split}
%\eeqa
\beqa \begin{split}
\hom_{\Top} (Y, X) =& \prod_{\sigma \in \hom_{\Top} (Y, A)} \hom_{\Top} ( \sigma^{-1}(\alpha), X_\alpha) =\\ \prod_{\sigma \in \hom_{\Top} (Y, A)} &\limit_{D_\alpha \in \sQ(X_\alpha)} \hom_{\Top} ( \sigma^{-1}(\alpha), D_\alpha)  =  \\
 \limit_{D \in \sQ(X)} &\hom_{\Top} (Y,D) = \hom_{\Top} (Y,\colimit_{D \in \sQ(X)} D) \;.
\end{split}
\eeqa
\par
$\mathit 3 \Rightarrow \mathit 4 \Rightarrow \mathit 5$ is clear.
\par
$\mathit 5 \Rightarrow \mathit 3$ By Theorem~\ref{Lind}  $X$ is a sum of Lindel\"of spaces. On the other hand, $X$ is 
$0$-dimensional since its open Stone subspaces form a basis for the topology of $X$; so, it is a sum of $0$-dimensional Lindel\"of spaces. By Lemma~\ref{LindStone}, $X$ is a sum of  Stone subspaces.
% In fact $\mathit 5$ simply means that $X  = \bigcup_{\alpha \in A} C_\alpha$ is the union of  open Stone subspaces $C_\alpha$. Then $X$ is obviously locally compact. It is paracompact (resp. $0$-dimensional) since any $C_\alpha$ is.
 \par
$\mathit 3 \Rightarrow \mathit 6$ follows from $\mathit 2$ of Lemma~\ref{stonespace}. 
\par
$\mathit 6 \Rightarrow \mathit 7$ because $\mathit 6$ is just a reformulation of the definition of a $td$-space. 
\par
$\mathit 7 \Leftrightarrow \mathit 3$ is Corollary~\ref{LindStone2}.
\par
\end{proof}
\begin{rmk}\label{diffNGO} Our definition of a $td$-space is more restrictive than Ngo's \cite[\S 1]{ngo} in that we require paracompactness. On the other hand, a $td$-space which is countable at infinity in the sense of Ngo, is precisely a 
Lindel\"of $td$-space in our sense \cite[3.8.C (b)]{eng}. 
\end{rmk}
 \begin{rmk} \label{limdiscr} It follows from $\mathit 1$ of Theorem~\ref{STSspaces} that, for any $td$-space $X$ and any cofinal 
 projective subsystem $\sJ$ of $\sQ(X)$,  \eqref{prodis2} holds 
% \beq \label{limdiscr1}
% X = \limit_{D \in \sJ} D
% \eeq
 in $\Top$. Conversely, any  cofiltered projective subsystem $\sJ$ of $\sQ(X)$ such that \eqref{prodis2} holds is cofinal in $\sQ(X)$. In fact, 
 $$
 \hom_{{\cT}op}(X,X) = \limit_{D \in \sQ(X)} \colimit_{E \in \sJ} \hom_{{\cS}ets}(E,D)
 $$
 is not empty because it contains $\id_X$. Therefore, for any $D \in \sQ(X)$, there must exist $E \in \sJ$ and 
  an arrow $E \to D$ in $\sQ(X)$. So, $\sJ$ must be cofinal in $\sQ(X)$. 
% See remark~\ref{noncofinal} below for a counterexample. 
% \par 
%In partial contrast,
Similarly, for any cofinal 
inductive subsystem $\{C_\alpha\}_{\alpha \in A}$ of $\cK(X)$
  \beq \label{limcomp}
 X =\colimit_{\alpha \in A} C_\alpha \;.
 \eeq
Again, any filtered inductive subsystem $\{C_\alpha\}_{\alpha \in A}$ of $\cK(X)$  satisfying \eqref{limcomp} is necessarily cofinal in $\cK(X)$. In fact, let $C$ be any element of $\cK(X)$. Then there exist a finite number $\alpha_1,\dots,\alpha_n$ of indices in $A$ such that $C \subset C_{\alpha_1} \bigcup \dots \bigcup C_{\alpha_n}$. Since $A$ is filtered, there exists $\alpha \in A$ such that $C_\alpha \supset  C_{\alpha_1} \bigcup \dots \bigcup C_{\alpha_n} \supset C$.
  \end{rmk}
 \begin{rmk} \label{filtstone} It follows from $\mathit 4$ of Theorem~\ref{STSspaces} that a  $td$-space $X$ is a filtered union of open Stone subspaces. Moreover, any compact subset of $X$ is contained in an open compact subset of $X$. 
 \end{rmk}
% \begin{rmk} Notice that a zero dimensional paracompact non locally compact space does  not necessarily satisfy \eqref{prodis}. A counterexample is $X = \Z$ equipped with the $p$-adic topology. Then the projective limit in \eqref{prodis} is $\Z_p \neq \Z$. 
%\end{rmk} 
  \begin{prop} \label{contproj} 
  Let $X, Y$ be  $td$-spaces and let
   \beq \label{contproj01}  
  X = \limit_{D \in \sJ} D \;\;,\;\; Y = \limit_{D' \in \sJ'} D' 
  \eeq
 (resp.
     \beq \label{contproj02}  
X = \colimit_{\alpha \in A} C_\alpha \;\;,\;\; Y = \colimit_{\beta \in B} C'_\beta \; 
  \eeq
  with $C_\alpha \in \cK(X)$, $C'_\beta \in \cK(Y)$)
  be a representation of $X$, $Y$, as in \eqref{prodis2} (resp. as in $\mathit 5$ of Theorem~\ref{STSspaces}).
      Then 
 \beq \label{contproj1} 
 \begin{split}
 \hom_{\Top} (X,Y) & =    \limit_{D' \in \sJ'} \colimit_{D \in \sQ(X)} \hom_{\sQ(X)} (D,D') = \\  \limit_{\alpha \in A} &
  \colimit_{\beta \in B} \hom_{\Top} (C_\alpha,C'_\beta)\;.
  \end{split}
\eeq
  \end{prop}
  \begin{proof}
  Immediate.
  \end{proof}
%By \eqref{prodis} we may assume that $Y=E$ is discrete.  Then, for any continuous map $f: X \to Y$ we obtain a partition   $X = \coprod_{y \in Y} f^{-1}(y)$ of $X$ as a sum of $td$-spaces.  This partition admits a refinement which is a partition into a family of open compact subspaces of $X$ which corresponds to a discrete quotient $\pi_D: X \longrightarrow D$ in $\sQ(X)$.   So, $f$ factors as 
%  $$X \map{\pi_D} D  \map{g} Y
%  \,$$
%   and the equality to be proven is induced by $f \longmapsto g$. 
%  \par
%%Notice that if \eqref{prodis2} holds, then  $\sJ$ is cofinal in $\sQ(X)$ and 
%%  $$\hom_{\Top} (X,Y) = \colimit_{D \in \sJ} \hom_{\Top} (D,Y) \;.
%%  $$
%% Then  $$\Loc(X,k) = \colimit_{D \in \sQ(X)} \sC(D,k) \;.$$
\begin{cor} \label{mSTSspaces}  The following properties of a topological space $X$ are equivalent:
\ben
\item $X$ is a $td$-space and the projective system  $\sQ(X)$ admits a cofinal sequential projective sub-system of locally finite maps $D_{n+1} \map{\pi_{n+1,n}} D_n$, for $n \in \N$, where each $D_n$ is a countable discrete set.
\item $X$ is the limit  
    \beq \label{mprodis2}
X= \limit_{n \in \N}  \, D_n 
\eeq
of a sequential projective system $\{D_n\}_{n \in \N}$ of discrete spaces and locally finite surjective maps. 
\item $X$ is the sum 
\beq \label{msumstone}
X = \coprod_{n \in \N} X_n
\eeq
 of a countable family  of metrizable Stone spaces.
 \item $X$ is a second-countable  locally compact totally disconnected Hausdorff space.
  \item $X$ is a  locally compact totally disconnected Hausdorff space whose family $\Sigma(X)$ of clopen subsets is at most countable.
\item $X$ is metrizable and admits a cover of open Stone subspaces.
\item $X$   is  the  colimit  of a sequence  $\{X_n\}_{n \in \N}$ of metrizable Stone spaces and open embeddings. 
\item $X$ is Hausdorff and any open cover of $X$ can be refined by a countable partition of $X$ consisting of open compact subspaces.
\item $X$ is a metrizable $td$-space.
\item $X$ is a second countable $td$-space. 
\een
\end{cor}
%  \begin{cor} \label{STSprod} A finite product of Stone (resp. $td$) spaces is a Stone (resp. a $td$) space.
%  \end{cor}
%  \begin{proof} The case of Stone spaces follows, for example, from $\mathit 4$ of Lemma~\ref{stonespace}.
%  The case of a   $td$-space equally follows from 
%  Theorem~\ref{STSspaces}.
%  \end{proof}
  \end{subsection}
\begin{subsection}{Uniform structures on a $td$-space} \label{uniftd}
\begin{subsubsection}{Generalities on uniform spaces.} 
\label{genunif}
We slightly modify the definitions and notation of  \cite[\S 8.1]{eng} on general uniform spaces  
to avoid separation conditions. According to \cite {isbell} our spaces would rather be \emph{pre-uniform spaces}, but we 
do not insist on this terminology, since eventually our spaces of interest  will be separated. 
 \par
 We recall from \cite[\S 5.1]{eng} that for any set $X$, there are two relevant relations among covers $\sP,\sQ$ of $X$. Namely, $\sQ \leq \sP$ \ie $\sQ$ is a \emph{refinement} of $\sP$ meaning that for any $B \in \sQ$ there is $A \in \sP$ such that $B \subset A$, and $\sQ \leq^\ast \sP$ \ie $\sQ$ is a \emph{star refinement} of $\sP$ meaning that  for every  $B \in \sQ$ there is an $A \in \sP$ such that, if $C \in \sQ$ and $B \cap C \neq \emptyset$, then $C \subset A$.
\begin{rmk} \label{partleq0} Notice that for 
any two \emph{partitions} $\sP$ and $\sQ$ of a set $X$, we have
\beq  \label{partleq}
\sQ \leq^\ast \sP \Leftrightarrow \sQ \leq \sP \Leftrightarrow \mbox{any $B \in \sQ$ is contained in some}\;\; A \in \sP
\;.
\eeq 
\end{rmk} 
\par A uniform space $(X,\Theta)$ (on the topological space $X$) may be viewed as a topological space $X$ equipped with a distinguished family $\Theta$ of open covers, called \emph{uniform}, which satisfy the properties 
\ben 
\item if $\sQ \leq \sP$ and  $\sQ \in \Theta$, then $\sP \in \Theta$;
\item if $\sP,\sQ \in \Theta$, there exists $\sR \in \Theta$ such that $\sR \leq \sP$ and $\sR \leq \sQ$;
\item for any $\sP \in \Theta$ there is  $\sQ \in \Theta$ such that $\sQ \leq^\ast \sP$;
\een
A family of open covers $\Theta$ of the topological space $X$ is a \emph{basis of uniform covers} of a well-defined uniform space, still denoted by $(X,\Theta)$, on  $X$  if it satisfies 2,3 above, in which case an open cover $\sP$ of $X$ is uniform iff there  is a refinement $\sQ \leq \sP$, $\sQ \in \Theta$.  We say that $\Theta$, and the uniform space $(X,\Theta)$,  \emph{define a uniformity on the topological space $X$}. 
Let $(X,\Theta_X)$ and $(Y,\Theta_Y)$ be uniform spaces, where $\Theta_X$ (resp. $\Theta_Y$) is meant to be a basis of uniform covers on $X$ (resp. $Y$). A map  $f: X \to Y$ defines a morphism $f: (X,\Theta_X) \to (Y,\Theta_Y)$ of uniform spaces or is \emph{uniformly continuous},   if $f$ is continuous and, for any $\sP \in \Theta_Y$, the cover  $\{f^{-1}(B)\,|\, B \in \sP\,\}$ of $(X,\Theta_X)$ is uniform.  
\par We denote by ${\sU}nif$ the category of uniform spaces and by
$$ 
{\rm top} : {\sU}nif \longrightarrow \Top \;\;,\;\; (X,\Theta) \longmapsto X\;,
$$
the natural functor.  There is a second functor 
$$\neat: {\sU}nif \longrightarrow \Top  \;\;,\;\; (X,\Theta) \longmapsto (X,\Theta)_\neat\;,
$$
where $(X,\Theta_X)_\neat = X_\neat \in \Top$ is the set $X$ on which a basis of open neighborhoods of any point $x \in X$ is given by the family 
\beq
\label{assneat}
\cU(x) = \{A \subset X \,|\, \exists \; \sP \in \Theta_X \; \mbox{such that} \; A \in \sP \,\} \;.
\eeq
Then $(X,\Theta) \longmapsto \id_X$ gives a natural transformation ${\rm top}  \map{\varphi} \neat$. We will say that $(X,\Theta)$ is \emph{neat} if $\varphi((X,\Theta))$ is a homeomorphism. 
%If ${\rm top}((X,\Theta)) = (X,\tau)$, we say that the uniformity of $X$ induced by $\theta$ is  \emph{compatible  with} $\tau$ 
%or is \emph{a uniformity on $(X,\tau)$}. 
% By \cite[Cor. 8.1.12]{eng}, any uniformity on $(X,\tau)$ has a basis of uniform covers consisting of $\tau$-open covers. 
A uniform space $(X,\Theta)$ is \emph{discrete} if  $\{\{x\} | x \in X\}$ is a uniform cover of $(X,\Theta)$ (in particular, $X$ is then discrete and $(X,\Theta)$ is neat).  
% A uniform space $(X,\Theta)$ is separated if and only if the further property holds~:
%\ben
%\item[4.] for any pair of distinct points $x,y$ of $X$, there exists $\sP \in \Theta$ such that no member of $\sP$ contains both $x$ and $y$.
%\een  
   \begin{prop} Let $\Theta$ be a basis of uniform covers for a  uniformity on the set $X$. Then, 
a basis of entourages for the corresponding uniformity of $X$   is $\{U_\sP\}_{\sP \in \Theta}$ where 
 $U_\sP \subset X \times X$ is of the form  $U_\sP = \bigcup_{A \in \sP} A \times A$.
 \end{prop}
 \begin{proof}
See \cite[Prop. 8.1.16]{eng}.
 \end{proof}
\begin{rmk} \label{linearunif}
Let $R$ be any linearly topologized ring (resp. let $M$ be any linearly topologized $R$-module).
Then $R$ (resp. $M$) carries a canonical uniformity $\Theta_R$ (resp. $\Theta_M$) compatible with the given topology, called the \emph{$R$-linear 
uniformity} of $R$ (resp. $M$). A basis 
of uniform covers of $\Theta_R$ (resp. $\Theta_M$) is given by the covers of the form 
$\sP_J = \{x+J\}_{x \in R}$ for $J\in \cP(R)$ (resp. $\sP_P = \{x+P\}_{x \in M}$ for ${P\in \cP_R(M)}$). Such covers 
are in $(1,1)$-correspondence with  the discrete quotients 
$$
\pi_J : R \longrightarrow R/J \;\;\;\mbox{(resp.}\; \pi_P : M \longrightarrow M/P \;\mbox{)}
$$
for $J \in \cP(R)$ (resp. for ${P\in \cP_R(M)}$). 

For any uniform space $(X,\Theta_X)$, a uniformly continuous  $R$-valued function on $(X,\Theta_X)$ is a morphism $(X,\Theta_X) \to (R,\Theta_R)$ of uniform spaces. Notice that for $R$  as before, and  $R$-linearly topologized $R$-modules $M$, $N$, any continuous $R$-linear map $M \to N$ is uniformly continuous.
% as in section~\ref{lincat}.
%Then the topological ring 
% $k$ is not locally compact in  general, hence is not a $td$-space. 
% Let $k$ be  as in section~\ref{lincat}. 
% It carries a canonical  uniformity defined by the discrete quotients 
%$$
%\pi_J : k \longrightarrow k/J
%$$
%for $J \in \cP(k)$. 
\end{rmk}
%\par Notice that for 
%any two partitions $\sP$ and $\sQ$ of a set $X$, we have
%\beq  \label{partleq}
%\sQ \leq^\ast \sP \Leftrightarrow \sQ \leq \sP \Leftrightarrow \mbox{any $B \in \sQ$ is contained in some}\;\; A \in \sP
%\;.
%\eeq 
%\par Notice that if $\Theta$ is a family of partitions  of  a $td$-space $X$ into  a family of compact open subsets
%satisfying 1,2,4 above and such that, for any $\sP \in \Theta$ there is  $\sQ \in \Theta$ such that 
%$\sQ \leq \sP$, then $\Theta$ satisfies 3, as well. 
\end{subsubsection}
\begin{subsubsection}{$td$-uniformities.}
We now describe a full subcategory  of ${\sU}nif$ consisting of certain uniform spaces whose associated topological space is a $td$-space.
Since a $td$-space $X$ is obviously a Tychonoff space, there exist   uniform  structures on $X$  
\cite[Thm. 8.1.20]{eng}. 
%We are especially  interested in the uniform structures carried by a $td$-space $(X,\tau)$. 
 \begin{defn} \label{unifdef10tdpart} Let $X$ be a topological space.  A partition $\sP$ of $X$  into a family  of   open Stone subspaces  is called a \emph{$td$-partition} of $X$. A collection  $\Theta$ of $td$-partitions of $X$ which is a basis of uniform covers for some uniformity on $X$, will be called \emph{a basis of $td$-partitions on $X$}.
 \end{defn}
 
 \begin{notation} \label{thetajay} Clearly, we may identify any $td$-partition $\sP$ of $(X,\tau)$ with the corresponding  map 
\beq \label{partquot}
\begin{split}
\pi_{\sP}: X &\longrightarrow \sP
\\
x &\longmapsto [x]
\end{split}
\eeq
where $[x]$, a compact open subset of $X$, denotes the element of $\sP$ containing $x \in X$. This map is clearly continuous and proper if $\sP$ is equipped with the discrete topology, so that \emph{we identify a $td$-partition $\sP$ with the corresponding  element  $\pi_\sP \in \sQ(X)$.} 
A basis $\Theta$  of $td$-partitions on $X$ may be regarded equivalently as
\ben  
 \item a projective cofiltered subsystem $\sJ$ of $\sQ(X)$;
 \item a family $\Theta$ of $td$-partitions of $X$ cofiltered by the relation $\leq$  of subsection~\ref{genunif}.
 \een
We write $\sJ \leftrightarrow \Theta_\sJ$ for the previous correspondence.  
% Notice that the identity on $X$ induces a continuous map 
% $$\iota_{(X,\Theta)}: (X,\tau) 
%  \map{}   {\rm top}(X,\Theta)\;.$$
\end{notation}
 \begin{prop} \label{tdunif}  Let $X$ be a $td$-space and let  $\Theta = \Theta_\sJ$ be a basis
 of $td$-partitions of $X$ where $\sJ$ is a cofiltered projective  subsystem  of $\sQ(X)$. Then $(X,\Theta)$ is neat  if and only if  one of the following equivalent conditions is satisfied
\ben 
\item  
%for any $x \in X$, any compact open neighborhood $C$ of $x$ in $(X,\tau)$ contains $A \subset X$ such that 
%$x \in A \in \sP$, for some $\sP \in \Theta$ (resp. 
for any $x \in X$, a fundamental system of neighborhoods of $x$ in $X$ is given by the subsets $A \subset X$ such that 
$x \in A \in \sP$, for some $\sP \in \Theta$;
%\item
%for any $\sP,\sQ \in \Theta$   there exists $\sR \in \Theta$ such that $\sR \leq \sP$ and $\sR \leq \sQ$.
%\een
%The previous condition $\mathit 1$  on $\Theta$ is equivalent to the following condition $\mathit 2$ on the projective cofiltered subsystem $\sJ$ of $\sQ(X)$:
%\ben 
%\item[2.] 
\item the family of fibers of maps $\pi_D : X \map{} D$ in $\sJ$ is a basis for the topology of $X$;
\item any point $x \in X$ has a compact open neighborhood $U$ such that the subsystem 
$$(\pi_{|U} : U \map{} \pi(U))_{\pi \in \sJ} \subset \sQ(U)$$ 
 is cofinal in $\sQ(U)$. 
\een
\end{prop}  
\begin{proof} Easy.
\end{proof} 
  \begin{defn} \label{unifdef10} 
  An object $(X,\Theta)$ of ${\sU}nif$ is a  \emph{$td$-uniform space} if $X$ is a $td$-space and 
  $\Theta = \Theta_\sJ$ is a  basis of $td$-partitions on $X$ as in Proposition~\ref{tdunif}. We sometimes shorten 
  $(X,\Theta_\sJ)$  into  $(X,\sJ)$. 
%  Let $X$ be a $td$-space and let $\Theta$ be a basis of $td$-partitions on $X$  as in Proposition~\ref{tdunif}. 
%  We say that $\Theta$ is  a \emph{basis of $td$-partitions of $(X,\tau)$} if $\iota_{(X,\Theta)}$ is a homeomorphism. 
%  A uniformity on $(X,\tau)$ which admits a basis of $td$-partitions will be called a \emph{$td$-uniformity}. 
%  The corresponding 
%  uniform space   $(X,\Theta)$  is called a \emph{$td$-uniform space}.   
  The full subcategory of 
${\sU}nif$ consisting of $td$-uniform  spaces, will be denoted by $td-\cU$. 
  \end{defn}
  \begin{prop}\label{morftd} Let   $(X,\Theta_\sJ)$, $(Y,\Theta_\sH)$ be objects  of $td-\cU$. 
  %with ${\rm top}(X,\Theta_\sJ) = (X,\tau)$ (resp. ${\rm top}(Y,\Theta_\sH) = (Y,\tau)$). 
A morphism 
  $$(X,\Theta_\sJ) \map{} (Y,\Theta_\sH)$$
 in $td-\cU$  is a continuous map $X \map{f} Y$  such that, for any $\pi_E:Y \map{} E$ in $\sH$, there exists $\pi_D :X \map{} D$ and a map of sets $f_{D,E}:D \map{}E$ such that   $\pi_E \circ f = f_{D,E} \circ \pi_D$.
 \end{prop}
 \begin{proof}
 Clear.
 \end{proof}
% \begin{rmk} \label{neattd} Let $(X,\Theta) \in td-\cU$. Then  $(X,\Theta)_\neat = X_\neat$ is a $td$-space. 
% \par In fact, let us recall  that $\varphi = \varphi((X,\Theta)): X \map{(1:1)} X_\neat$ is a continuous bijection and let $\sP  \in \Theta$ be 
% a $td$-partition of $X$ in the basis of $td$-partitions $\Theta$. Then, for any $A \in \sP$, $\varphi(A)$ is a compact open subset of 
%$X_\neat$. 
%Any $x \in X_\neat$ has a fundamental system of open neighborhoods which are open 
% Stone subspaces of $X$
% \end{rmk}
\begin{rmk} \label{unifdef101}  A  topological space $X$ on which there exists a $td$-partition (\emph{a fortiori} a $td$-uniformity)
%uniformity (compatible with $\tau$) 
is necessarily a $td$-space. 
 Conversely, any neat uniform structure $(X,\Theta)$ on a $td$-space $X$ may be refined to a neat 
 $td$-uniformity $(X,\Lambda)$,    as follows.  
 %still compatible with $\tau$.  
% To see this, let $\cC$ be a basis of uniform covers for $\cU$.
% As recalled above, we may (and will) assume that any $\sQ \in \cC$ is an open covering of $X$. 
Let $\sQ \in \Theta$ be a uniform cover in $\Theta$.  
 Since $X$ is a $td$-space, for any $x \in X$,  we may pick a compact open neighborhood  $K_x$ of $x$ in $X$,  such that for some $C \in \sQ$, $x \in K_x \subset C$. 
The cover $\sK := \{K_x\}_{x \in X}$ refines $\sQ$. Let 
 $X$ be the sum of a family of Stone spaces as in \eqref{sumstone}. Then, for any $\alpha \in A$ such that  $x \in X_\alpha$, let 
$$
K_{x,\alpha} = X_\alpha \cap K_x
$$
So, 
$$
\sP = \{ K_{x,\alpha} \,|\, \alpha \in A\; \mbox{and}\; x \in X_\alpha\,\}
$$
is a cover on $X$ by  open Stone subspaces which refines $\sQ$.
Then, for any $\alpha \in A$, there is a finite index set $F_\alpha =\{1,2,\dots,n_\alpha\}$ and a family  $\{K_{\alpha,j}\}_{j\in F_\alpha}$ in $\sP$, such that $$X_\alpha = \bigcup_{j\in F_\alpha} K_{\alpha,j}\;.$$
We may then set 
$$
U_{\alpha, 1} = K_{\alpha, 1}\;,\;U_{\alpha, 2} = K_{\alpha, 2} - K_{\alpha, 2}\cap U_{\alpha, 1} \;,\; \dots\;,\; U_{\alpha, n_\alpha} = K_{\alpha, n_\alpha} - K_{\alpha, n_\alpha} \cap U_{\alpha, n_\alpha-1} \;.
$$
Finally,  we have obtained a $td$-partition 
$$\sP' := \{U_{\alpha, j} \;|\; \alpha \in A\;,\; j \in F_\alpha\;\}$$
of $X$ which refines ($\sP$ hence) $\sQ$. Notice that there may not exist a further uniform cover $\sQ' \in \Theta$ which refines 
$\sP'$. See a counterexample in the next remark~\ref{counterex}. 
  \end{rmk}
  \begin{rmk}\label{counterex} 
  Let $D$ be an infinite discrete space. We consider the family $\cD$ of partitions  $\sD$ of $D$ into a finite set 
  $\{D_1,\dots,D_n, D_\infty \}$ of subsets $D_i \subset D$, where $D_1,D_2,\dots,D_n$ are finite subsets of $D$. Then $\cD$ is a basis  of uniform covers for a neat uniformity of $D$ (that is 
  $(D,\cD)_\neat = D$, equipped with the 
  discrete topology). Any $td$-refinement of the cover $\sD$ is necessarily infinite, hence cannot be refined by any  $\sD' \in \cD$. 
  \end{rmk}
%   \begin{rmk}\label{colimunif} The category of uniformities on a topological space $(X,\tau)$ has a final object. It follows from Remark~\ref{unifdef101} that, on a $td$-space $(X,\tau)$,  there is a final $td$-uniformity $\Theta_{\max}$ compatible with $\tau$. 
% For any $td$-uniform space $(X,\Theta)$ on $(X,\tau)$, the identity $\id_X$ is a uniform map $(X, \Theta) \map{} (X,\Theta_{\max})$.   \end{rmk}
 
\begin{rmk} \label{topvsunif}       
By   \cite[Ex. 8.2.B  (c)]{eng}, \cite[Chap. II, \S 2, N. 7]{topgen}  
in the category ${\sU}nif$  projective limits are representable. 
For a projective system $\fX_\bullet =((X_\alpha,\Theta_{\sJ_\alpha}))_{\alpha \in A}$  in $td-\cU$, the limit 
$(X,\Theta) := \limit_\alpha \fX_\alpha$ 
of $\fX_\bullet$ in ${\cU}nif$ is obtained as follows. We set $X = \limit_\alpha X_\alpha$ in ${\cT}op$, so that 
$X$ is a $td$-space.   For any $\alpha$, let $X \map{f_\alpha} X_\alpha$ be the projection. Then any proper quotient map $\pi_{\alpha,D} : X_\alpha \map{} D$ in  $\sJ_\alpha$ induces a proper quotient map $\pi_{\alpha,D} \circ f_\alpha : X \map{} D$ to the discrete space $D$. The union $\sJ = \bigcup_\alpha \sJ_\alpha$ is a projective sub-system of $\sQ(X)$, and $(X,\Theta) = (X,\Theta_\sJ)$. We conclude that $td-\cU$ is closed under limits in ${\cU}nif$. 
\par 
 In particular,  
for any  $td$-uniform space $(X,\Theta_\sJ)$ and any $\pi_D:X \map{} D$ in $\sJ$, let  $\sJ_D = \{\pi_D\}$ and let us write $\Theta_D$ for $\Theta_{\sJ_D}$.
% and let 
%let $\sJ_D$ be the projective sub-system of $\sJ$ of $\pi_E:X \map{} E$, such that $\pi_D$ factors in $\sQ(X)$ as 
%$\pi_D = \pi_{E,D} \circ \pi_E$
Then $(X,\Theta_D) = (X,\{\pi_D\})$ and 
\beq \label{projlimunif}
(X,\Theta_\sJ) = \limit_{D \in \sJ} (X,\Theta_D) = \limit_{D \in \sJ} (X,\{\pi_D\}) \;,  
\eeq
in the category $td-\cU$.  
\end{rmk} 
% \begin{cor} \label{tdunif21} 
%  Let $(X,\tau)$ be a $td$-space and let $\sJ$ be a projective cofiltered subsystem of $\sQ(X)$. 
%  Let $(X,\tau_\sJ)$ be the  topological space associated to the $td$-uniform space $(X,\Theta_\sJ)$, defined by \eqref{projlimunif}. Then \ben
%  \item
%  the identity map on $X$ induces a continuous map $\iota_{X,\sJ}: (X,\tau) \map{} (X,\tau_\sJ)$; 
%  \item $\iota_{X,\sJ}$ is a homeomorphism   if and only if, for any $x \in X$, the family 
%\beq  \label{tdunif210} 
% \cF_x = \{ \pi_D^{-1}(\pi_D(x)) \,|\, (\pi_D : X \longrightarrow D) \in \sJ\;\}
%\eeq
% is a basis of neighborhoods of $x$ in $(X,\tau)$. 
% \een
%\end{cor}
%\begin{proof} This is a reformulation of Proposition~\ref{tdunif}.
%\end{proof}  
\begin{defn} \label{metrunif1}  A neat $td$-uniform space $(X, \Theta_\sJ)$ as in 
Definition~\ref{unifdef10} is \emph{metrizable} if the topological space $X$ is metrizable. 
 \end{defn}
 \begin{prop} \label{metrunif2}  A neat $td$-uniform space $(X, \Theta_\sJ)$  is  metrizable if and only if $\sJ$ admits a countable 
 cofinal subsystem $\sJ'$ such that, for any $X \map{\pi_D} D$ in $\sJ'$, $D$ is countable. 
 \end{prop}
 \begin{proof}
 Clear. 
 \end{proof}
\begin{rmk} \label{maxunif} 
 It follows from Exercise 8.1.B (b) of \cite{eng} that, for any $td$-space $X$ (more generally, for any Tychonoff space),  there is a finest uniformity $\Theta_{\univ}$ of $X$ compatible with the topology of $X$.  It is called the \emph{universal uniformity} (on the Tychonoff space $X$). Since any uniformity on a $td$-space compatible with the topology is refined by a $td$-uniformity 
we have 
\beq \label{maxunif1}
(X, \Theta_{\univ}) = (X,\Theta_{\sQ(X)}) \;.
\eeq
 \end{rmk}
\begin{lemma} \label{sts complete}
A  uniform $td$-space is complete. 
\end{lemma}
\begin{proof}
In fact, for any Cauchy filter $\cG$ for $(X,\Theta)$ and any $\sP \in \Theta$, there exists $C \in \cG$ such that $C \times C \subset U_\sP$. So,   $C \times C \subset \bigcup_{A \in \cP} A \times A$ and, if there exist $A_1 \neq A_2$ in $\cP$ such that $C \cap A_i \neq \emptyset$ for $i=1,2$, then $(C \cap A_1) \times (C \cap A_2)$
cannot be contained in $U_\sP$. Therefore, there exists $A \in \sP$ such that $C \subset A$.  But then one of the elements of $\cG$, namely $A$, is compact.   
 Then $\bigcap_{G \in \cG} \ol{G} \neq \emptyset$, since all $\ol{G} \cap A$ have the finite intersection property in the compact subset $A$. So, $\bigcap_{F \in \cG} \ol{G}$ consists of a single point $x \in X$, and the filter $\cG$ converges to $x$.
\end{proof}
\begin{rmk} \label{noncofinal} 
%A cofinal projective subsystem $\sJ$ in $\sQ(X)$,  necessarily satisfies property \eqref{prodis2}.
%But the converse does not hold in general. 
The $td$-space $\Q_p$ carries a natural uniformity  for which a basis of uniform covers is $\Theta := \{\sP_h\}_h$, for $h =0,1,2,\dots$, where $\sP_h$ consists of the family of balls of the same radius $p^h$~:
$$\sP_h = \{ a + p^h \Z_p\,|\, a \in \Q_p  \;\} \;.
$$
The corresponding  cofiltered projective system 
of discrete quotients of $\Q_p$, by our dictionary, is
$$
\pi_h : \Q_p \longrightarrow \Q_p/p^h \Z_p \;,\; h =0,1,2,\dots\;.
$$
Then $\sJ := (\pi_h)_{h \in \N}$  and $\Theta =\Theta_\sJ$ clearly satisfies \eqref{projlimunif}. Moreover, 
$(\Q_p,\Theta_\sJ)$ satisfies the conditions of Proposition~\ref{tdunif}, so that 
it is neat, \ie 
$(\Q_p,\Theta_\sJ)_\neat$ is 
$\Q_p$ with its usual topology. 
 But $\Q_p$ does not coincide with $\limit_h \Q_p/p^h \Z_p$ in $\Top$ because $\sJ$ is not cofinal in $\sQ(\Q_p)$. For example,  the $td$-partition 
$$\{\Z_p\} \bigcup  \{x + p^{-v_p(x)}\Z_p \,|\, x \in \Q_p - \Z_p\} 
$$
of $\Q_p$
cannot be refined by any $\pi_h$ as before. Then formula \eqref{prodis2} does not hold.  Still, we have 
$$
(\Q_p,\Theta_\sJ) = \limit_h(\Q_p,\sP_h)
$$
in $td-\cU$. 
\end{rmk}

\begin{notation} \label{booleanradius} Let $(X,\Theta_\sJ)$  be a $td$-uniform space. 
%For any objects $(X,\Theta_\sJ)$ and $(Y,\Theta_\sH)$ of 
%$td-\cU$ morphisms in 
%$$\hom_{td-\cU}((X,\Theta_\sJ),(Y,\Theta_\sH))$$
%are often called \emph{uniformly continuous maps} of $(X,\Theta_\sJ)$ to $(Y,\Theta_\sH)$, or simply 
%of $X$ to $Y$ if there is no ambiguity on the uniform structures. 
A discrete quotient  $\pi_D: X \to D$ in $\sJ$,    (resp. any fiber $\pi_D^{-1}(d)$, for $d \in D$), will also be called a \emph{radius} (resp. a \emph{ball of radius $D$}) of 
$(X,\Theta_\sJ)$, and the open subsets of $X$ of the form $\pi_D^{-1}(A)$, for any $A \subset D$, will be called \emph{uniformly measurable subsets of radius $D$}. 
They form a boolean subalgebra $\Sigma_D(X)$ of $\fP(X)$. The union of $\Sigma_D(X)$, for $D \in \sJ$, is the boolean subalgebra 
$\Sigma_\sJ(X) \subset \fP(X)$ of \emph{uniformly measurable subsets of  $(X,\Theta_\sJ)$}.  For $U \subset X$, $\chi_U$ is a 
uniformly continuous map of $(X,\Theta_\sJ)$ to $k$ if and only if $U \subset X$ is a uniformly measurable subset of $X$.
\end{notation} 
\begin{rmk} \label{linearunif1} A topological ring  $k$ as in section~\ref{lincat} 
is not locally compact in  general, hence is not a $td$-space. In particular its canonical $k$-linear uniformity $\Theta_k$ is not, in general, a $td$-uniformity. 
\end{rmk}
\begin{prop} \label{maptodiscrete} Let $(X,\Theta_\sJ)$ be a $td$-uniform space. For any discrete uniform space $E$,  a continuous map $f:X \to E$ is uniformly continuous  on $(X,\Theta_\sJ)$ if and only if $f$ factors via some $X \map{\pi_D} D$, with $D \in \sJ$. Therefore, if 
$(Y,\Theta_\sH)$ is a second object of $td-\cU$, 
\beq 
\label{homunif} 
\hom_{td-\cU} ((X,\Theta_\sJ), (Y,\Theta_\sH))=  \limit_{E \in \sH} \colimit_{D \in \sJ} \hom_{\sQ(X)}(D,E) \; .
\eeq
\end{prop} 
\begin{proof}
Immediate.
\end{proof}
\begin{rmk} \label{maxunifmaps} Let $(X,\Theta_\sJ)$, $(Y,\Theta_\sH)$ be objects of $td-\cU$, and assume $\Theta_\sJ = \Theta_{\univ}$, the finest uniformity compatible with the topology of $X$. Then \cite[8.1.C (a)]{eng}
\beq 
\label{homunifmax} 
\hom_{td-\cU} ((X,\Theta_{\univ}), (Y,\Theta_\sH))= \hom_{td-\cS}(X,Y) \; .
\eeq
\end{rmk}
\end{subsubsection}
\end{subsection}
\end{section}

\begin{section}{Continuous functions on a $td$-space} \label{STScontinuity}  
\begin{subsection}{Continuous 
%($\sC(X,k) \in \limit_{\cK(X)} \cCLM^\can_k$) 
and uniformly continuous 
%($\sC_\unif(X,k) \in  \cCLM^\can_k$) 
functions}
We fix the topological ring  $k$ as in section~\ref{lincat}.
  \begin{notation}  For any topological space $X$, 
 we denote by $\Loc(X,k)$  the $k$-module of locally constant (automatically continuous) functions $X \longrightarrow k$. 
  \end{notation}
\begin{defn} \label{unifcontdef} Let $X$ be a $td$-space 
%as in Theorem~\ref{STSspaces} 
(resp. let $(X,\Theta_\sJ)$ be as in Definition~\ref{unifdef10}). 
%a $td$-uniform space such that ${\rm top}()$ 
%be a  uniform $td$-space, where $X$ is the cofiltered projective limit of  $D \in \sJ$ 
%as in  \eqref{projlimunif}). 
We denote by $\sC(X,k)$ (resp. 
$\sC_\unif(X,k) = \sC_\unif((X,\Theta_\sJ),k)$) the $k$-module of continuous (resp. uniformly continuous) 
functions $X \to k$, equipped with the topology 
of uniform convergence on compact subsets of $X$ (resp. of uniform convergence on $X$). We also denote by $\sC^\circ(X,k)$ the $k$-canonical module $\sC(X,k)^\can$.
\end{defn}

%If $k$ is discrete and $X$ is an STS space, $\sC(X,k)$ is the $k$-module of locally constant functions $X \to k$. A fundamental system of open submodules is $\{U(C)\}_C$  indexed by open compact subsets $C$ of $X$, where 
%$$
%U(C) = \{ f: X \longrightarrow k \,|\,f(x) = 0\,,\, \forall\, x \in C\,\}
%$$
\begin{rmk} \label{ringtop} In the situation of Definition~\ref{unifcontdef}, it is clear that  both rings $\sC(X,k)$ and $\sC_\unif(X,k)$ are linearly topologized, \ie  admit a basis of neighborhoods of 0 consisting of open ideals.  It follows from Remark~\ref{filtstone} that a  basis of open submodules of $\sC(X,k)$ is the family of ideals $\{U(C,I)\}_{C, I}$ indexed by open compact subsets $C$ of $X$ and $I \in \cP(k)$, where 
\beq \label{complim0} U(C,I) =\{ f \in \sC(X,k)\,:\, f(x) \in I\,,\, \forall x \in C\,\} \;.
\eeq
If $X$ is a metrizable $td$-space, then $\sC(X,k) \in \cCLMou_k$. 
\par
A  basis of open submodules of $\sC_\unif(X,k)$ is the family of ideals $\{\sC_\unif(X,I)\}_{I \in \cP(k)}$. It follows from the definition that there is a natural  continuous injection 
\beq \label{uniftocont}
\sC_\unif((X,\Theta_\sJ),k) \map{} \sC(X,k) \;.
\eeq
\end{rmk}
\begin{rmk} \label{unifmaxhom}  It follows from Remark~\ref{maxunifmaps} that 
$$\sC_\unif((X,\Theta_{\univ}),k) =\sC^\circ(X,k) \in \cCLMcan_k\;,$$  
\ie coincides with  $\sC(X,k)^\for$ equipped with the naive $k$-canonical topology. 
\end{rmk}

We have
\beq \label{complim1}
\sC(X,k) = \limit_{C \in \cK(X), I \in \cP(k)} \sC(C,k/I) 
\eeq
where $\sC(C,k/I)$ is discrete and the limit is taken in $\cCLMu_k$.
Therefore
\beq \label{discrlim}
\sC(X,k) = \limit_{I \in \cP(k)} \sC(X,k/I)  
\eeq
and
\beq \label{complim2}
\sC(X,k) = \limit_{C \in \cK(X)} \sC(C,k)  \;.
\eeq
We also observe that 
\beq \label{complim3}
\sC^\circ(X,k) = \limit^\can_{C \in \cK(X)} \sC(C,k)  \;.
\eeq
A discrete space $D$ may be viewed as a uniform space in our sense only if it is equipped with the discrete uniformity
$\Theta_D$. Then by $\sC_\unif(D,k)$ we understand 
$$\sC_\unif((D,\Theta_D),k) = \sC^\circ(D,k) \in \cCLMcan_k\;.$$ 
 We have  
\beq \label{unifcont1}
\sC_\unif(D,k) =  \limit_{I\in \cP(k)}  ({\prod}_{x \in D} (k/I) \chi_x^{(D)})^\dis
= {\prod}_{x \in D}^\can k \chi_x^{(D)}   \;.
\eeq
So, $\sC_\unif(D,k) = k^D$ equipped with the naive canonical topology. 
%while
%\beq \label{discrcont1}
%\sC(D,k) =  \limit_{I\in \cP(k)}  \prod_{x \in D} (k/I) \chi_x^{(D)} 
%= \prod_{x \in D}  k \chi_x^{(D)}  \in \cCLMu_k \;,
%\eeq
By \eqref{homunif},  for any  uniform $td$-space $(X,\Theta_\sJ)$ we have
\beq \label{unifcont}
\sC_\unif((X,\Theta_\sJ),k) = \colimit^\un_{D \in \sJ} \sC_\unif(D,k)   \in \cCLMcan_k \;.
\eeq 
\begin{lemma} \label{contcompl} Let $C$ be a Stone space. Then 
\beq \label{contcompl01}
\begin{split}  \sC(C,k) = \limit_{I \in \cP(k)} & \sC(C,k/I)^\dis =   \limit_{I \in \cP(k)} \colimit^\un_{F \in \sQ(C)} (k/I)^F =  \\
&\colimit^\un_{F \in \sQ(C)} k^{(F,\un)}\;,
\end{split}
\eeq
where $\sC(C,k/I)$ is a discrete $k/I$-module and any $F \in \sQ(C)$ is a finite discrete space. 
In particular,  \ben
\item
 $\sC(C,k) \in \cCLMcan_k$;
 \item $\Loc(C,k)$ is dense in $\sC(C,k)$;
 \item For any uniformity $\Theta$ on $C$ we have 
\beq \label{contcompl02}
\sC_\unif((C,\Theta),k) = \sC(C,k)  \;.
\eeq
\een
\end{lemma}
%\begin{rmk} \label{locdense} 
%An elementary argument  shows that $\Loc(C,k)$ is dense in $\sC(C,k)$.
%%An elementary argument reducing to the case of $X$ a Stone space shows that $\Loc(X,k)$ is dense in $\sC(X,k)$.
%\end{rmk}
\begin{proof} \eqref{contcompl01} is clear. 
\par
$\mathit 1.$ Follows from \eqref{contcompl01} and Remark~\ref{limcan}.
\par 
$\mathit 2.$ Follows from \eqref{contcompl01} and the fact that in $\Mod_k$
$$\Loc(C,k) = \colimit_{F \in \sQ(C)} k^{(F)}\;.
$$
%case of $k =k/I$ discrete. 
\par
$\mathit 3.$ The only non trivial part is \eqref{contcompl02} which is well known and follows from the fact that, if $\Theta = \Theta_\sJ$, then, for any $\pi_D: C \to D$ in $\sJ$, $D$ is finite so that $\sC_\unif((D,\Theta_D),k) = \sC(D,k) = k^D = k^{(D)}$ with the $k$-canonical topology. 
% Let $\{I_n\}_{n \in \N}$ be a sequence of open ideals of $k$ which is a basis of neighborhoods of $0$ and satisfies $I_{n+1} \subset I_n$, for any $n \in \N$.  Then $ \{U(X,I_n)\}_n$ is a basis of open $k$-submodules of 
%$\sC(X,k)$. Let $n \longmapsto f_n$  be a  Cauchy sequence in $\sC(X,k)$. Then for any $x \in X$,  $n \longmapsto f_n(x)$ converges in $k$. We let $f(x)$ be its limit. The function $x \longmapsto f(x)$ is continuous. 
%The first part follows immediately from the definition. We prove that $\sC(X,k) \in \cCLMcan_k$. 
%Let $\{I_n\}_{n \in \N}$ be a sequence of open ideals of $k$ which is a basis of neighborhoods of $0$ and satisfies $I_{n+1} \subset I_n$, for any $n \in \N$. The map $\sC(X,k/I_{n+1}) \longrightarrow \sC(X,k/I_n)$ is surjective $\forall n$, since a function $f \in \sC(X,k/I_n)$ is 
%so that the the same holds for the map $\sC(X,k) \longrightarrow \sC(X,k/I_n)$
%We have an exact sequence of M-L projective systems 
%$$
%0 \longrightarrow  I_n\sC(X,k)  \longrightarrow \sC(X,k) \longrightarrow \sC(X,k/I_n) \longrightarrow 0$$
\end{proof}
\begin{lemma} \label{denslc} For any $td$-space $X$, 
$\Loc(X,k)$ is dense in $\sC(X,k)$.
\end{lemma}
\begin{proof}
For any $C \in \cK(X)$, any $f \in \sC(X,k)$, and any $I \in \cP(k)$, there is a finite $td$-partition $\{U_1,\dots,U_n\}$ of $C$ such that $f(x)-f(y) \in I$ as soon as  $x,y \in U_i$ for some $1 \leq i \leq n$. Let $\sP$ be a $td$-partition of $X$ such that, for any $1 \leq i \leq n$, the elements of $\sP$ contained in $U_i$ cover $U_i$. Then, for any $i$, let $x_i \in U_i$  and  let 
$$\psi = \sum_{i=1}^n f(x_i) \chi_{U_i}^{(X)} \in \Loc(X,k) \;.$$  
Then $f-\psi \in U(C,I)$. Therefore $\Loc(X,k)$ is dense in $\sC(X,k)$. 
\end{proof}
%\begin{lemma} \label{contcompl} Let $X$ be a Stone space. Then 
%$$\sC(X,k) = \limit_{I \in \cP(k)}  \sC(X,k/I)  \in \cCLMcan_k\;,$$
%where $\sC(X,k/I)$ is a discrete 
%\end{lemma}
%\begin{proof}
%$$ \sC(X,k/I) =  \colimit_{D \in \sQ(X)}  \sC(D,k/I)^\dis $$
%\end{proof}
%\begin{lemma}\label{limprocomp} For any STS space $X$, 
%$\sC(X,k)$ is the limit in $\cCLMu_k$ of the cofiltered projective system in $\cCLMcan_k$ $\{\sC(C,k)\}_C$ 
%where $C$ runs  over the compact open subsets of $X$.  Namely, 
%we have: 
%\beq 
%\label{limprocomp1}  
%\sC(X,k) =  \limit_{C \subset C' \subset X} (\sC(C,k), \pi_{C \subset C'}) \; ,
%\eeq
%where $C,C'$ run  over the compact open subsets of $X$, and, for $C \subset C'$, 
%$$\pi_{C \subset C'} : \sC(C',k) \longrightarrow \sC(C,k)\;\;,\;\; f \longmapsto f_{|C}\;,
%$$
%is the restriction to a subspace.   
%\end{lemma}
%\begin{proof}
%Immediate from the definition.
%\end{proof}
%\begin{cor} For any STS space $X$
%$$\sC(X,k) = \limit_{I \in \cP(k)}  \sC(X,k/I) \;.$$
%\end{cor}
%\begin{proof}
%Follows from the combination of  Lemmas \ref{limprocomp} and \ref{contcompl}. 
%\end{proof}
%$$\sC(X,k) = \colimit^\un_{D \in \sQ(X)}  \sC(D,k) $$
%where the topological space $D$ is discrete  and finite. 
 %So, $\sC(X,k)$ is a  $k$-algebra object of $\cCLMu_k$.
 \begin{cor} \label{contdiscr} Let $D$ be a discrete topological space. Then from \eqref{complim2} we deduce
\beq  \label{limprocomp2}  
\sC(D,k) =  \limit_{F \in \cF(D)} \sC(F,k) \; .
\eeq 
Explicitly
\beq  \label{limprocomp3}   \sC(D,k) = \prod_{x \in D} k \chi^{(D)}_x \;.
\eeq
 \end{cor} 
 \begin{prop} \label{indconst} For any $D \in \sQ(X)$ the pull-back of functions induces a closed embedding
$$
\pi_D^\ast :  \sC(D,k) \longrightarrow \sC(X,k)\;\;,\;\; f \longmapsto f \circ \pi_D \;.
$$
\end{prop}
\begin{proof} The map $\pi_D^\ast$  is clearly injective.  The topology of $\sC(D,k)$ viewed as a subset of $\sC(X,k)$ is the topology of uniform convergence on compact open subsets, so that it is  a subspace of $\sC(X,k)$. Since $\sC(D,k)$ is complete, it is a closed subspace of $\sC(X,k)$. 
 \end{proof}
For any $td$-space $X$ and for $\pi_D:X \longrightarrow D$ in $\sQ(X)$,  
 $\sC(D,k)$ identifies to the sub-object 
 $$
 \sC(D,k) = \{f:X \to k \,:\, \forall \, x \in D \, ,\, f\,\mbox{is constant on} \, \pi_D^{-1}(x)\,\} 
 $$
 of $\sC(X,k)$ in $\cCLMu_k$. Because of \eqref{prodis}  
 it follows from Proposition~\ref{contproj}  
that, in $\Mod_k$,
\beq
\label{contdiscr1}\colimit_{D \in \sQ(X)} \sC(D,k) = \Loc (X,k)\;,
\eeq
while, in $\cCLMu_k$,
\beq
\label{contdiscr2}\colimit^\un_{D \in \sQ(X)} \sC(D,k) = \sC(X,k)\;.
\eeq
For $D$ discrete, the  natural morphism
\beq
\label{unifdiscr1}
\sC_\unif((D,\Theta_D),k) \longrightarrow \sC(D,k)
\eeq
coincides with
\beq
\label{unifdiscr2}
{\prod}_{x \in D}^\can k \chi_x^{(D)}  \longrightarrow 
{\prod}_{x \in D}  k \chi_x^{(D)} 
\eeq
hence is  bijective. 
%\begin{cor} \label{unimeasdisc} For $D,E$ discrete spaces, the map 
%\beq  \label{unimeasdisc1} f \wt^\un_k g \longmapsto ((x,y)  \longmapsto f(x)g(y))
%\eeq
%induces a natural injective morphism in $\cCLMcan_k$  
%\beq  \label{unimeasdisc2} \cL^{(D,E)}_\unif: \cC_\unif(D,k) \wt^\un_k \cC_\unif(E,k) \map{}  \sC_\unif(D \times E,k) \;.
%\eeq
%\end{cor}
%\begin{proof} For $D$ discrete  $\sC_\unif(D,k)$ is simply $k^D$ equipped with the naive canonical 
%topology; so the existence of the morphism \eqref{unimeasdisc2} and the fact that it is injective is clear.  
%\end{proof}
%\begin{rmk} \label{prodfcts0unif1}
%The fact that \eqref{unimeasdisc2} is 
%not in general an isomorphism,  is already visible when $k$ is discrete and $D$, $E$ are discrete but  infinite. 
%In fact in the latter case, $\sC_\unif(D,k) = (k^D)^\dis$, $\sC_\unif(E,k) = (k^E)^\dis$,  and $\sC_\unif(D \times E,k) = (k^{D \times E})^\dis$ are discrete so that 
%$\sC_\unif(D,k) \wt_k^\un \sC_\unif(E,k) = \sC_\unif(D,k) \otimes_k \sC_\unif(E,k)$. But if, say, 
% $D = E$, the characteristic  function of the diagonal does not belong to the image of $\cL^{(D,E)}_\unif$. 
%\end{rmk}
\end{subsection} 
\begin{subsection}{Functions with almost compact support $\sC_\acs(X,k) \in  \cCLMcan_k$}
\label{acs} 
\begin{lemma} \label{compunif} Let $(X,\Theta_\sJ)$ be a $td$-uniform space. For any compact open subset  $C \subset X$  the continuous $k$-linear map
$$
i_{C , X} : \sC(C,k) \longrightarrow  \sC_\unif((X,\Theta_\sJ),k)  \;,
$$
sending $f \in \sC(C,k)$ to its extension by $0$ to $X$ is an isomorphism onto a 
direct summand of $\sC_\unif((X,\Theta_\sJ),k)$ in the category $\cCLMcan_k$. The complement of $\sC(C,k)$ in 
$\sC_\unif((X,\Theta_\sJ),k)$ is $\sC_\unif((X-C,\Theta_\sH),k)$ where a basis of uniform covers of $(X-C,\Theta_\sH)$ consists of those $td$-partitions of $X-C$ which, together with a $td$-partition of the compact open $C \in \cK(X)$, form a basis of uniform covers of $(X,\Theta_\sJ)$. 
%closed $k$-submodule of 
%$\sC_\unif((X,\Theta_\sJ),k)$ equipped with the relative topology. 
%In particular it is a strict monomorphism of $\cCLMcan_k$. 
%strict monomorphism of $\cCLMcan_k$.  
%identifies $\sC(C,k)$  to a strict $\cCLMcan_k$-subobject of $\sC_\unif((X,\Theta_\sJ),k)$.  
\end{lemma}
\begin{proof} The partition $X = C \,\dot\cup \,(X - C)$  is a $td$-partition of $X$ which we identify  to the discrete quotient 
$\pi_D : X \to D = \{C, X-C\}$. We may replace $\sJ$ by its cofinal subsystem $\sJ'$ of quotients $\pi_{D'}: X \to D'$ which factor 
$\pi_D$ as $\pi_{D,D'} \circ \pi_{D'}$. Any such $\pi_{D'} : X \to D' \in \sJ'$ decomposes into the sum of 
$(\pi_{D'})_{|C} =: \pi_1: C \to D_1 \in \sJ_1 \subset \sQ(C)$ and 
$(\pi_{D'})_{|X-C} =: \pi_2: X-C \to D_2 \in \sJ_2 \subset \sQ(X-C)$. 
Then $(X,\Theta_\sJ) = (X,\Theta_{\sJ'})$ is the sum of the 2 uniform spaces 
$$ (X,\Theta_{\sJ'}) = (C,\Theta_{\sJ_1}) \coprod (X-C, \Theta_{\sJ_2}) \;. $$
 We conclude that 
\beq 
 \sC_\unif((X,\Theta_\sJ),k) =  \sC_\unif((C,\Theta_{\sJ_1}),k)  \bigoplus  \sC_\unif((X-C,\Theta_{\sJ_2}),k) \;.
\eeq
\end{proof}
\begin{rmk} \label{compunif2} In this particular case, the subspace topology of 
$\sC(C,k) \subset \sC_\unif((X,\Theta_\sJ),k)$ is the topology of uniform convergence on $C$, so that it is the naive canonical topology of $\sC(C,k)$. 
%the last assertion is a trivial consequence of the previous part of the lemma. 
But in general this is not the case. Namely, if $M \in \cCLMcan_k$, \ie if $M$ is complete in its naive canonical topology, a closed $k$-linear subspace $N \subset M$ does not necessarily carry its naive $k$-linear topology, but a weaker one. On the other hand, as recalled in Remark~\ref{8.3.3}, $N$ is complete in its naive $k$-canonical topology, hence the morphism $N^\can \to N$ is bijective. \emph{So, in general, $N^\can \to M$, but not $N \to M$, is a strict monomorphism in $\cCLMcan_k$}.  
%Now, assume $P \subset M$ is a second closed $k$-linear subspace of $M$ and that $M = N \oplus P$ in $\cCLMou_k$. (This is the case if, for 
% the closed embedding $j:N \to M$,  there is a $\cCLMou_k$-morphism $\pi: M \to N$ such that $\pi \circ j = \id_N$: take $P = \Ker \, \pi$.)
%Then, $M^\can = M$ implies that both $N^\can \to N$ and $P^\can \to P$ are isomorphisms of $\cCLMcan_k$, so that $N \to M$  is a strict monomorphism in $\cCLMcan_k$. 
%This is what happens for $M = \sC_\unif((X,\Theta_\sJ),k)$ and $N=\sC(C,k)$, since $F \longmapsto f_{|C}$ is a left-inverse to $i_{C , X}$. 
\end{rmk}

\begin{defn} \label{acsfcts} We  define the space of continuous functions $X \to k$ with \emph{almost compact support}, as 
\beq  \label{acsfcts1}
\sC_\acs(X,k) := \colimit^\un_{C \subset C' } (\sC (C,k), i_{C , C'})  \in \cCLMcan_k
\;,
\eeq 
where $C,C'$ vary over the family of compact open subsets of $X$ and, for $f \in \sC (C,k)$,  $i_{C , C'}(f) \in \sC (C',k)$ is the extension of $f$ by $0$ to $C'$. We will  denote by 
$$i_C: \sC (C,k) \to \sC_\acs(X,k)$$ 
the canonical morphism called \emph{extension by $0$} (from $C$) to $X$. 
\end{defn}
%\bmau The natural morphism
%\beq \label{canacs1}
%i_\acs: \sC_\acs(X,k) \longrightarrow \sC(X,k)
%\eeq
% is injective. \emau
 \begin{rmk} \label{colimacs} For any $td$-space $X$  let $\sC_\cs(X,k)$ be the $k$-submodule of $\sC(X,k)$ consisting of   functions  with compact support. Then
 $\sC_\acs(X,k)$ is the  completion of  $\sC_\cs(X,k)$ for the $k$-canonical topology, \ie for the  topology of uniform convergence on $X$. 
% Notice that it follows from \eqref{contcompl01} and the definition that $\sC_\acs(X,k)$ is the colimit in $\cCLMcan_k$ of a filtered  inductive system of finite free $k$-modules equipped with the naive canonical topology. Namely these modules are parametrized by finite disjoint families $\{C_1,\dots,C_n\}$ of compact open subsets of $X$ and, for any such family, the corresponding $k$-module consists of the functions $C_1\cup\dots\cup C_n \longrightarrow k$, constant on each $C_i$. 
\end{rmk}
\begin{rmk} \label{limdiscr2}   For any $I \in \cP(k)$,  $\sC_\acs(X,k/I)$ is discrete and, by definition of $\colimit^\un$, 
$$
\sC_\acs(X,k) =  \limit_{I \in \cP(k)} \sC_\cs(X,k/I) \; .
$$
\end{rmk}  
 \begin{cor} \label{acsdiscr} For $D$ discrete, 
  \beq \label{acsdiscr0}
\sC_\cs(D,k) = \colimit_{F \in \cF(D) } (\sC (F,k), i_{F , F'}) =\bigoplus_{x \in D} k \chi_x \in \Mod_k
\;,
\eeq 
 \beq \label{acsdiscr1}
\sC_\acs(D,k) = \colimit^\un_{F \in \cF(D) } (\sC (F,k), i_{F , F'}) =\bigoplus^\un_{x \in D} k \chi_x \in \cCLMcan_k
\;.
\eeq 
 \end{cor}
\begin{cor} \label{acscor}  For any uniform structure $\Theta_\sJ$ on $(X,\tau)$,  $\sC_\acs(X,k)$ is a subspace of  
$\sC_\unif((X,\Theta_\sJ),k)$. 
%and $\sC_\cs(X,k)$ is the closure of 
% $\sC_\cs(X,k)$ in $\sC_\unif((X,\Theta_\sJ),k)$.   
% Application of the functor $\colimit^\un_{C \subset X}$ to the system of morphisms 
%$\{\iota_{C,X}\}_C$ of Lemma~\ref{compunif}  identifies $\sC_\acs(X,k)$ with the closure of $\sC_\cs(X,k)$ in $\sC_\unif((X,\Theta_\sJ),k)$ equipped with the subspace topology. In particular,
The natural morphism 
\beq \label{canacs2}
i_\acs: \sC_\acs(X,k) \longrightarrow \sC_\unif((X,\Theta_\sJ),k)   
\eeq  
 is a strict monomorphism  in the quasi-abelian category $\cCLMcan_k$.  
There is  a sequence of  injective $\cCLMu_k$-morphisms
\beq \label{canacs3}
\sC_\acs(X,k) \map{i_\acs} \sC_\unif((X,\Theta_\sJ),k) \longrightarrow  \sC(X,k)
\eeq
 such that the natural composite morphism is a dense  injection.  
\end{cor}
\begin{proof} The morphism $i_\acs$ is obtained as follows. One first considers the colimit 
$$ 
  \colimit_{C \subset C' } (\sC (C,k), i_{C , C'}) 
\;,
$$ 
in $\Mod_k$ which is the $k$-submodule of $\sC_\cs (X,k)$ of $\sC_\unif((X,\Theta_\sJ),k)^\for$. 
The image of  $i_\acs$ in $\cCLMcan_k$ coincides set-theoretically with  the closure of the set-theoretic image of 
$\sC_\cs (X,k)$ in $\sC_\unif((X,\Theta_\sJ),k)$, but it is equipped with the   $k$-canonical topology which is finer than the subspace topology induced by $\sC_\unif((X,\Theta_\sJ),k)$. So,  
$i_\acs: \sC_\acs(X,k) \map{} \sC_\unif((X,\Theta_\sJ),k)$ is a strict monomorphism in $\cCLMcan_k$. 
 This proves the first part of the Corollary. 
\par \medskip
Taking into account \eqref{uniftocont}, we are only left to prove that the composite morphism $\varphi:\sC_\acs(X,k) \longrightarrow \sC(X,k)$ in \eqref{canacs3}
 has dense image.   
For any compact open $C \subset X$, we have a morphism 
\beq
\begin{split}
\pi_C: \sC_\acs(X,k) &\longrightarrow \sC(C,k) \\
f &\longmapsto \varphi(f)_{|C}  
\end{split}
\eeq
and the morphism $\varphi$ is the limit of the projective system $\{\pi_C\}_C$. 
Now, for any $C$, the map $\pi_C$ is surjective because it admits the canonical  right inverse $\sC(C,k) \map{i_C}\sC_\acs(X,k)$, so that $\varphi$  has dense image. 
\end{proof}
%\begin{rmk} \label{acscompl} $\sC_\acs(X,k)$ has a complement in $\sC_\unif((X,\Theta_\sJ),k)$, namely a strict subobject $T$ such that 
%$$\sC_\unif((X,\Theta_\sJ),k) = \sC_\acs(X,k) \bigoplus T\;.$$
%For any $D \in \sJ$, let 
%$$
%T_D := \bigcap_{F \in \cF(D)} \sC_\unif(D - F,k) 
%$$
%Then 
%$$
%T = \bigcap \pi^\ast_D(T_D) \;.
%$$
%\end{rmk}
\begin{rmk} \label{explsupport} If $k$ is discrete  the $k$-module underlying $\sC_\acs(X,k)$ identifies with  $\sC_\cs(X,k)$
%the $k$-submodule of $\sC(X,k)$ consisting of locally constant compactly supported functions $X \to k$ 
equipped with the discrete topology.
 For general $k$, the $k$-module underlying $\sC_\acs(X,k)$ is the submodule of $\sC(X,k)$ consisting of uniform limits on $X$   of compactly supported continuous functions $X \to k$. It may be described as the
 $k$-submodule of $\sC(X,k)$ of continuous functions $f : X \to k$ which satisfy the   following property (where {\bf (ACS)} stands for ``almost compact support'')
\par
{\bf(ACS)} \emph{For any $I \in \cP(k)$, there exists a compact subset $Z_I \subset X$ such that $f(x) \in I$, for any $x \in X - Z_I$\;.} 
\par \noindent 
Of course, in {\bf(ACS)} we might as well have imposed that $Z_I$ be a compact \emph{open} subset of $X$. 
%It follows from Lemma~\ref{compunif} that, for any uniform structure $\Theta_\sJ$ on $X$,
%$\sC_\acs(X,k)$ identifies  with a closed $k$-submodule of $\sC_\unif((X,\Theta_\sJ),k)$ with the subspace topology, hence, in particular, is strictly closed.
\end{rmk}  
\begin{prop}\label{proacsfcts} For any $td$-space $X$ we have
\beq \label{proacsfcts1}
\sC_\acs(X,k) = \colimit^\un_{D \in \sQ(X)} \sC_\acs(D,k)\;.
\eeq
\end{prop} 
\begin{proof} We have by Remark~\ref{limdiscr2}
$$
\sC_\acs(X,k) =  \limit_{I \in \cP(k)} \sC_\cs(X,k/I) \; .
$$
On the other hand, by \eqref{prodis},  for any $I \in \cP(k)$ we have in $\Mod_{k/I}$  
$$
\sC_\cs(X,k/I) = \colimit_{D \in \sQ(X)} \sC_\cs(D,k/I)\;.
$$
So, 
$$
\sC_\acs(X,k) =  \limit_{I \in \cP(k)} \sC_\cs(X,k/I) =  \limit_{I \in \cP(k)} \colimit_{D \in \sQ(X)} \sC_\cs(D,k/I) = \colimit^\un_{D \in \sQ(X)} 
\sC_\acs(D,k)\; .
$$
\end{proof}
\end{subsection}
\begin{subsection}{Summary of results on rings of continuous functions}
 The first two corollaries have already been proven and the third is obvious.
\begin{cor} \label{acsalc}
If $X = C$ is a Stone space then all
morphisms in \eqref{canacs3} are isomorphisms and (see Lemma~\ref{contcompl})
\beq
\label{acsalc2} \sC(C,k) = \colimit^\un_{F \in \sQ(C)} \sC(F,k)  = \colimit^\un_{F \in \sQ(C)} k^{(F,\un)} \in \cCLMcan_k
\eeq
where $F$ varies over finite discrete quotients of $C$. 
%\bmau Strong dual
%$$
% \sD_\strong(C,k) = \limit^{\square,\un}_{F \in \sQ(C)} k^{(F,\un)} \in \cCLMcan_k \subset \colimit^\un \cF_k
%$$
%Weak dual
%$$
% \sD_\weak(C,k) = \limit_{F \in \sQ(C)} k^{(F,\un)} \in \limit\, \cF_k
%$$
%\emau
 If $X = D$ is discrete   and the uniformity on $D$ is understood to be the discrete uniformity,    \eqref{canacs3} becomes
\beq
\label{acsalc3dis}   
\sC_\acs(D,k) = {\bigoplus}^\un_{x \in D} k \chi_x   \longrightarrow \sC_\unif(D,k) =  {\prod}_{x \in D}^\can  k \chi_x  \longrightarrow 
\sC(D,k) = {\prod}_{x \in D} k \chi_x \;.
\eeq
Notice  that the first morphism is a strict monomorphism in the quasi-abelian category  $\cCLMcan_k$, while the second is bijective. The composite morphism has dense image.  
%\bmau Strong dual of $\sC(D,k)$ is
%$$
%\sD_\acs(D,k) := {\bigoplus}^\un_{x \in D} k \delta_x 
%$$
%Strong dual of $\sC_\unif(D,k) = {\prod}_{x \in D}^\can  k \chi_x$ will not be used.
%Strong dual of $\sC_\acs(D,k) = {\bigoplus}^\un_{x \in D} k \chi_x $ is
%$$
%\sD_\can (D,k) = {\prod}^\can_{x \in D} k \delta_x 
%$$
%Weak dual of $\sC(D,k)$ is still $\sD_\acs(D,k)$. Weak dual of $\sC_\unif(D,k)$ will not be used.
%Weak dual of $\sC_\acs(D,k)$ is 
%$$
%\sD_\pro (D,k) = {\prod}_{x \in D} k \delta_x 
%$$
%Strong dual of \eqref{canacs3} for $X=D$ discrete \beq
%\label{acsalc3sdual}   
%{\bigoplus}^\un_{x \in D} k \chi_x \iso  {\bigoplus}^\un_{x \in D} k \chi_x  \longrightarrow  ({\prod}_{x \in D}^\can  k \chi_x )'_\strong \longrightarrow 
%{\prod}^\can_{x \in D} k \delta_x  
%\eeq
%\beq
%\label{acsalc3sdualbis}   
%\sD_\acs(D,k) \iso  \sD_\acs(D,k)  \longrightarrow  ({\prod}_{x \in D}^\can  k \chi_x )'_\strong \longrightarrow 
%\sD_\can(D,k) 
%\eeq
%Weak dual of \eqref{canacs3} for $X=D$ discrete 
%\beq
%\label{acsalc3wdual}   
%{\bigoplus}^\un_{x \in D} k \chi_x \iso  {\bigoplus}^\un_{x \in D} k \chi_x  \longrightarrow  ({\prod}_{x \in D}^\can  k \chi_x )'_\strong \longrightarrow 
%{\prod}_{x \in D} k \delta_x  
%\eeq
%\beq
%\label{acsalc3wdualbis}   
%\sD_\acs(D,k) \iso  \sD_\acs(D,k)  \longrightarrow  ({\prod}_{x \in D}^\can  k \chi_x )'_\weak \longrightarrow 
%\sD_\pro(D,k) 
%\eeq
%\emau
\end{cor}
We summarize our results on $\sC(X,k)$ and $\sC_\unif(X,k)$.
\begin{cor} \label{contcase} Let  $X$ be a $td$-space. Then 
\ben  
\item (see \eqref{complim2})
$$\sC(X,k) = \limit_{C \in \cK(X)} \sC(C,k) = \limit_{C \in \cK(X)}  \colimit^\un_{F \in \sQ(C)} k^{(F,\un)} \in \limit \, \cCLMcan_k \;.$$
\item (see \eqref{contdiscr2} and  \eqref{limprocomp3})
$$
\sC(X,k) = \colimit^\un_{D \in \sQ(X)}  \sC(D,k)  = \colimit^\un_{D \in \sQ(X)}  {\prod}_{x \in D} k \chi^{(D)}_x 
\;.
$$
\item (Definition~\eqref{acsfcts})
$$
\sC_\acs(X,k) = \colimit^\un_{C \in \cK(X)}  \sC (C,k)  \in \cCLMcan_k
\;.
$$
\item (see \eqref{proacsfcts1})
$$
\sC_\acs(X,k) = \colimit^\un_{D \in \sQ(X)} \sC_\acs(D,k) \in \cCLMcan_k
\;.
$$
\een
%\bmau Strong dual:
%$$
%\sD_\strong(X,k) := \colimit^\un_{C \in \cK(X)} \sD_\strong(C,k) = \colimit^\un_{C \in \cK(X)} \limit^{\square,\un}_{F \in \sQ(C)} k^{F}
%\in \cCLMcan_k$$
%Weak dual:
%$$
%\sD_\weak(X,k) :=\colimit^\un_{C \in \cK(X)} \sD_\weak(C,k) = \colimit^\un_{C \in \cK(X)} \limit_{F \in \sQ(C)} k^{F}
%$$
%\emau
%\bmau Strong dual:
%$$
%\sD_{\alc\strong} (X,k) := \limit^{\square,\un}_{D \in \sQ(X)}  {\bigoplus}^\un_{x \in D} k \chi^{(D)}_x
%$$
%Weak dual:
%$$
%\sD_{\alc\weak} (X,k) := \limit_{D \in \sQ(X)}  {\bigoplus}^\un_{x \in D} k \chi^{(D)}_x
%$$
%\emau
%\bmau 
%Strong dual:
%$$
%\sD_{\acs\strong}(X,k) := \limit^{\square,\un}_C   \limit^{\square,\un}_{F \in \sQ(C)} k^F 
%$$
%Weak dual:
%$$
%\sD_{\acs\weak}(X,k) := \limit_C   \limit_{F \in \sQ(C)} k^F 
%$$
%\emau
\end{cor}
\begin{cor} \label{decomp} Assume  $X$ admits a  decomposition as a sum   of Stone subspaces  $\{X_\alpha\}_{\alpha \in A}$ as in \eqref{sumstone}. Then 
$$
\sC(X,k)  = {\prod}_{\alpha \in A}   \sC(X_\alpha,k)  \;.$$
%= 
%{\prod}_{\alpha \in A} \colimit^\un_{F_\alpha \in \sQ(X_\alpha)} \sC(F_\alpha,k)
%=
%$$
%$$ {\prod}_{\alpha \in A} \colimit^\un_{F_\alpha \in \sQ(X_\alpha)}   k^{(F_\alpha,\un)} 
%\;.
%$$
\end{cor}
\begin{rmk} \label{Xmetrtd} If $X$ is a metrizable $td$-space, then 
$$\sC(X,k) \in \limit_\N \, \cCLMcan_k  \subset \cCLMou_k\;,
$$
is a sequential limit of $k$-canonical modules. 
\end{rmk}
\begin{cor} \label{unifcase}  We now assume that  $X$  as in Corollary~\ref{contcase} is also given a $td$-uniform structure $\Theta_\sJ$ as in Definition~\ref{unifcontdef}.  Then
\ben  
\item 
$$
\sC_\unif((X,\Theta_\sJ),k)  = \limit_{I \in \cP(k)} \sC_\unif((X,\Theta_\sJ),k/I)  \in \cCLMcan_k\; .
$$
%\item
%$$
%\sC_\unif((X,\Theta_\sJ),k)  = {\prod}_{\alpha \in A}^\can  \sC(X_\alpha,k)  = 
%{\prod}_{\alpha \in A}^\can \colimit^\un_{D_\alpha \in \sQ(X_\alpha)} \sC(D_\alpha,k)
%\;.
%$$
%$$
%\sC_\unif((X,\Theta_\sJ),k)  =  \limit^\can_{F \in \sF(A)}  \sC(\bigcup_{\alpha \in F} X_\alpha,k)  = 
%\limit^\can_{F \in \sF(A)} \colimit^\un_{D \in \sQ(\bigcup_{\alpha \in F} X_\alpha)} \sC_\unif(D,k)
%\;.
%$$
\item 
$$
\sC_\unif((X,\Theta_\sJ),k)  = \colimit^\un_{D \in \sJ} \sC_\unif (D,k) 
= \colimit^\un_{D \in \sJ} {\prod}_{x \in D}^\can k \chi_x^{(D)} \;.
%{\limPRO}^{\square,u}_{F \in \sF(D)}\sC(F,k)  
$$
%\bmau Strong and weak duals of $\sC_\unif((X,\Theta_\sJ),k)$ are not good because there is the dual of ${\prod}_{x \in D}^\can k \chi_x^{(D)}$ that we have avoided! So we need an independent definition of $\sD_\unif((X,\Theta_\sJ),k)$. We expect sequences
%\beq \label{canacsdualstrX3}
%\sD_\strong(X,k) \longrightarrow \sD_{\alc\strong} (X,k)  \longrightarrow \sD_{\unif} ((X,\Theta_\sJ),k) \longrightarrow
%\sD_{\acs\strong}(X,k) 
%\eeq
%\beq \label{canacsdualweakX3}
%\sD_\weak(X,k) \longrightarrow \sD_{\alc\weak} (X,k)  \longrightarrow \sD_{\unif} ((X,\Theta_\sJ),k) \longrightarrow
%\sD_{\acs\weak}(X,k) 
%\eeq
%\emau
\item
The natural morphism 
$$\sC_\unif((X,\Theta_\sJ),k) \longrightarrow \sC(X,k)
$$
is injective and has dense image. 
%\bmau Natural dual:
%$$
%\colimit^\un_{C \in \cK(X)} \limit_{F \in \sQ(C)} k^{F} \longrightarrow \limit_{D \in \sJ} {\bigoplus}_{x \in D}^\un k \chi_x^{(D)} 
%$$
%\emau
\een
\end{cor} 
\begin{proof}
$\mathit 1.$ This follows from the definition. Namely, $f:X \to k$ is uniformly continuous for the uniformity $\Theta_\sJ$  if and only if for any $I \in \cP(k)$ 
there exists a discrete quotient $D_I \in \sJ$ of 
$X$ such that $f_I := f \mod I : X \to k/I$ factors through the projection $\pi_{D_I}:X \to D_I$.  Conversely, given a coherent system $(f_I)_{I \in \cP(k)}$ of 
functions $f_I  : X \to k/I$ such that $f_I$ factors through  $\pi_{D_I}:X \to D_I$ and that, if $I \subset J$, $f_J = f_I \mod J$, 
 the sequence  $(f_I)_{I \in \cP(k)}$ converges uniformly on $X$ to  a uniformly continuous function $f \in \sC_\unif((X,\Theta_\sJ),k)$. 
It is clear that the topology of uniform convergence on $X$ on $\sC_\unif((X,\Theta_\sJ),k)$ is precisely the weak topology of the projections 
$\sC_\unif(X,k) \to \sC_\unif(X,k/I)$, $f \mapsto f \mod I$, to the discrete $k/I$-module $\sC_\unif(X,k/I)$, for $I \in \cP(k)$. 
\par
%$\mathit 2.$ 
%We have 
%$$
%\sC_\unif(X,k)  = \limit^\un_{I \in \cP(k)} \sC_\unif(X,k/I)  = \limit^\un_{I \in \cP(k)} ({\prod}_{\alpha \in A} \sC(X_\alpha,k/I))^\dis = 
%$$
%$$\limit^\un_{I \in \cP(k)} {\prod}_{\alpha \in A}^\can \sC(X_\alpha,k/I) =  {\prod}_{\alpha \in A}^\can \sC(X_\alpha,k) = 
%{\prod}_{\alpha \in A}^\can \colimit^\un_{D_\alpha \in \sQ(X_\alpha)} \sC(D_\alpha,k)
%\;.
%$$
%\par
$\mathit 2.$ We consider the first equality. 
Notice that, if   $k$ is discrete, 
 $\sC_\unif (X,k)$ consists of the  $k$-module of maps $f : X \to k$ 
 which are constant on the fibers of some projection $\pi_D : X \to D$, for $D \in \sJ$, 
 equipped with the discrete topology. So, in that case
 $$\sC_\unif ((X,\Theta_\sJ),k) 
 %= \bigcup_{D \in \sJ} k^D  
 = \colimit^\un_{D \in \sJ} \sC_\unif (D,k)
 $$ 
 and the result holds true. Then
$$
\sC_\unif((X,\Theta_\sJ),k)  =   \limit_{I \in \cP(k)} \sC_\unif((X,\Theta_\sJ),k/I) =
$$
$$
\limit_{I \in \cP(k)}  \colimit^\un_{D \in \sJ} \sC_\unif (D,k/I) =
$$
$$
 \colimit^\un_{D \in \sJ} \sC_\unif (D,k) \;.
$$
For the second equality, we just replace $\sC_\unif (D,k)$ with the central term of \eqref{acsalc3dis}  
%$$
%\sC_\unif(X,k)  =   \limit_{I \in \cP(k)} \sC_\unif(X,k/I) =
%$$
%$$
%= \limit_{I \in \cP(k)} \colimit^\un_{D \in \sJ} \left({\prod}_{x \in D} (k/I) \chi_x^{(D)}\right)^\dis =
%$$
%$$
% \colimit^\un_{D \in \sJ} {\prod}_{x \in D}^\can k \chi_x^{(D)} \;.
%$$
\par $\mathit 3.$ Has been proven already. Injectivity follows from the definition \eqref{uniftocont} and density follows from   the last assertion of Corollary~\ref{acscor}. 
%Alternatively:
%by \eqref{discrlim} and $\mathit 1$, it suffices to prove the statements when $k$ is discrete. Then  
%injectivity is clear.
% To show that the map has dense image, we observe that it is enough to show that for any compact open subset $Z$ of $X$, any continuous function $g:Z \to k$ extends to a uniformly continuous function on $X$. But in fact $g$ is locally constant on $Z$, and since $Z$ is compact and the fibers of projections $\pi_D:X \to D$ for $D \in \sJ$ are a basis for the topology of $X$, then 
%there is some $D \in \sJ$ and a finite subset $F \subset D$ such that $Z = \pi_D^{-1}(F)$, and $g$ is constant on the fibers $\pi_D^{-1}(x)$, for $x \in F$. Then it is clear that the extension of $g$ by 0 on $X-Z$ is uniformly continuous. 
\end{proof}
\begin{cor} \label{unifcasedec}  Let  $(X,\Theta_\sJ)$ be  as in Corollary~\ref{unifcase}  and let   $\{X_\alpha\}_{\alpha \in A}$  be a partition of $X$ associated to a projection $\pi_A :X \to A$ in $\sJ$.   Then
$$
\sC_\unif((X,\Theta_\sJ),k)  = {\prod}_{\alpha \in A}^\can  \sC(X_\alpha,k)  
$$
%= 
%{\prod}_{\alpha \in A}^\can \colimit^\un_{D_\alpha \in \sQ(X_\alpha)} \sC(D_\alpha,k)
%\;.
%$$
%$$
%\sC_\unif((X,\Theta_\sJ),k)  =  \limit^\can_{F \in \sF(A)}  \sC(\bigcup_{\alpha \in F} X_\alpha,k)  = 
%\limit^\can_{F \in \sF(A)} \colimit^\un_{D \in \sQ(\bigcup_{\alpha \in F} X_\alpha)} \sC_\unif(D,k)
%\;.
%$$
\end{cor} 
\begin{proof}
We have 
$$
\sC_\unif(X,k)  = \limit_{I \in \cP(k)} \sC_\unif(X,k/I)  = \limit_{I \in \cP(k)} ({\prod}_{\alpha \in A} \sC(X_\alpha,k/I))^\dis = 
 {\prod}_{\alpha \in A}^\can \sC(X_\alpha,k) \;.
$$
%$$\limit^\un_{I \in \cP(k)} {\prod}_{\alpha \in A}^\can \sC(X_\alpha,k/I) =  {\prod}_{\alpha \in A}^\can \sC(X_\alpha,k) = 
%{\prod}_{\alpha \in A}^\can \colimit^\un_{D_\alpha \in \sQ(X_\alpha)} \sC(D_\alpha,k)
%\;.
%$$
\end{proof}
\end{subsection}
\begin{subsection}{Tensor products of rings of continuous functions} \label{tenscontfcts}
%\begin{cor} \label{unimeasdisc} For $D,E$ discrete spaces, the map 
%\beq  \label{unimeasdisc1} f \wt^\un_k g \longmapsto ((x,y)  \longmapsto f(x)g(y))
%\eeq
%induces a natural injective morphism in $\cCLMcan_k$  
%\beq  \label{unimeasdisc2} \cL^{(D,E)}_\unif: \cC_\unif(D,k) \wt^\un_k \cC_\unif(E,k) \map{}  \sC_\unif(D \times E,k) \;.
%\eeq
%\end{cor}
%\begin{proof} 
%\end{proof}

\begin{prop} \label{contvsunif} Notation as in Definition~\ref{unifcontdef}.
% and Theorem~\ref{STSspaces}.  Then 
For any pair of $td$-spaces $X,Y$,  we have a morphism in  the category $\cLMu_k$, 
\beq \label{prodfcts0} \begin{split}
\cL^{(X,Y)}: \sC(X,k) \otimes_k \sC(Y,k) &\longrightarrow \sC(X \times Y,k) \\ f \otimes_k g  &\longmapsto \left( (x,y) \mapsto f(x) g(y) \right) \;.
\end{split}\eeq
\ben
\item For any pair of Stone spaces $C,C'$, the map \eqref{prodfcts0} for $X=C$ and $Y =C'$, 
extends to an  isomorphism
\beq \label{prodfcts1}
\cL^{(C,C')}: \sC(C,k) \wt_k^\un \sC(C',k) \iso \sC(C \times C',k)  
\eeq  
of ring-objects of $\cCLMcan_k$.
\item For any pair of $td$-spaces $X,Y$,  the isomorphisms $\cL^{(C,C')}$ 
of \eqref{prodfcts1} 
for compact opens $C$ of $X$  and $C'$ of $Y$, respectively, 
determine an isomorphism
\beq \label{prodfctsacs}
 \cL^{(X,Y)}_\acs: \sC_\acs(X,k) \wt_k^\un \sC_\acs(Y,k) \iso \sC_\acs(X \times Y,k)  
\eeq 
of (non-unital) ring-objects of $\cCLMcan_k$.
\item For any pair of $td$-spaces $X,Y$, the map \eqref{prodfcts0}
extends to an isomorphism of ring-objects of  $\cCLMu_k$
\beq \label{prodfctsmor}
\cL^{(X,Y)}: \sC(X,k) \wt_k^\un \sC(Y,k) \iso \sC(X \times Y,k) \;.
\eeq   
\item For any pair $D,E$ of discrete spaces, equipped with their discrete uniformity, the map \eqref{prodfcts0} 
induces a natural injective morphism in $\cCLMcan_k$  
\beq  \label{unimeasdisc2} \cL^{(D,E)}_\unif: \cC_\unif(D,k) \wt^\un_k \cC_\unif(E,k) \map{}  \sC_\unif(D \times E,k) \;.
\eeq
\item For any pair of uniform $td$-spaces $(X,\Theta_\sJ)$, $(Y, \Theta_\sH)$, let $\sJ \times \sH$ be the directed set of 
$$\pi_{D \times E} = \pi_D \times \pi_E : X \times Y \to D \times E \;,
$$
for $\pi_D \in \sJ$ and $\pi_E \in \sH$. Then 
%the isomorphisms 
%\beq \label{prodfcts4}  
%\cL_\cont^{(D,E)}: \sC(D,k) \wt_k^\un \sC(E,k)  \iso \sC(D \times E,k) \;,
%\eeq
%for $\pi_D:X \longmapsto D$ in $\sJ$ and  $\pi_E:Y \longmapsto E$ in $\sH$, 
%determine an isomorphism
the  natural  injective morphisms \eqref{unimeasdisc2}  
%of Corollary~\ref{unimeasdisc}
%\beq \label{prodfcts5} \begin{split}
%\cL_\unif^{(D,E)}: \sC_\unif(D,k) \otimes^\un_k \sC_\unif(E,k) & \longrightarrow 
%\sC_\unif(D \times E ,k)  \\
%f \wt^\un_k g  &\longmapsto \left( (x,y) \mapsto f(x) g(y) \right)
%\end{split}
%\eeq
induce an injective morphism  
\beq \label{prodfcts11}
 \cL_\unif^{(X,Y)}: \sC_\unif((X,\Theta_\sJ),k) \wt_k^\un \sC_\unif((Y, \Theta_\sH),k)  \longrightarrow 
\sC_\unif((X \times Y, \Theta_{\sJ \times \sH}),k) 
\eeq 
of ring-objects of $\cCLMcan_k$.
%\item For any pair of $td$-spaces $X,Y$ and any $D \in \sQ(X)$, $E \in \sQ(Y)$, the isomorphisms $\cL^{(D,E)}$  of \eqref{prodfcts1} induce a continuous bilinear map
%\beq \label{prodfctsalc0}  \sC_\alc(X,k) \times \sC_\alc(Y,k) \longrightarrow  \sC_\alc(X \times Y,k)
%\eeq
%and therefore a morphism
%\beq \label{prodfctsalc1}
% \cL^{(X,Y)}_\alc: \sC_\alc(X,k) \wt_k^\un \sC_\alc(Y,k) \longrightarrow   \sC_\alc(X \times Y,k) \;.
%\eeq  
\een
\end{prop}
\begin{proof} 
 \hfill \break \indent
\par
$\mathit 1.$ 
Let $C$ and $C'$ be Stone spaces. Then, $\sC(C,k)$ (resp. $\sC(C',k)$, resp. $\sC(C \times C',k)$) is a colimit in $\cCLMcan_k$ of  free $k$-modules of finite type endowed with the naive canonical topology 
of the form $\sC(F,k)$ (resp. $\sC(G,k)$, resp. $\sC(F \times G,k)$), where $F$ (resp. $G$) is a finite quotient of $C$ (resp. $C'$).  Since, by Proposition~\ref{tensind}, $\wt^\un_k$ commutes with colimits in $\cCLMcan_k$, we are reduced to the case $C=F$, $C'=G$ finite discrete spaces. In this case the map 
$$
\cL: \sC(F,k) \otimes_k \sC(G,k) \longrightarrow \sC(F \times G,k)
$$
of \eqref{prodfcts0} 
 is obviously an isomorphism of ring-objects of $\cF_k$, so we are done. 
\par $\mathit 2.$ Follows from $\mathit 1$ taking $\colimit^\un_{C,C'}$, where $C,C'$ are compact opens in $X,Y$, respectively,  and using Proposition~\ref{tensind}. 
\par $\mathit 3.$ We obtain the morphism $\cL^{(X,Y)}$ of \eqref{prodfctsmor} as follows. 
We have a natural morphism 
\beqa \begin{split}
\sC(X,k) & \wt_k^\un \sC(Y,k) =  \limit_{C \in \cK(X)}   \sC(C,k)  \wt_k^\un  \limit_{C' \in \cK(Y)}   \sC(C',k)  \map{\eqref{tensfunct2}} \\
&\limit_{C \in \cK(X), C' \in \cK(Y)}   \sC(C,k)  \wt_k^\un    \sC(C',k) \;.
\end{split}
\eeqa   
Since   $$\sC(C,k)  \wt_k^\un    \sC(C',k) = \sC(C \times C',k)$$ 
by $\mathit 1$,  the previous formula continues into 
$$= \limit_{C \times C' \in \cK(X \times Y)}   \sC(C \times C',k) = \sC(X \times Y,k) \;,$$
and the composition is $\cL^{(X,Y)}$ of \eqref{prodfctsmor}.  
 The morphism  $\cL^{(X,Y)}$ is an isomorphism because it is the completion of the isomorphism \eqref{prodfctsacs} in the topology of $ \sC_\acs(X,k) \wt_k^\un \sC_\acs(Y,k)$ (resp. of $\sC_\acs(X \times Y,k)$)
as a subspace of $\sC(X,k)\wt_k^\un \sC(Y,k)$ (resp. $\sC(X \times Y,k)$).
 \par $\mathit 4.$ From \eqref{unifdiscr2} we see that  
 for $D$ discrete  $\sC_\unif(D,k)$ is simply $k^D$ equipped with the naive canonical 
topology; so the existence of the morphism \eqref{unimeasdisc2} and the fact that it is injective is clear.  
 \par $\mathit 5.$ 
We obtain the morphism in \eqref{prodfcts11} by taking colimits of the filtered inductive system of injective morphisms $\{\cL_\unif^{(D,E)}\}_{\pi_D \in \sJ,\pi_E \in \sH}$ \eqref{unimeasdisc2}  in $\cCLMcan_k$,  and using Proposition~\ref{tensind}.   
To show that the map $\cL_\unif^{(X,Y)}$ of \eqref{prodfcts11} is injective we use Corollary~\ref{unifcase} by which
$$
\sC_\unif((X,\Theta_\sJ),k)  = \limit_{I \in \cP(k)} \sC_\unif((X,\Theta_\sJ),k/I)  =  
\limit_{I \in \cP(k)}  \colimit^\un_{D \in \sJ} \sC_\unif (D,k/I) \;,
$$ 
where $\colimit_{D \in \sJ}^\un$ equals set-theoretically  $\colimit_{D \in \sJ}$ in $\Mod_{k/I}$.
By \eqref{unimeasdisc2}, for any $I \in \cP(k)$, we get injective morphisms in $\Mod_{k/I}$
$$\cL^{(D,E)}_\unif: \cC_\unif(D,k/I) \otimes_{k/I}  \cC_\unif(E,k/I) \map{}  \sC_\unif(D \times E,k/I) \;.
$$
Since $\sJ$ is filtered, the functor $\colimit_{D \in \sJ}$ preserves injectivity in $\Mod_{k/I}$, so we get an injective map 
$$\sC_\unif((X,\Theta_\sJ),k/I) \otimes_{k/I} \sC_\unif((Y, \Theta_\sH),k/I)  \longrightarrow 
\sC_\unif((X \times Y, \Theta_{\sJ \times \sH}),k/I) \;.$$
We now apply $\limit_{I \in \cP(k)}$ to the previous inductive system of morphisms and use Proposition~\ref{tensind} to get the result. 

%\par $\mathit 4.$ From \eqref{prodfctsacs} for $X=D$ and $Y=E$, 
%we obtain continuous $k$-bilinear maps
%$$
% \sC (D,k) \times  \sC (E,k) \longrightarrow \sC (D \times E,k)  
%$$
%and therefore, by composition with $\sC (D \times E,k)  \longrightarrow \sC_\alc (X \times Y,k)$, a projective system of 
%continuous $k$-bilinear maps
%\beq \label{projsysDE}
% \sC (D,k)  \times  \sC (E,k) \longrightarrow   \sC_\alc (X \times Y,k)  \;.
%\eeq
%In the product category 
%$\cCLMu_k \times \cCLMu_k$ we naturally identify 
%$$\colimit_{D \times E} ( \sC (D,k)  \times  \sC (E,k) ) \;\;\mbox{with}\;\;\colimit^\un_D \sC (D,k)  \times \colimit^\un_E \sC (E,k)$$
% so that the projective system \eqref{projsysDE} induces a continuous $k$-bilinear map
%$$\colimit^\un_D \sC (D,k)  \times \colimit^\un_E \sC (E,k)= \sC_\alc (X,k) \times \sC_\alc (Y,k) \longrightarrow   \sC (X \times Y,k) \;.$$
%From this  we obtain \eqref{prodfctsalc1}.
 \end{proof} 
 \begin{rmk} \label{prodfcts0unif1} Notice that $\cL_\unif^{(X,Y)}$ in \eqref{prodfcts11} is 
not an isomorphism in general.  
%The fact that \eqref{unimeasdisc2} is 
%not in general an isomorphism, 
This fact is already visible in case of \eqref{unimeasdisc2} when $k$ is discrete and $D$, $E$ are discrete but  infinite. 
In fact in the latter case, $\sC_\unif(D,k) = (k^D)^\dis$, $\sC_\unif(E,k) = (k^E)^\dis$,  and $\sC_\unif(D \times E,k) = (k^{D \times E})^\dis$ are discrete so that 
$\sC_\unif(D,k) \wt_k^\un \sC_\unif(E,k) = \sC_\unif(D,k) \otimes_k \sC_\unif(E,k)$. But if, say, 
 $D = E$, the characteristic  function of the diagonal does not belong to the image of $\cL^{(D,E)}_\unif$. Of course, if $k^D$, $k^E$ are endowed with their natural (product) topology, then the map
 $$
 k^D \wt_k^\un k^E \map{} k^{D \times E} \;\;,\;\; (a_\alpha)_\alpha \times (b_\beta)_\beta \longmapsto (a_\alpha b_\beta)_{\alpha,\beta}
 $$
 is an isomorphism and we are not contradicting  \eqref{prodfctsmor}. 
\end{rmk}
 \begin{rmk} \label{prodfctsstr} Let $X$ (resp. $(X,\Theta_\sJ)$)
be  a   $td$ space (resp. a uniform $td$ space). 
 It follows from  
 Proposition~\ref{contvsunif} that for $? =\emptyset, \acs$ (resp. for $? =\emptyset, \acs, \unif$), the diagonal map $\Delta_X : X \longrightarrow X \times X$ induces an associative and commutative product $(f,g) \longmapsto \mu_?(f,g)$, composite of the maps
$$
  \cC_?(X,k) \times \cC_?(X,k) \map{\wt^\un_k} \cC_?(X,k) \wt^\un_k \cC_?(X,k) \map{\cL^{(X,X)}_?}  \sC_?(X \times X,k) \map{\Delta_X^\ast} \sC_?(X,k)   \;,
$$
with the required properties (in particular uniform continuity in $(f,g)$ for the product uniformity). 
Then   $\sC(X,k)$ (resp.  $\sC_\unif(X,\Theta_\sJ)$) 
is a  commutative unital $k$-algebra object  of the category $\cCLMu_k$ (resp. and of the category $\cCLMcan_k$), with identity the constant  function $1$ on $X$.  If $X$ is a metrizable $td$-space, then $\sC(X,k)$  is a $k$-algebra object  of $\cCLMou_k$.
The product of  $\sC_\acs(X,k) \in \cCLMcan_k$ is also associative and commutative, but 
fails to define a unital $k$-algebra object of $\cCLMcan_k$ because the constant function $1$  is not in $\sC_\acs(X,k)$ unless $X$ is compact. 
 We recall that 
$\cL^{(X,X)}$ and $\cL^{(X,X)}_\acs$ are  isomorphisms, while $\cL^{(X,X)}_\unif$ is only an injective morphism.
 \end{rmk} 
%\bmau \begin{rmk} \label{alcring} We do not know whether  $\sC_\alc(X,k)$ is  a ring object of
%$\cCLMu_k$  in general. This is the main reason why that $k$-module does  not appear in our final statements.
%On the other hand  the definition of $\sC_\alc(X,k)$ does play  a useful, if purely technical, role in our study of functions and measures on $td$-spaces.
%%for not pursuing the investigation of $\sC_\alc(X,k)$ in this paper. 
%\end{rmk} \emau
%\par \smallskip
%In general, on  a compact topological space $(X,\tau)$ there is a unique uniform structure $\Theta_\tau$ compatible with the given topology.   
% In case $(X,\tau)$ is a Stone space any open partition is finite and compact, so $\Theta_\tau$ consists of all open partitions. 
%That is, the family of uniform open compact partitions consists of the family 
%of finite open partitions of $X$.  
\end{subsection}
 \end{section}
\begin{section}{Measures with values in $k$ on $td$-spaces  via duality} \label{measdual}  
We will now construct functors on the categories $td-\cS$ and  $td-\cU$ with values in the category $\cCLMu_k$, dual, as much as possible,  to the functors constructed in section~\ref{STScontinuity}. 
Let $X$ be a $td$-space (resp. $X=(X,\Theta_\sJ)$ be a $td$-uniform space, 
for a cofiltered sub-projective system $\sJ$ of $\sQ(X)$).
%where $X$ is the projective limit of a projective system $\sJ$ as in  \eqref{prodis2}). 
For the $td$-space $X$ we have defined 2 types of rings of continuous $k$-valued functions on $X$, namely  $\sC_\acs(X,k) \subset  \sC(X,k)$. For the $td$-uniform space $X=(X,\Theta_\sJ)$ we defined a further $k$-algebra $\sC_\unif((X,\Theta_\sJ),k)$ of uniformly continuous  functions situated between $\sC_\acs(X,k)$ and $\sC(X,k)$.  They are related by  sequence~\ref{canacs3}. 
%\bmau As already indicated, $\sC_\alc(X,k)$ will only play a minor, purely technical role, 
%and will not appear in the final statements. \emau
\begin{subsection}{Case of $X = C$ a Stone space}
We recall 
%from Corollary~\ref{acsalc} 
that in this case 
 all
morphisms in \eqref{canacs3} are isomorphisms and that 
\eqref{acsalc2} holds. 
We then introduce both the strong dual 
\beq
\label{compstrong} \begin{split}
 \sD^\circ(C,k) &:= \sC(C,k)'_\strong =\limit^{\square,\un}_{F \in \sQ(C)} k^{(F,\un)} \\ &= \limit^\can_{F \in \sQ(C)} k^{(F,\un)}  \in 
 \cCLMcan_k
\end{split}
\eeq
and the weak dual
\beq
\label{compweak} 
 \sD(C,k) := \sC(C,k)'_\weak = \limit_{F \in \sQ(C)} k^{(F,\un)} \in \limit\, \cF_k
\eeq
of  $\sC(C,k)$. Notice that, if $C$ is metrizable, then $\sD(C,k)$ is a sequential limit of objects of $\cF_k$, hence is an object of $\cCLMou_k$.  
\par \medskip The natural morphism 
$$
 \sD^\circ(C,k) \longrightarrow  \sD(C,k)
$$
is bijective by \eqref{sqlimit1}, so that we may regard the underlying $k$-module of \eqref{compstrong} and of 
\eqref{compweak} as the same $k$-module $\sD(C,k)^\for = \hom_{\cCLMu_k}(\sC(C,k),k)$ of \emph{measures on} (the Stone space) $C$ equipped with the \emph{strong} (resp. \emph{weak}) topology. By Lemma~\ref{denslc} locally constant functions $C \longrightarrow k$ are dense in $\sC(C,k)$, so a $k$-valued measure on $C$ may be viewed as a  map  
\beq \label{measStone}
\mu: \Sigma(C) = \cK(C)  \longrightarrow k
\eeq
such that if $C_1, \dots, C_n$ are a finite family of disjoint clopen subsets of $C$, 
\beq \label{measStone1} \mu(\bigcup_{i=1}^nC_i) = \sum_{i=1}^n \mu(C_i)\;.
\eeq
(We are using the fact that, for $M,N \in \Mod_k$, any $k$-linear map $M \to N$ is continuous for the naive canonical topology.)
\begin{rmk} \label{basiscomp} Obviously, a measure $\mu$ on the Stone space $C$ is determined by the restriction of 
\eqref{measStone} to any basis of the topology of $C$ contained in $\Sigma(C)$. 
\end{rmk}
A fundamental system of open $k$-submodules for  $\sD^\circ(C,k)$  is given by the family 
$\{V_I\}_{I \in \cP(k)}$ where
\beq \label{unifnhd}
V_I = \{ \mu \in  \sD(C,k)^\for \,| \, \mu(U) \in I\;,\;\forall \; U \in \Sigma(C)\;\}\;.
\eeq
For $\pi_F:C  \longrightarrow F$ in $\sQ(C)$, we identify $F$ with the corresponding clopen partition $\sP_F = \{\pi_F^{-1}(x)\}_{x\in F}$ of $C$. Then,  a fundamental system of open $k$-submodules for   $\sD(C,k)$ is given by the family of $k$-submodules  
\beq \label{simplenhd}
V_{F,I} :=\{ \mu \in  \sD(C,k)^\for \,|\, \mu(U) \in I \;,\;\forall \; U \in \sP_F\;\} 
\eeq
indexed by $F \in \sQ(C)$ and $I \in \cP(k)$.
Notice  that the natural maps  
%$$ \bigoplus_{x \in C} k \delta_x \longrightarrow \sD_\weak(C,k) \;,
%$$
\beq \label{compinj} 
 \bigoplus_{x \in C} k \delta_x \longrightarrow \sD^\circ(C,k)   \map{(1:1)} \sD(C,k) \;,
\eeq
where $\delta_x$ is interpreted as the Dirac mass at $x$, are obviously injective. It is also clear that the 
composite map has dense set-theoretic image. 
 We 
may then explicitly describe $\sD^\circ(C,k)$ and $\sD(C,k)$. 
\begin{prop}  The natural map  
$$ \bigoplus_{x \in C} k \delta_x \longrightarrow \sD^\circ(C,k) \;,
$$
is injective and 
$$
V_I \bigcap \bigoplus_{x \in C} k \delta_x = I \bigoplus_{x \in C} k \delta_x = \bigoplus_{x \in C} I \delta_x 
$$
so that 
\beq \sD^\circ(C,k) = {\bigoplus}_{x \in C}^\un k\delta_x \;\;\mbox{and}\;\; V_I =  I {\bigoplus}_{x \in C}^\un k \delta_x ={\bigoplus}_{x \in C}^\un I \delta_x \;.
\eeq 
\end{prop}

\begin{prop} \label{meascompexpl} Let $F \in \sQ(C)$ and $I \in \cP(k)$. 
Then
$$\sD(C,k)^\for = ({\bigoplus}_{x \in C}^\un k\delta_x)^\for$$
and  $\sD(C,k)$ identifies with (resp. the completion of) the $k$-module $({\bigoplus}_{x \in C}^\un k\delta_x)^\for$ (resp. $\bigoplus_{x \in C} k\delta_x$) equipped with the $k$-linear topology with fundamental system of open $k$-submodules $\{V_{F,I}\}_{F,I}$ (resp. $\{V'_{F,I}\}_{F,I}$), where 
\beq \label{weaktopmeas}
V_{F,I} = \{\sum_{x \in C} a_x \delta_x \in  {\bigoplus}_{x \in C}^\un k\delta_x \,|\, \sum_{x \in P} a_x \in I \;,\;\forall \; P \in \sP_F\,\} 
 \eeq
(resp. 
\beq \label{weaktopmeas1}
\begin{split}
V'_{F,I} :=& V_{F,I} \bigcap {\bigoplus}_{x \in C} k\delta_x  = \\ &\{\sum_{x \in C} a_x \delta_x \in  {\bigoplus}_{x \in C} k\delta_x \,|\, \sum_{x \in P} a_x \in I \;,\;\forall \; P \in \sP_F\,\} \;\mbox{)\;.}
\end{split} \eeq 
\end{prop}
\begin{rmk} \label{weakstrong2}  Notice that the expression $\sum_{x \in P} a_x$ in \eqref{weaktopmeas} represents  an unconditionally convergent series in $k$, while in  \eqref{weaktopmeas1} it is a finite sum. 
\end{rmk}
%\begin{rmk} \label{weakstrong} In this context the  natural topology of ${\bigoplus}_{x \in C}^\un k \delta_x$ (= the naive $k$-linear topology) corresponds to the topology of  $\sD_\strong(C,k)$ and is referred to as the \emph{strong} topology, 
%while the topology of the $\{V_{F,I}\}_{F,I}$ is the one of $\sD_\weak(C,k)$ and will be called the \emph{weak}  topology of 
%${\bigoplus}_{x \in C}^\un k \delta_x$. 
%\end{rmk}

%\begin{cor} \label{equidistr} Let $S$ be any dense subset of $C$. Then 
%$${\bigoplus}_{x \in S}^\un k\delta_x \subset {\bigoplus}_{x \in C}^\un k\delta_x$$
%is closed in the naive 
%\end{cor} 

\begin{cor} \label{unifballs} $\sD(C,k)$ is  precisely the  space of $k$-valued measures on $C$ in the sense of \eqref{measStone} and  \eqref{measStone1}, equipped  with the $k$-linear topology with fundamental system of open $k$-submodules $\{V_{F,I}\}_{F,I}$, for $F \in \sQ(C)$ and $I \in \cP(k)$, defined in \eqref{simplenhd}. Instead, $\sD^\circ(C,k) = \sD(C,k)^\can$, is the  space of $k$-valued measures on $C$ equipped with its naive $k$-canonical topology. 
 \end{cor}
 \begin{rmk} \label{funcmeas}   It follows from Corollary~\ref{ringtens1} that the transpose of the map $\cL^{(C,C')}$ of 
 \eqref{prodfcts1} in 
 $\mathit 1$ of Proposition~\ref{contvsunif}  gives a natural isomorphism 
 of $k$-coalgebra objects of $\cCLMu_k$
 \beq \label{funcmeas1}
 \cL^{(C,C')}_\sD :  \sD^!(C \times C',k) \iso \sD^!(C,k) \wt^\un_k  \sD^!(C',k)  
\eeq 
for $! = \emptyset,\circ$. 
On the other hand, it follows from Remark~\ref{basiscomp} that, for Stone spaces $C,C'$,  there is a natural map  
$$
 \sD(C,k) \times  \sD(C',k)  \map{} \sD(C \times C',k) 
 $$ 
 taking $(\mu,\nu)$ to the measure $\mu \times \nu$ on $C \times C'$ such that 
 $$(\mu \times \nu)(U \times V) = \mu(U) \nu(V)\;\;,\;\; \forall \: U \in \Sigma(C)\,,\, V \in \Sigma(C') \;.
 $$
 The map $\mu \otimes \nu \longmapsto \mu \times \nu$ is the inverse of $\cL^{(C,C')}_\sD$. For any $x \in C$ and $y \in C'$,  $\delta_x \otimes \delta_y \longmapsto \delta_{(x,y)}$, so  
the fact that \eqref{funcmeas1} is an isomorphism  also follows from the explicit  description of Proposition~\ref{meascompexpl}. 
 \end{rmk}
 We observe that by \eqref{compweak} (resp. \eqref{compstrong}) the duality rule 
\eqref{weakhomprolim} (resp. \eqref{weakhomprolimsq}) gives, taking into account \eqref{contcompl01},
\beq \label{stonebidual} \begin{split}
\sD(C,k)'_\weak  = \sD(C,k)'_\strong  = (\sD_\strong(C,k))'_\weak &= (\sD_\strong(C,k))'_\strong = 
\\  \colimit^\un_{F \in \sQ(C)} k^{(F,\un)}   = \sC(C,k)&\;.
\end{split}
\eeq
\begin{defn} \label{meascompact}
When $C$ is a Stone space, we are more  interested in  $\sD_\weak(C,k)$, so we simply denote it by $\sD(C,k)$ and understand that it carries the weak topology. That topology will be called the \emph{natural topology} of $\sD(C,k)$. 
We call $\sD(C,k)$ the $k$-module of \emph{measures on}  $C$ with values in $k$.  We will call instead $\sD^\circ(C,k)$ the $k$-canonical module $\sD_\strong(C,k) = \sD(C,k)^\can$. 
\end{defn}
\end{subsection}
\begin{subsection}{Case of $X =D$ a discrete space} 
 If $X = D$ is discrete   and the uniformity on $D$ is understood to be the discrete uniformity,    \eqref{canacs3} becomes \eqref{acsalc3dis}, which we now dualize. 
 \par  
By the second line of  formulas in Corollary~\ref{biduality},  the strong and weak duals coincide for $\sC(D,k)$. We set
\beq \label{acsmeasdisc}
\sD_\acs(D,k) := \sC(D,k)'_\strong = \sC(D,k)'_\weak = {\bigoplus}^\un_{x \in D} k \delta_x  \cong \sC_\acs(D,k) \in \cCLMcan_k
\eeq
and call it the $k$-module of ($k$-valued) measures \emph{with almost compact support} (on the discrete space $D$). 
\par \medskip  
The strong and weak duals $\sC_\acs(D,k)'_\strong$ and $\sC_\acs(D,k)'_\weak$ of $\sC_\acs(D,k)$
are 
\beq \label{Dstrongdef}
\sD^\circ(D,k) := {\prod}^{\square,\un}_{x \in D} k \delta_x\;\;\mbox{and}\;\; \sD(D,k) :={\prod}_{x \in D} k \delta_x
\;,
\eeq
respectively. 
We observe immediately that the second and third lines of formulas in Corollary~\ref{biduality}  imply
\beq \label{stupid} \begin{split}
\sD^\circ(D,k)'_\strong = \sD^\circ(D,k)'_\weak &= \sD(D,k)'_\strong =\sD(D,k)'_\weak = \\ &\sC_\acs(D,k) \;.
\end{split}
\eeq
%$\sD_\strong(D,k)$ and $\sD_\weak(D,k)$ 
%respectively, of $\sC_\acs(D,k)$,  
%coincide as $k$-modules with 
%\beq
%{\prod}_{x \in D} k \delta_x 
%\eeq
%(product taken in $\Mod_k$). So, 
%$$\sD_\strong(D,k) =  {\prod}^{\square,\un}_{x \in D} k \delta_x\;\;\mbox{and}\;\; \sD_\weak(D,k)=  {\prod}_{x \in D} k \delta_x
%\;,$$ 
%may be regarded as the $k$-module 
%$\sD(D,k)$ of ($k$-valued) \emph{measures} on the discrete space $D$, equipped with the \emph{strong} (resp. \emph{weak}) topology. 
We are mainly  interested in 
\beq \label{Dprodefn}
\sD(D,k) =   \sC_\acs(D,k)'_\weak =  {\prod}_{x \in D} k \delta_x
\;.
\eeq
We call $\sD(D,k)$ the $k$-module of ($k$-valued) \emph{measures}  (on the discrete space $D$). From the dense injection  $\sC_\acs(D,k) \subset \sC(D,k)$ in $\cCLMu_k$
we obtain by (strong and weak)  duality  the two injections
\beq \label{discrmeas} 
\sD_\acs(D,k) \longrightarrow  \sD^\circ(D,k) \map{(1:1)}   \sD(D,k) 
\eeq
whose composite has dense image. 
Explicitly they are
\beq \label{discrmeasexpl} \begin{split}
&\sD_\acs(D,k)  = \sC(D,k)'_\weak  =  \sC(D,k)'_\strong  = {\bigoplus}^\un_{x \in D} k \delta_x \longrightarrow  \\  
\sD^\circ(D,k) =  & \sC_\acs(D,k)'_\strong = {\prod}^{\square,\un}_{x \in D} k \delta_x  \map{(1:1)}    \sD(D,k) =   \sC_\acs(D,k)'_\weak =  {\prod}_{x \in D} k \delta_x \;. 
\end{split}
\eeq
By (weak and strong) duality again, we re-obtain the   injections  
\beq \label{discrmeasdual}
\sC_\acs(D,k)   \map{}   \sC^\circ(D,k)  \map{}   \sC(D,k) \;. 
\eeq
Explicitly they are
\beqa \label{discrmeasdual1} \begin{split}
 &\sD(D,k)'_\weak  =  (\sD^\circ(D,k))'_\strong =   {\bigoplus}^\un_{x \in D} k \chi_x =  \sC_\acs(D,k)   \map{}  \\   
 \sD_\acs(D,k)'_\strong &=  {\prod}^{\square,\un}_{x \in D} k \chi_x = \sC^\circ(D,k) 
 \map{}  \sD_\acs(D,k)'_\weak = {\prod}_{x \in D} k \chi_x = \sC(D,k) \;.\end{split}
\eeqa  
%\bmau
%Notice that
%\beq \label{Dprodual}
%\sD_\pro(D,k)'_\weak = {\bigoplus}^\un_{x \in D} k \chi_x = \sC_\acs(D,k)\;.
%\eeq
%We may also regard $\sD_\strong(D,k)$ (resp. $\sD_\weak(D,k)$) as the space of 
%$k$-valued functions $D \to k$ 
%equipped with the topology of uniform (resp. of simple) convergence on $D$. So,  $\sD_\strong(D,k) \cong \sC_\unif(D,k)$ while  $\sD_\weak(D,k) \cong \sC(D,k)$. \emau
\par \medskip  

Let $\sF(D)$ (resp. $\fP(D)$) be the set of finite (resp. of all) subsets of $D$. Then $\sD(D,k) =  {\prod}_{x \in D} k \delta_x$ (resp. 
$\sD_\strong(D,k) = {\prod}^{\square,\un}_{x \in D} k \delta_x$) may be described as 
the family of functions $\mu: \sF(D) \longrightarrow k$ such that, for any $X \in \sF(D)$,
\beq \label{discrmeas2}
\mu(X) = \sum_{x \in X} \mu(\{x\})  \;.
\eeq
Instead $\sD_\acs(D,k) =  {\bigoplus}^\un_{x \in D} k \delta_x$ is the family of functions $\mu: \fP(D) \longrightarrow k$
 with the following two properties~:
 \ben
 \item
 for any $I \in \cP(k)$, there exists a finite subset $F_I \subset D$, such that 
$$
\mu(x) \in I \;,\;\; \forall \; x \in D - F_I  \;;
$$
\item
for any $X \in \fP(D)$, the series $\sum_{x \in X} \mu(\{x\})$ is unconditionally  convergent to $\mu(X)$\;.
\een
Then $\sD(D,k)$  (resp. $\sD^\circ(D,k)$, $\sD_\acs(D,k)$) is equipped with the topology of simple (resp. of uniform) convergence 
on (the singletons of) $D$. We now observe that 
the strong and the weak dual of $\sC_\unif(D,k) = {\prod}_{x \in D}^\can  k \chi_x$ coincide and also coincide with the strong and the weak dual of $\sC(D,k) = {\prod}_{x \in D} k \chi_x$. We set
\beq \label{unifmeasdis}
\sD_\unif(D,k)  :=  \sC_\unif(D,k)'_\weak = \sC_\unif(D,k)'_\strong = \sC(D,k)'_\weak = \sC(D,k)'_\strong 
\eeq  
so that 
\beq \label{unifmeasdis1}
\sD_\unif(D,k)  =  \sD_\acs(D,k)= {\bigoplus}^\un_{x \in D} k \delta_x  \in \cCLMcan_k\;.
\eeq 
By the first line of  formulas in Corollary~\ref{biduality}, 
\beq \label{unifmeasdual}
\sD_\unif(D,k)'_\strong = \sC_\unif(D,k) \;\;,\;\; \sD_\unif(D,k)'_\weak = \sC(D,k) \;.
\eeq  
%\bmau so that, by formula \eqref{tensdual?}, 
%we see that $\sC_\unif(D,k) \in \cCLMcan_k$ is  a commutative $k$-algebra object of  $\cCLMcan_k$. This fact 
%was already proved in Remark~\ref{prodfctsstr}, for a general $td$-uniform space $X$. \emau
%\footnote{In fact, also of  $\cCLMcan_k$, but we have no use for this.}
Moreover,  \eqref{acsalc3dis} may be reinterpreted as
\beq
\label{acsalc3dual}   \begin{split}
\sC_\acs&(D,k)  =  \sD(D,k)'_\weak = \sD(D,k)'_\strong  = \\ &\sD^\circ(D,k)'_\weak = \sD^\circ(D,k)'_\strong   \;\;\; \longrightarrow \\ \sC_\unif(D,k)  =  & \sD_\unif(D,k)'_\strong   
\;\;\;  \longrightarrow \;\;\;
\sC(D,k) = \sD_\unif(D,k)'_\weak \;.
\end{split}
\eeq
%or, more concisely,
%\beq \label{acsalc3dualconc}
%\sC_\acs(D,k)    \longrightarrow \sC_\unif(D,k)  \longrightarrow  \sC_\alc(D,k) = \sC(D,k) \;. \eeq
%\eqref{Dstrongdef}
%$$\sD_\strong(D,k) := {\prod}^{\square,\un}_{x \in D} k \delta_x$$
We also have
\beq
\label{acsalc3meas}   \begin{split}
\sD_\acs(D,k) &= \sD_\unif(D,k)   \longrightarrow  \sD_\strong(D,k) =  \sC_\acs(D,k)'_\strong    \\
  \longrightarrow 
&\sD(D,k) = \sC_\acs(D,k)'_\weak   
\end{split}
\eeq
corresponding to the natural maps
$$
{\bigoplus}^\un_{x \in D} k \delta_x  \longrightarrow  {\prod}^{\square,\un}_{x \in D} k \delta_x  \longrightarrow   {\prod}_{x \in D} k \delta_x  \;.
$$
%or, more concisely,
%\beq
%\label{acsalc3measconc}    
%\sD_\acs(D,k)  = \sD_\unif(D,k)   \longrightarrow  \sD_\strong(D,k)  
%  \longrightarrow  \sD_\weak(D,k)\;.
%\eeq

%The strong (resp. weak) dual of \eqref{acsalc3dis}  is then 
%\beq
%\label{acsalc3sdual}   
%\sD_\acs(D,k) = {\bigoplus}^\un_{x \in D} k \delta_x   \longrightarrow  ({\prod}_{x \in D}^\can  k \chi_x )'_\strong \longrightarrow 
%\sD_\strong(D,k) = {\prod}^\can_{x \in D} k \delta_x  
%\eeq
%(resp.
%\beq
%\label{acsalc3sdual}   
%\sD_\acs(D,k) = {\bigoplus}^\un_{x \in D} k \delta_x   \longrightarrow  ({\prod}_{x \in D}^\can  k \chi_x )'_\weak \longrightarrow 
%\sD_\weak(D,k) = {\prod}_{x \in D} k \delta_x  \;\;\mbox{).}
%\eeq
%\bmau 
%Weak dual of \eqref{canacs3} for $X=D$ discrete 
%\beq
%\label{acsalc3wdual}   
%{\bigoplus}^\un_{x \in D} k \chi_x \iso  {\bigoplus}^\un_{x \in D} k \chi_x  \longrightarrow  ({\prod}_{x \in D}^\can  k \chi_x )'_\strong \longrightarrow 
%{\prod}_{x \in D} k \delta_x  
%\eeq
%\beq
%\label{acsalc3wdualbis}   
%\sD_\acs(D,k) \iso  \sD_\acs(D,k)  \longrightarrow  ({\prod}_{x \in D}^\can  k \chi_x )'_\weak \longrightarrow 
%\sD_\pro(D,k) 
%\eeq
%\emau
\begin{prop} \label{discrmeas} Let $D$ be a discrete topological space. Then 
\ben
\item 
$\sD(D,k) =  {\prod}_{x \in D} k \delta_x$ is the space  of functions $\mu: \sF(D) \longrightarrow k$ such that, for any $X \in \sF(D)$, \eqref{discrmeas2} holds. It is equipped with the topology of simple convergence 
on the singletons of $D$. 
\item 
$\sD^\circ(D,k) =  {\prod}^\can_{x \in D} k \delta_x$ coincides set-theoretically with $\sD(D,k)$. 
%is the space  of functions $\mu: \sF(D) \longrightarrow k$ such that, for any $X \in \sF(D)$, \eqref{discrmeas2} holds. 
It is equipped with its naive $k$-canonical topology \ie with the topology of uniform convergence 
on the singletons of $D$. 
\item
$\sD_\acs(D,k) = \sD_\unif(D,k) = {\bigoplus}^\un_{x \in D} k \delta_x $ is the space of functions 
\beq  \label{discrmeas1}
\mu: \fP(D) \longrightarrow k
\eeq
such that for any $X \subset D$, the $X$-series 
\eqref{discrmeas2}
is unconditionally convergent in $k$ to $\mu(X)$. It is equipped with the topology of uniform convergence 
on $\fP(D)$ which makes it a (closed) subspace of  $\sD^\circ(D,k)$. 
\een
\end{prop}
\begin{proof}
Clear. 
\end{proof}
\end{subsection} 
\begin{subsection}{Functoriality of functions and measures} \label{functoriality}
\begin{subsubsection}{Restriction of functions and extension by $0$ of measures on Stone spaces}  
\label{restrfcts} For $C \subset C'$ compact opens in $X$ we have ``restriction'' maps 
\beq \label{restr}
\begin{split}
\rho_{C,C'}: \sC(C',k) &\longrightarrow \sC(C,k) \\
f &\longmapsto f_{|C}
\end{split}
\eeq
and, by ?-duality, for $? = \weak, \strong$, ``extension by 0'' maps for measures, for $!=\emptyset,\circ$, respectively, 
\beq \label{ext0meass}
\begin{split}
j_{C,C'} : \sD^!(C,k) &\longrightarrow \sD^!(C',k) \\
\mu&\longmapsto j_{C,C'}(\mu) \;, 
\end{split}
\eeq
where $j_{C,C'}(\mu)(f) = \mu(f_{|C})$, for any $f \in \sC(C',k)$. In particular, $j_{C,C'}(\delta_x) = \delta_x$, for any $x \in C$.  
\par \smallskip
Since  
$$\sC(X,k) = \limit_{C \in \cK(X)} (\sC(C,k), \rho_{C,C'})  
$$
the following definition makes sense.
%we deduce by  
%weak duality from \eqref{Dacsdef}
%\beq \label{Dweak} 
%\sD_\acs(X,k) =  \sC(X,k)'_\weak  \in \colimit^\un \, \limit \, \cF_k 
%\eeq
%while we deduce by strong duality from \eqref{Dacsodef} 
%\beq \label{Dstrong}  
%\sD^\circ_\acs(X,k) = \sC(X,k)'_\strong  \in \cCLMcan_k \;.
%\eeq
\begin{defn} \label{Dacsdef0}
We set 
\beq \label{Dacsdef}
\sD_\acs(X,k) := \sC(X,k)'_\weak = \colimit^\un_{C \in \cK(X)} (\sD(C,k), j_{C,C'})  
\eeq
and
\beq \label{Dacsodef}
\sD_\acs^\circ(X,k)  := \sC(X,k)'_\strong = \colimit^\un_{C \in \cK(X)} (\sD^\circ(C,k), j_{C,C'})  \in \cCLMcan_k \;.
\eeq
\end{defn}
By application of $\limit_C$ to \eqref{restr} we obtain
\beq \label{restrext0dual1} \rho_C: \sC(X,k)  \longrightarrow \sC(C,k) \;\;\mbox{(restriction)} 
\eeq
and from \eqref{Dacsdef} (resp. \eqref{Dacsodef}) we get the weak (resp. strong) dual \emph{extension by $0$ map} 
\beq \label{restrext0dual}
\begin{split}
j_C : \sD(C,k) &\longrightarrow \sD_\acs(X,k) \;\;\mbox{(weak case)}  \;, \\
\mbox{(resp. \;}\; j_C : \sD^\circ(C,k) &\longrightarrow \sD^\circ_\acs(X,k) \;\;\mbox{(strong case)\;)}  \;.
\end{split}
\eeq 
%\beq \label{restrext0dual}
%\begin{split}
%\rho_C: \sC(X,k) &\longrightarrow \sC(C,k) \;\;\mbox{(restriction)}\;,\\
%j_C : \sD(C,k) &\longrightarrow \sD_\acs(X,k) \;\;\mbox{()}  \;, 
%\end{split}
%\eeq
%weakly dual to each other.  
\end{subsubsection}
\begin{subsubsection}{Extension by $0$ of functions and restriction of measures on Stone spaces}
\label{extfcts}
We have  an ``extension by 0'' map for functions 
\beq \label{ext0}
\begin{split}
i_{C,C'}: \sC(C,k) &\longrightarrow \sC(C',k) \\
f &\longmapsto i_{C,C'}(f) \;, 
\end{split}
\eeq
and its   weak transpose ``restriction'' map for measures
\beq \label{ext0}
\begin{split}
\sigma_{C,C'}  : \sD(C',k) &\longrightarrow \sD(C,k) \\ 
\mu &\longmapsto \sigma_{C,C'}(\mu) \;, 
\end{split}
\eeq
such that $\sigma_{C,C'}(\mu)(f) = \mu (i_{C,C'}(f))$, for any $\mu \in \sD(C',k)$ and $f \in \cC(C,k)$. 
\begin{defn} \label{dualdefn}
We set  
 \beq \label{dualdefn5}
\sD(X,k)   := \sC_\acs(X,k)'_\weak = \limit_{C \in \cK(X)} (\sD (C,k),\sigma_{C,C'})  \in  \limit \, \cF_k 
\eeq  
and obtain 
\beq \label{ext0restrdual}
\begin{split}
i_C: \sC(C,k) &\longrightarrow \sC_\acs(X,k)  \;\;\mbox{(extension by $0$)} \;,\\
\sigma_C : \sD(X,k) &\longrightarrow \sD(C,k)   \;\;\mbox{(restriction)}\;, 
\end{split}
\eeq
weakly dual to each other. Similarly, 	
we set 
 \beq \label{dualdefn5st}
\sD^\circ(X,k)   := \sC_\acs(X,k)'_\strong = \limit^{\square,\un}_{C \in \cK(X)} (\sD^\circ(C,k),\sigma_{C,C'})  \in   \cCLMcan_k
\eeq  
and then $\sigma_C$ in \eqref{ext0restrdual} is replaced by the restriction map 
$$\sigma_C : \sD^\circ(X,k) \longrightarrow \sD^\circ(C,k) $$
where $i_C$ and $\sigma_C$ are strongly dual to each other.
\end{defn}
Notice that, for any $x \in X$, 
\beq \sigma_C(\delta_x) = 
\begin{cases} 
\delta_x
    & \text{if } x \in C \\
&\\  
 0
    & \text{if }  x \in X-C \;.
\end{cases} 
\eeq
\end{subsubsection}
\begin{subsubsection}{Pull-back of functions and push-forward of measures on discrete spaces}
\label{discrpbs}
For $\pi_D:X \to D$ and $\pi_{D'}:X \to D'$ in $\sQ(X)$ and the morphism $\pi_{D,D'}:D' \to D$ in 
$\sQ(X)$, we  have the ``pull-back'' maps of functions 
\beq  
\label{?pullback}
\begin{split}
\pi^\ast_{D,D'} : \sC_?(D,k) &\longrightarrow \sC_?(D',k) \\
f &\longmapsto f \circ \pi_{D,D'}
\end{split}
\eeq
for $? = \emptyset, \acs$, and  the  weak transpose of each of them. Let $!=\acs$ if $?=\emptyset$ and 
$!=\emptyset$ if $?=\acs$. 
We get ``push-forward'' maps of measures 
\beq \label{contpushout}
h_{D,D'}  : \sD_!(D',k) \longrightarrow \sD_!(D,k) 
 \eeq
such that $h_{D,D'} (\mu)(f) = \mu(f \circ \pi_{D,D'})$ for any $\mu \in \sD_!(D',k)$ and $f \in \sC_?(D,k)$. The 
push-forward  of a measure $\mu$  attributes a $\mu$-weight to each fiber  by adding the $\mu$-weights of all points in each fiber\;:
\beq \label{contpushout1}
\mu \longmapsto  \sum_{x \in D} \mu( \pi^{-1}_{D,D'}(x)) \, \delta_x\;.
\eeq
By application of $\colimit^\un_D$ and $\limit_D$, respectively, we obtain,  
\ben
\item  in case $? = \emptyset$, $! = \acs$
\beq \label{ext1restrdual}
\begin{split}
\pi_D^\ast: \sC(D,k) &\longrightarrow \sC(X,k)   \;\;\mbox{(pull-back of functions)}\;, \\
h_D : \sD_\acs(X,k) &\longrightarrow \sD_\acs(D,k)  \;\;\mbox{(push-forward of measures)}\;, 
\end{split}
\eeq
weakly dual to each other;
%where we have set
%\beq \label{alcdef}
% \sC_\alc(X,k):=  \colimit^\un_D   \sC(D,k) 
% \eeq
% and
%\beq \label{proacsdef}
%\sD_\proacs(X,k) :=  \limit_{D \in \sQ(X)} (\sD_\acs(D,k), h_{D,D'} ) \;.
%\eeq  
\item
in case $? = \acs$, $! = \emptyset$
\beq \label{ext1prorestrdual}
\begin{split}
\pi_D^\ast: \sC_\acs(D,k) &\longrightarrow \sC_\acs(X,k)   \\
h_D : \sD(X,k)  &\longrightarrow \sD(D,k)   \;, 
\end{split}
\eeq
weakly dual to each other.
\een
The  map $h_D$  in both \eqref{ext1restrdual} and \eqref{ext1prorestrdual} is given by 
\beq \label{contpushoutinfty}
\mu \longmapsto  \sum_{x \in D} \mu( \pi^{-1}_D(x)) \, \delta_x\;,
\eeq
which makes sense because $\pi^{-1}_D(x) \in \cK(X)$, for any $x \in D$, so that $\mu( \pi^{-1}_D(x))$ is defined. 
%limit of the formula  \eqref{contpushout1} for $D' \in \sQ(X)/D$. 
\end{subsubsection}
 \end{subsection}
\begin{subsection}{The case of $X$ a general $td$-space}
 \label{generalcase} 
\begin{subsubsection}{Measures with  and without support conditions} 
\begin{rmk} \label{acscirc}
For $X=D$ discrete, comparing \eqref{Dacsdef} and \eqref{Dacsodef} with \eqref{acsmeasdisc} we see that 
$$
\sD_\acs(D,k) = \sD^\circ_\acs(D,k) \in \cCLMcan_k  \;.
$$
For $X=C$ a Stone space, $\sD_\acs(C,k) = \sD(C,k)$ while $\sD^\circ_\acs(C,k) = \sD^\circ(C,k)$.
 \end{rmk}
 We have 
 \beq \label{generalbidual} \begin{split} \sD_\acs(X,k)'_\weak = &\l(\colimit^\un_{C \in \cK(X)}  (\sD(C,k) ,j_{C,C'})\r)'_\weak = \\ \limit_{C \in \cK(X)} & \sD(C,k)'_\weak   =   \limit_{C \in \cK(X)} 
  \sC(C,k)  = \sC(X,k)\;,
 \end{split}
\eeq
where we have used the biduality formula \eqref{stonebidual} for the Stone space $C$. Similarly, we also have   
 \beq \label{generalbidualst} \begin{split} \sD^\circ_\acs(X,k)'_\strong = &\l(\colimit^\un_{C \in \cK(X)} (\sD^\circ(C,k) ,j_{C,C'})\r)'_\strong = \\  \limit^{\square,\un}_{C \in \cK(X)}&  \sD^\circ(C,k)'_\strong  =    
 \limit^\can_{C \in \cK(X)} \sC(C,k)  = \sC^\circ(X,k)\;,
 \end{split}
\eeq
where again  we have applied  \eqref{stonebidual} to $C$.  
In the same way we also check that
 \beq \label{bisform}
 \sD^\circ_\acs(X,k)'_\weak = \sD_\acs(X,k)'_\weak \;\;\mbox{and} \;\;
 \sD_\acs(X,k)'_\strong =\sD^\circ_\acs(X,k)'_\strong \;.
 \eeq
We also observe for completeness that  from  \eqref{weakhomprolimsq} we deduce
 \beq \label{bisform1} \begin{split} 
 (\sC^\circ(X,k))'_\strong &= ( \limit^\can_{C \in \cK(X)} \sC(C,k) )'_\strong = \\ \colimit^\un_{C \in \cK(X)} (\sD^\circ(C,k),j_{C,C'})& =  \sD^\circ_\acs(X,k)\;.
\end{split}  \eeq
\par \bigskip
Morphisms $C \longrightarrow X$ in $\cK(X)$ and  $\pi_D: X \longrightarrow D$ in $\sQ(X)$ determine  sequences 
\beq \label{fundseq0} \begin{split}
\sC(D,k)  \map{\pi_D^\ast}  &\sC(X,k) \map{\rho_C} \sC(C,k) \\
\sD(C,k)   \map{j_C} &\sD_\acs(X,k)    \map{h_D} \sD_\acs(D,k) \;,
 %\\ \mbox{(resp.} \;\sD_\strong(C,k)   \map{j_C} &\sD^\circ_\acs(X,k)    \map{h_D} \sD^\circ_\acs(D,k)\;\mbox{)}\;,
\end{split}
\eeq
(resp. 
\beq \label{fundseq0st} \begin{split}
\sC^\circ(D,k)   \map{\pi_D^\ast}  &\sC^\circ(X,k) \map{\rho_C} \sC(C,k) \\
 \sD^\circ(C,k)   \map{j_C} &\sD^\circ_\acs(X,k)    \map{h_D} \sD_\acs(D,k)\;\mbox{)}\;,
\end{split}
\eeq
weakly (resp. strongly) dual to each other.    
\par \medskip 
Since we have the two equalities
%\eqref{contdiscr2} and \eqref{?pullback}, we have 
 $$\sC(X,k) = \colimit^\un_{D \in \sQ(X)} (\sC(D,k), \pi_{D,D'}^\ast)= \limit_{C \in \cK(X)} (\sC(C,k), \rho_{C,C'})
 $$ 
 we conclude that we have at the same time 
\beq \label{Dacslimdiscr} \begin{split}
\sD_\acs(X,k) =& \limit_{D \in \sQ(X)} ( \sD_\acs(D,k), h_{D,D'}) \\
\sD^\circ_\acs(X,k) =& \limit^{\square,\un}_{D \in \sQ(X)} ( \sD_\acs(D,k), h_{D,D'}) \in \cCLMcan_k
\end{split}\eeq
and 
\beq \label{Dacscolimcan} \begin{split}
\sD_\acs(X,k) =& \colimit^\un_{C \in \cK(X)} ( \sD(C,k), j_{C,C'}) \\
\sD^\circ_\acs(X,k) =& \colimit^\un_{C \in \cK(X)} (\sD^\circ(C,k), h_{D,D'}) \in \cCLMcan_k \;.
\end{split}
\eeq
%with canonical projections \eqref{ext1restrdual}
%\beq \label{Dacslimdiscr1}
%\sD^\circ_\acs(X,k) \map{(1:1)} \sD_\acs(X,k) \map{h_D} \sD_\acs(D,k) \;.
%\eeq
\begin{cor} \label{bijmap} The canonical morphism 
$$
 \sD^\circ_\acs(X,k) \map{} \sD_\acs(X,k)
 $$
 induced by the bijective morphisms $\sD^\circ(C,k) \map{(1:1)} \sD(C,k)$, for any $C \in \cK(X)$, is bijective. Therefore, 
 \beq  \label{bijmap1}
  \sD^\circ_\acs(X,k) = \sD_\acs(X,k)^\can\;.
 \eeq
\end{cor}
%Since $\sC(X,k) = \limit_{C \in \cK(X)} (\sC(C,k), \rho_{C,C'})$ we also have 
%\beq \label{Dacscolimcan} \begin{split}
%\sD_\acs(X,k) =& \colimit^\un_{C \in \cK(X)} ( \sD(C,k), j_{C,C'}) \\
%\sD^\circ_\acs(X,k) =& \colimit^\un_{C \in \cK(X)} (\sD_\strong(C,k), h_{D,D'}) \in \cCLMcan_k \;.
%\end{split}
%\eeq
Morphisms $C \longrightarrow X$ in $\cK(X)$ and  $X \longrightarrow D$ in $\sQ(X)$ determine, by \eqref{dualdefn5}, \eqref{ext0restrdual}, 
 \eqref{restrext0dual} and \eqref{ext1prorestrdual}, 
 sequences 
 \beq \label{fundseq2} \begin{split} \sC_\acs(D,k) \map{\pi_D^\ast} \sC_\acs(X,k) &\longrightarrow \sC(X,k) \map{\rho_C} \sC(C,k) 
 \map{i_C} \sC_\acs(X,k) \\
\sD(X,k) \map{\sigma_C}  \sD(C,k) &\map{ j_C} \sD_\acs(X,k) \longrightarrow    \sD(X,k)  \map{h_D} \sD(D,k) \;,
 \end{split}
 \eeq
weakly dual to each other. 
\begin{prop} \label{bidual} For any $td$-space $X$, we have
\beq \label{bidual1}
 \begin{split} \sD(X,k)'_\weak  =& \sD(X,k)'_\strong =  \sD^\circ(X,k)'_\weak = \sD^\circ(X,k)'_\strong = \\
& \sC_\acs(X,k) \in \cCLMcan_k 
 \end{split}
 \eeq
 \end{prop}
 \begin{proof}
From Definition~\ref{dualdefn} it follows that 
 \beq \label{bidual2} \begin{split}
\sD(X,k)'_\weak &=   \colimit^\un_{C \in \cK(X)} \sD (C,k)'_\weak =  \\ &\colimit^\un_{C \in \cK(X)} \sC(C,k) = 
\sC_\acs(X,k) \in \cCLMcan_k 
\end{split}
\eeq
and 
 \beq \label{bidual3} \begin{split}
\sD^\circ(X,k)'_\strong &=   \colimit^\un_{C \in \cK(X)} \sD^\circ (C,k)'_\strong =  \\ &\colimit^\un_{C \in \cK(X)} \sC(C,k) = \sC_\acs(X,k) \in \cCLMcan_k \;.
\end{split}
\eeq
Similarly, 
\beq \label{bidual4} \begin{split}
\sD(X,k)'_\strong &=   \colimit^\un_{C \in \cK(X)} \sD(C,k)'_\strong  =   \colimit^\un_{C \in \cK(X)}  \sC(C,k) \\
\sD^\circ(X,k)'_\weak &=  \colimit^\un_{C \in \cK(X)} \sD^\circ(C,k)'_\weak =  \colimit^\un_{C \in \cK(X)}  \sC(C,k) \;.
\end{split}
\eeq
\end{proof}
%\beq \label{fcttrasp}
%\sC(D,k)  \map{\pi_D^\ast} \sC(X,k) \map{\rho_C} \sC(C,k)
%\eeq
%and, by weak duality,   morphisms
%\beq \label{functweak}
%\sD(C,k)  \mbox{(}= \sD_\acs (C,k)\mbox{)}  \map{j_C} \sD_\acs(X,k)   \map{} \sD_\acs(D,k) \;.
%\eeq
%From the definitions of  $\sD_\acs(X,k)$ \eqref{Dacsdef} and of $\sD_\proacs (X,k)$ \eqref{proacsdef} 
%we may factor  the second morphism in \eqref{functweak} as 
%\beq \label{functweak1} \begin{split}
%  (\colimit^\un_C \sD(C,k), j_{C,C'}) =& \sD_\acs(X,k) \longrightarrow \sD_\proacs (X,k) = \limit_D (\sD_\acs(D,k), h_{D.D'}) \\
%   &\map{h_D} \sD_\acs(D,k) \;.
%   \end{split}
%\eeq
\end{subsubsection}

\begin{subsubsection}{Uniform measures} 
Let $(X,\Theta_\sJ) = \limit_{D \in \sJ} (X,\{\pi_D\})$ be a neat $td$-uniform space,  
as in \eqref{projlimunif}. 
  Then we define    
    \beq \label{measunifdef} 
  \sD_\unif ((X,\Theta_\sJ),k) := \limit_{D \in \sJ} (\sD_\unif(D,k), h_{D,D'})  
    \eeq
    and
      \beq \label{measunifodef} 
  \sD^\circ_\unif ((X,\Theta_\sJ),k) := \limit^\can_{D \in \sJ} (\sD_\unif(D,k), h_{D,D'}) \in \cCLMcan_k\;.
    \eeq 
    The elements of $\sD_\unif ((X,\Theta_\sJ),k)$ will be called \emph{uniform} ($k$-valued) \emph{measures} on 
    $(X,\Theta_\sJ)$. We recall that, by \eqref{unifmeasdis1},  $\sD_\unif(D,k)$ may be replaced by $\sD_\acs(D,k)$ in 
    \eqref{measunifdef}  so that, by 
   \eqref{Dacslimdiscr},  we deduce a morphism
%   commutative square 
%   \beq \label{measunifacssquare} 
% \;\;\;  \begin{tikzcd}[column sep=2.5em, row sep=2.5em] 
%\sD^\circ_\acs(X,k)  \arrow{r}{(1:1)} \arrow{d}{\emb^\circ}  & \sD_\acs(X,k)  \arrow{d}{\emb}
%\\ 
%\sD^\circ_\unif ((X,\Theta_\sJ),k)     \arrow{r}{(1:1)} &\sD_\unif ((X,\Theta_\sJ),k)  \;\;.
%\end{tikzcd}
%\eeq
%where the vertival maps are obtained from the push-forwards of measures  $h_D:  \sD_\acs(X,k) \to \sD_\acs(D,k)$, for $D \in \sJ$ \eqref{ext1restrdual}. 
    \beq \label{measunifacs1}  
  \mbox{(}\;  \sD^\circ_\acs(X,k) \map{(1:1)} \; \mbox{)}\; \sD_\acs(X,k) \map{\inc}    \sD_\unif ((X,\Theta_\sJ),k)  
    \eeq
   with dense set-theoretic image  obtained from the push-forwards of measures  $h_D:  \sD_\acs(X,k) \to \sD_\acs(D,k)$, for $D \in \sJ$ \eqref{ext1restrdual}. 
\par  
Since $(X,\Theta_\sJ)$ is neat,  the $k$-linear map  
\beq \label{diracunif}
{\bigoplus}_{x \in X} k \delta_x \longrightarrow \sD_\unif ((X,\Theta_\sJ),k) = \limit_{D \in \sJ} {\bigoplus}^\un_{d \in D} k\delta_d 
\eeq
sending any $x \in X$ to the projective system $\{\delta_{\pi_D(x)}\}_{D \in \sJ}$ is an 
injection with dense image.  
So that $\sD_\unif ((X,\Theta_\sJ),k)$ contains a copy of the  $k$-module freely generated by $\{\delta_x\}_{x \in X}$. 
\par \medskip  
We then have  
  \beq \label{funcunifdual} \begin{split}
  \sC_\unif & ((X,\Theta_\sJ),k)'_\weak =   \colimit^\un_{D \in \sJ} \sC_\unif(D,k)'_\weak = \\& \limit_{D \in \sJ} \sD_\unif(D,k) =   \sD_\unif  ((X,\Theta_\sJ),k)   \end{split} 
\eeq
  \beq \label{funcunifdualst} \begin{split}
  \sC_\unif & ((X,\Theta_\sJ),k)'_\strong =  \colimit^\un_{D \in \sJ} \sC_\unif(D,k)'_\strong = \\& \limit^{\square,\un}_{D \in \sJ} \sD_\unif(D,k) =   \sD^\circ_\unif  ((X,\Theta_\sJ),k)   \end{split} 
\eeq
    \beq \label{measunifdefst} \begin{split} 
    \sD_\unif & ((X,\Theta_\sJ),k)'_\strong = ( \limit_{D \in \sJ} (\sD_\unif(D,k), h_{D,D'}))'_\strong = \\&
    \colimit^\un_{D \in \sJ} (\sC_\unif(D,k),\pi_{D,D'}^\ast) = \sC_\unif ((X,\Theta_\sJ),k)
  \in \cCLMcan_k \;.
  \end{split}  
  \eeq
   \beq \label{funcunifdef} \begin{split} \sD_\unif &((X,\Theta_\sJ),k)'_\weak =  (\limit_{D \in \sJ}  \sD_\unif(D,k))'_\weak =\\
  & \colimit^\un_{D \in \sJ} \sC(D,k)   =  \sC  (X,k)  \;.
  \end{split}  \eeq
 Only the last  formula needs  justification.  For any $D \in \sJ$, we have a natural bijective morphism,
 $\sC_\unif(D,k) \map{} \sC(D,k)$ and  therefore taking (filtered) set-theoretic colimits a bijection  in $\Mod_k$
\beq \label{setth}   \colimit_{D \in \sJ} \sC_\unif(D,k)  \map{}  \colimit_{D \in \sJ} \sC(D,k)\;.\eeq
 The natural morphism 
 $$  \colimit^\un_{D \in \sJ}  \sC_\unif(D,k)  = \sC_\unif ((X,\Theta_\sJ),k) \map{}  \colimit^\un_{D \in \sJ} \sC(D,k) $$
comes from completion of the l.h.s. of \eqref{setth} in the topology of uniform convergence on $X$, which produces $\sC_\unif ((X,\Theta_\sJ),k)$, and of the r.h.s. of \eqref{setth}, or equivalently of $\sC_\unif ((X,\Theta_\sJ),k)^\for$,  in the topology of uniform convergence on the family 
$$\{\pi_D^{-1}(F)\,|\, D \in \sJ\;,\;F \in \cF(D)\;\}\;.$$ 
But the latter family if cofinal with $\cK(X)$, and the completion of 
$\sC_\unif ((X,\Theta_\sJ),k)^\for$ in that topology is therefore $\sC(X,k)$. 
So 
completion of \eqref{setth}  produces an isomorphism
 $$\sC_\unif ((X,\Theta_\sJ),k) \iso  \colimit^\un_{D \in \sJ} \sC(D,k) = \sC(X,k)\;,$$
 as claimed.   
   \begin{rmk} \label{dualex} The integration pairing $(\mu , f) \longmapsto \mu \circ f := \int_X f(x) d\mu(x)$ on 
\beq \label{dualex1} \sD_\unif((X,\Theta_\sJ),k) \times \sC_\unif((X,\Theta_\sJ),k) \longrightarrow k 
\eeq
 is weak-left perfect   by \eqref{funcunifdual} and strong-right perfect \eqref{measunifdefst}. The same integration pairing  for
\beq \label{dualex2} \sD_\acs(X,k) \times \sC(X,k) \longrightarrow k\;\;,\;\; (\mu , f) \longmapsto \int_X f(x) d\mu(x) \;.
\eeq
 is weak-perfect. 
    \end{rmk} 
\end{subsubsection}
\begin{subsubsection}{Coproduct of measures} 
  We deduce from \eqref{discrmeasexpl} and \eqref{unifmeasdis1} that, for $D$ and $E$ discrete,  the 
$k$-linear isomorphism  
\beq \label{liniso}
 {\bigoplus}_{x \in D, y \in E} k \delta_{(x,y)} \iso
 {\bigoplus}_{x \in D} k \delta_x \otimes_k {\bigoplus}_{y \in E} k \delta_y 
\eeq
extends to an isomorphism   
\beq
\label{disprod?meas} \sD_?(D \times E,k) 
\iso
\sD_?(D,k) \wt^\un_k \sD_?(E,k)   \;, 
\eeq 
 for $? = \acs,\emptyset$. Notice that the morphism
 \beq
\label{disprod?meas1}
\sD^\circ(D \times E,k) 
\map{}
\sD^\circ(D,k) \wt^\un_k \sD^\circ(E,k)
\eeq
is injective but not necessarily an isomorphism.
  \par \medskip

We define a $k$-coalgebra structure on $\sD_?(D,k)$ via the $k$-linear continuous coproduct
\beq \label{discrcoalgebra}
\P^{(D)}_\sD : \sD_?(D,k) \longrightarrow \sD_?(D \times D,k) \iso \sD_?(D,k) \wt^\un_k \sD_?(D,k)
\eeq
sending $\delta_x \longmapsto \delta_{(x,x)} \longmapsto \delta_x \wt^\un_k \delta_x$, for any $x \in D$. Similarly on 
$\sD^\circ(D,k)$.
On  $ \sD_\acs(D,k)$ we also have the ($k$-linear continuous) augmentation 
$$
\veps^{(D)}_\sD : \sD_\acs(D,k) \to k 
%\;\;\mbox{(resp.}\;   \sD_\unif(D,k) \to k \;\mbox{)} 
\;,
$$ 
 sending $\delta_x$ to 1, for any $x \in D$, so that $\sD_\acs(D,k)$ ($=\sD_\unif(D,k)$) is a (unital) $k$-coalgebra object of $\cCLMcan_k$.
\par \medskip
By Proposition~\ref{tensind}, the $\cCLMcan_k$-isomorphisms $\cL^{(C,C')}_\sD$ of  \eqref{funcmeas1}, for $! = \circ$, 
 induce by application of the functor $\colimit^\un_{C \in \cK(X),C'\in \cK(Y)}$
   an isomorphism 
\beq \label{LDacs1}
 \cL^{(X,Y)}_\sD    :  \sD^\circ_\acs(X \times Y,k) \iso   \sD^\circ_\acs (X,k) \wt^\un_k  \sD^\circ_\acs(Y,k) 
\eeq 
in $\cCLMcan_k$. 
For $! = \emptyset$ we \emph{a priori} only obtain  a morphism \eqref{tensfunct1}
 \beq \label{LDacs2}
 \cL^{(X,Y)}_\sD  :  \sD_\acs(X \times Y,k) \map{} \sD_\acs(X,k) \wt^\un_k  \sD_\acs(Y,k)  \;.
\eeq 
It follows from Corollary~\ref{bijmap} that $\sD^\circ_\acs (X,k) \wt^\un_k  \sD^\circ_\acs(Y,k)$ is a dense subset of  $\sD_\acs(X,k) \wt^\un_k  \sD_\acs(Y,k)$. So,  the morphism \eqref{LDacs2} is obtained by completion of  
the $k$-linear isomorphism \eqref{LDacs1}  in $\cCLMcan_k$ for a weaker topology. Hence it is an isomorphism of 
$\cCLMu_k$.  
\par \smallskip  
Similarly, by Corollary~\ref{ringtens1}, we may dualize  (both weakly and strongly) the $\cCLMcan_k$-isomorphism \eqref{prodfctsacs}  and, 
for $! = \emptyset,\circ$, we obtain isomorphisms
  \beq \label{prodmeaspro}
 \cL^{(X,Y)}_\sD :  \sD^!(X \times Y,k)   \iso \sD^!(X,k) \wt^\un_k  \sD^!(Y,k) \;.
\eeq 
Finally, by Lemma~\ref{ringtens} we may dualize  (both weakly and strongly) the injective morphism \eqref{prodfcts11} to obtain, for $! = \emptyset,\circ$, a morphism
  \beq \label{prodmeasunif}
 \cL^{(X,Y)}_\sD :  \sD^!_\unif(X \times Y,k)   \map{} \sD^!_\unif(X,k) \wt^\un_k  \sD^!_\unif(Y,k) \;.
\eeq 
%\beq 
%%\label{unifprodsquare}
% \;\;\;  \begin{tikzcd}[column sep=2.5em, row sep=2.5em] 
%\sD^!_\unif(X \times Y,k)  \arrow{r}{\eqref{prodmeasunif}} \arrow{d}{\eqref{uniffact2}}  &  
%\sD^!_\unif(X,k) \wt^\un_k  \sD^!_\unif(Y,k) \arrow{d}{\eqref{uniffact2}}\\ 
% \sD^!(X \times Y,k)   &  \arrow{l}{\eqref{prodfctsmor}} \sD^!(X,k) \wt^\un_k  \sD^!(Y,k)   \;\;.
%\end{tikzcd}
%\eeq

We summarize what we have proven. 
 \begin{prop} \label{Hopfstruct} Notation as in Remark~\ref{prodfctsstr}.  \ben
 \item
 For $X$ a $td$-space, the product morphism  
\beq \label{prodmor}
 \mu^{(X)}: \sC_? (X,k) \wt^\un_k \sC_? (X,k) \map{\cL^{(X,X)}}  \sC_?(X \times X,k) \map{\Delta_X^\ast} \sC_?(X,k)   
\eeq
of \lc, for $?=\emptyset, \acs$ weakly (resp. strongly)  transposes to a co-associative and co-commutative coproduct for $!=\acs,\emptyset$, respectively,
 \beq \label{coprodcontpro} 
\P_{\sD}^{(X)}:  \sD_!(X,k) \map{(\Delta_X^\ast)^\tu} \sD_!(X \times X,k)   \map{\cL^{(X,X)}_\sD} \sD_!(X,k)   \wt^\un_k \sD_!(X,k)  
\eeq
(resp. 
 \beq \label{coprodcontpro1} 
\P_{\sD}^{(X)}:  \sD^\circ_!(X,k) \map{(\Delta_X^\ast)^\tu} \sD^\circ_!(X \times X,k)   \map{\cL^{(X,X)}_\sD} \sD^\circ_!(X,k)   \wt^\un_k \sD^\circ_!(X,k) \; \mbox{)} 
\eeq
such that $\P_{\sD}^{(X)}(\delta_x) =\delta_x \wt^\un_k \delta_x$, for any $x\in X$.
In case $! = \acs$, the coproduct  $\P_{\sD}^{(X)}$ of \eqref{coprodcontpro} (resp. of \eqref{coprodcontpro1}) together with the augmentation 
$$ \veps_\sD^{(X)}:  \sD_\acs(X,k)  \map{} k \;\;,\;\; \mu \longmapsto \mu(X)$$
(resp. 
$$ \veps_\sD^{(X)}:  \sD_\acs^\circ(X,k)  \map{} k \;\;,\;\; \mu \longmapsto \mu(X)\;\mbox{)}\;,$$
determine  a $k$-coalgebra structure on the object $\sD_\acs(X,k)$ (resp. $\sD_\acs^\circ(X,k)$) of $\cCLMu_k$ (resp. of $\cCLMcan_k$).  In case $! = \emptyset$, the coproduct  $\P_{\sD}^{(X)}$ still exists but 
 fails to define a $k$-coalgebra structure on $\sD(X,k)$ (resp. on $\sD^\circ(X,k)$), for lack of an augmentation. 
 \item Let $(X,\Theta_\sJ)$ be a neat $td$-uniform space,  
as in \eqref{projlimunif}. Then the product morphism \eqref{prodmor}, for $?=\unif$, weakly (resp. strongly)  transposes to the co-associative and co-commutative coproduct \eqref{coprodcontpro} for $! =\unif$, such that $\P_{\sD}^{(X)}(\delta_x) =\delta_x \wt^\un_k \delta_x$, for any $x\in X$. That coproduct, together with the augmentation 
$$ \veps_\sD^{(X)}:  \sD_\unif(X,k)  \map{} k \;\;,\;\; \mu \longmapsto \mu(X)$$
(resp. 
$$ \veps_\sD^{(X)}:  \sD_\unif^\circ(X,k)  \map{} k \;\;,\;\; \mu \longmapsto \mu(X)\;\mbox{)}\;,$$
determine  a $k$-coalgebra structure on the object $\sD_\unif(X,k)$ (resp. $\sD_\unif^\circ(X,k)$) of $\cCLMu_k$ (resp. of $\cCLMcan_k$).
 \een
 \end{prop}
\end{subsubsection}
\end{subsection}
\begin{subsection}{Summary of functions vs. measures duality} 
We summarize the essentials of what we have proven in section~\ref{measdual}. 
 \begin{cor} \label{uniffact} Let $(X,\Theta_\sJ)$ be a neat $td$-uniform space,  
as in \eqref{projlimunif}.
 \par There is a natural sequence 
  \beq \label{uniffact1} 
 \sC_\acs(X,k) \longrightarrow  \sC_\unif((X,\Theta_\sJ),k)   \longrightarrow 
\sC(X,k)   \eeq
of associative and commutative (non-unital) $k$-algebra objects of $\cCLMu_k$. The second morphism 
is a morphism of unital $k$-algebras, as well. The first morphism of the sequence \eqref{uniffact1}  
 is a kernel in the category 
$\cCLMcan_k$,  all  morphisms are injective and the composite morphism has 
 dense set-theoretic image.
 \par
The weak (resp. strong) dual of the sequence \eqref{uniffact1}
is a sequence 
\beq \label{uniffact2}    
 \sD_\acs(X,k) 
\map{\inc}   \sD_\unif((X,\Theta_\sJ),k)   \longrightarrow  \sD(X,k)
 \eeq 
 (resp.
 \beq \label{uniffact3}    
 \sD^\circ_\acs(X,k) 
\map{\inc}   \sD^\circ_\unif((X,\Theta_\sJ),k)   \longrightarrow  \sD^\circ(X,k)\;\mbox{)}
 \eeq 
of coassociative and cocommutative (non-counital) $k$-coalgebra objects of $\cCLMu_k$ (resp. of $\cCLMcan_k$). All morphisms are injective and have dense set-theoretic image;
 the morphism $\inc$ is a morphism  of counital $k$-coalgebras, as well. 
 \par
The sequence \eqref{uniffact1} may be  reconstructed from \eqref{uniffact2} via
 \beq \label{canacs3final}  \begin{split}
 \sC_\acs(X,k) =  \sD(X,k)'_\weak &= \sD(X,k)'_\strong  \longrightarrow \\ \sC_\unif((X,\Theta_\sJ),k) &:= 
 \sD_\unif((X,\Theta_\sJ),k)'_\strong 
 \\ 
%\longrightarrow \sC_\alc(X,k) = \sD_{\alc,\weak} (X,k)'_\weak 
 \longrightarrow \sC(X,k) &= \sD_\acs(X,k)'_\weak
\end{split} \eeq
%while \eqref{uniffact2}  (resp. \eqref{uniffact3}) is the weak (resp. strong) dual of \eqref{uniffact1}.
%\beq \label{seqmeas1final} \begin{split}
% \sD_\acs(X,k)& := \; \sC(X,k)'_\weak  
%%\longrightarrow \sD_{\alc,\weak} (X,k):= \sC_\alc(X,k)'_\weak  
% \map{\inc} \sD_\unif((X,\Theta_\sJ),k) =  \sC_\unif((X,\Theta_\sJ),k)'_\weak \\ &  \longrightarrow \sD(X,k) := \sC_\acs(X,k)'_\weak \;.
%\end{split}
%\eeq
By \eqref{Dacsdef}
and the first line of 
 \eqref{Dacslimdiscr}, we have both  
\beq \label{surpriseacs}
\sD_\acs(X,k) = \colimit^\un_{C \in \cK(X)} (\sD(C,k) ,j_{C,C'})  = \limit_{D \in \sQ(X)} ( \sD_\acs(D,k), h_{D,D'}) \;.
\eeq
By \eqref{Dacsodef} and the second line of  \eqref{Dacslimdiscr}, we have both 
 \beq \label{Dacsodef}
\sD_\acs^\circ(X,k) = \colimit^\un_{C \in \cK(X)} (\sD^\circ(C,k), j_{C,C'})   = \limit^\can_{D \in \sQ(X)} ( \sD_\acs(D,k), h_{D,D'})  
\eeq
in $\cCLMcan_k$.
  \end{cor} 
%  \begin{rmk} \label{acsunif}
%The $k$-coalgebra 
%\beq \label{acsunif}
%\sD_\acs(X,k) =\sD_\unif ((X,\Theta_\univ),k)
%\eeq
%  \end{rmk}
% \begin{proof} We only need to prove the last assertion. The map \eqref{diracunif} 
%naturally factors through an injective $k$-linear map
%\beq {\bigoplus}_{x \in X} k \delta_x \longrightarrow  \sD_\acs(X,k)
%\eeq 
%and since \eqref{diracunif}  has dense image, so does the first morphism
%\beq \sD_\acs(X,k)  \longrightarrow   \sD_\unif((X,\Theta_\sJ),k)
%\eeq
%in \eqref{uniffact2}. The composite  morphism $\sD_\acs(X,k) \longrightarrow \sD_\pro(X,k)$ in \eqref{uniffact2} is the weak transpose of the morphism $\sC_\acs(X,k)  \longrightarrow  \sC(X,k)$  which has dense image, and is therefore injective. By definition  of $\sD_\pro(X,k)$ that composite  morphism  
%is the limit of the projective system of morphisms 
%$$\sD_\acs(X,k) \longrightarrow \sD (C,k)
%$$ for $C \in \cK(X)$, where $\delta_x \in \sD_\acs(X,k)$, for $x \in X$, is sent to $\delta_x$ if $x \in C$ and to $0$ otherwise. In particular it has dense image.  
%\par \medskip
%We now discuss the second morphism in   \eqref{uniffact2}, namely $ \sD_\unif((X,\Theta_\sJ),k)   \longrightarrow  \sD_\pro(X,k)$. For the moment we only know that it has dense image. It originates from 
%\beq \label{DtoC}
%\sD_\unif(D,k) ={\bigoplus}^\un_{d \in D}  k\delta_d  \longrightarrow  \sD(C,k) 
%\eeq
%taking $\limit_{C \in \cK(X)}$. In turn \eqref{DtoC} reduces to the map $\delta_d \longmapsto (f \longmapsto $
%\end{proof} 
From the discussion  we have obtained 
\begin{cor} \label{intpairing} For any $td$-space $X$, we have 
weak perfect pairings 
\beq \label{acsproint}
\circ : \sC_\acs(X,k)   \times  \sD(X,k)    \longrightarrow   k
\eeq  
and 
\beq \label{contweakint}
\circ : \sC(X,k)   \times  \sD_\acs(X,k)    \longrightarrow   k
\eeq
We also have strong 
perfect pairings 
\beq \label{acsproint}
\circ : \sC_\acs(X,k)   \times  \sD^\circ(X,k)    \longrightarrow   k
\eeq  
and 
\beq \label{contweakint}
\circ : \sC^\circ(X,k)   \times  \sD^\circ_\acs(X,k)    \longrightarrow   k \;.
\eeq
For any  uniform $td$ space $(X,\Theta_\sJ)$ we have a  strong-left  and weak-right perfect pairing
\beq \label{unifint}
\circ : \sC_\unif((X,\Theta_\sJ),k)    \times  \sD_\unif((X,\Theta_\sJ),k)   \longrightarrow   k \;.
\eeq 
Any of those  pairings  will be denoted by 
$$
(f,\mu) \longmapsto f \circ \mu =: \int_X f(x) \mu(x)
$$
and will be called  \emph{integration over $X$}. 
The previous integration pairings are compatible in an obvious sense with the sequence  \eqref{canacs3final}.
\end{cor}
\end{subsection}
\end{section}
%\begin{section}{Measures with values in $k$ on $td$ spaces: direct definition} \label{measdir}
\begin{section}{Measures: direct definition} \label{meastheory}
 \begin{defn} \label{measinterdef}  
 Let $X$ (resp. $(X,\Theta_\sJ)$) be an object of $td-\cS$ (resp. a neat object of $td-\cU$). 
We let $\Sigma(X)$ (resp. $\Sigma_D(X)$ for $D \in \sJ$) be the boolean ring
 of subsets $U \subset X$ which are \emph{clopen}, \ie both open and closed, 
 (resp. of the form $U =\pi_D^{-1}(V)$ for $V \subset D$) in $X$.  
 We regard the multiplicative monoid  $(\Sigma(X),\cdot = \cap)$ 
 (resp. $(\cK(X),\cdot = \cap)$) as a multiplicative submonoid of 
   $\sC(X,k)$ (resp. of  $\sC_\acs(X,k)$)   via $U \longmapsto \chi^{(X)}_U$. These embeddings extend to inclusions  of $k$-algebras on the previous multiplicative monoids
\beq \label{boole1}
 k[\cK(X)] \longrightarrow \sC_\acs(X,k) \;\;,\;\;  k[\Sigma(X)] \longrightarrow \sC(X,k) \;.
\eeq 
If $(X,\Theta_\sJ)$ is given, we set  $\Sigma_\sJ(X) := \bigcup_{D \in \sJ} \Sigma_D(X)$ and call its elements \emph{uniformly measurable subsets of $(X,\Theta_\sJ)$}.
  For any $D \in \sJ$ we have inclusions of rings $\Sigma_D(X) \leq \Sigma_\sJ(X) \leq \Sigma(X)$.  
Then $(\Sigma_\sJ(X),\cdot)$ is a multiplicative submonoid of 
   $\sC_\unif(X,k)$ and $U \longmapsto \chi^{(X)}_U$
    extends to  inclusions   of $k$-algebras  \eqref{unifcont}
   \beq \label{booleunif} 
  k[\Sigma_\sJ(X)] \longrightarrow \bigcup_{D \in \sJ} \sC_\unif(D,k) \longrightarrow \sC_\unif((X,\Theta_\sJ),k)\;.
\eeq
 \end{defn}
% \begin{rmk} \label{boole2} Let $(X,\Theta_\sJ)$ be as in Definition~\ref{measinterdef}. Then, $(\cK(X), \cdot = \cap)$ is a multiplicative monoid and, by $\mathit 3$ of Lemma~\ref{contcompl},
% for any $C \in \cK(X)$, $\chi^{(X)}_C \in \sC_\unif((X,\Theta_\sJ),k)$.  
% Therefore, similarly to \eqref{boole1}, we have an embedding 
% \beq \label{boole21}
% k[\cK(X)] \longrightarrow  \sC_\unif((X,\Theta_\sJ),k)\;.
%\eeq
% \end{rmk}
 The following proposition provides a more explicit alternative definition of the three spaces of measures which appear in \eqref{uniffact2}. 
 \begin{prop} \label{measinter}
 Let $X$ (resp. $(X,\Theta_\sJ)$) be an object of $td-\cS$ (resp. of $td-\cU$). 
   Then, 
   \ben                                                                                                                       
   \item  the inclusion $k[\Sigma(X)] \longrightarrow \sC(X,k)$ 
   %$k[\cK(X)] \longrightarrow \sC_\acs(X,k)$,  $k[\Sigma_\sJ(X)] \longrightarrow \sC_\unif(X,k)$ 
   of \eqref{boole1}
    %\eqref{booleunif} 
    is dense;
\item
the topological $k$-coalgebra $\sD_\unif((X,\Theta_\sJ),k)$   of uniform $k$-valued measures
  on $(X,\Theta_\sJ)$,  coincides with the set of maps  
\beq \label{setmeas}
\mu: \Sigma_\sJ(X) \longrightarrow k 
\eeq 
    such that  
 \medskip
 \par {\bf (UNIFMEAS)}
  \emph{For  any  $D \in \sJ$ and any disjoint subfamily $\{U_\alpha\}_{\alpha \in A}$ of  $\Sigma_D(X)$, 
the $A$-series $\sum_{\alpha \in A}\mu(U_\alpha)$
 converges unconditionally to $\mu(\bigcup_{\alpha \in A} U_\alpha)$.}
 \par  \medskip \noindent
The map \eqref{setmeas} corresponding to  $\mu \in \sD_\unif((X,\Theta_\sJ),k) = \sC_\unif((X,\Theta_\sJ),k)'_\weak$ is given
by the formula   
\beq \label{intX}
 U \longmapsto \mu(U) := \chi_U^{(X)} \circ \mu = \int_X \chi_U^{(X)}(x) \mu(x)  \;.
\eeq
 Conversely, any map \eqref{setmeas} satisfying {\bf (UNIFMEAS)} above
% is continuous for the topology induced by 
%$\sC_\unif((X,\Theta_\sJ),k)$, hence 
extends uniquely to an element of $\sD_\unif((X,\Theta_\sJ),k)$. 
 The coproduct 
 $$
\P_{\sD}^{(X)}:  \sD_\unif((X,\Theta_\sJ),k) \longrightarrow \sD_\unif((X,\Theta_\sJ),k) \wt_k^\un \sD_\unif((X,\Theta_\sJ),k)
 $$ 
 of \eqref{coprodcontpro}  
 is given
 by
\beq \label{coprmeas}
\P_{\sD}^{(X)}(\mu) (U \times V) = \mu (U \cap V) \;,
 \eeq
 for any $\mu \in \sD_\unif((X,\Theta_\sJ),k)$, and $U,V \in \Sigma_\sJ(X)$.
The topology  of $\sD_\unif((X,\Theta_\sJ),k)$ has a basis of open submodules 
 $\sU_{D,J}$ 
indexed by   ideals $J \in \cP(k)$ and  discrete quotients of $X$  $\pi_D: X \to D$  in $\sJ$, 
where 
\beq \label{unifeqradius}
\sU_{D,J} = \{ \mu \in  \sD_\unif(X,k)\,|\, \mu(U) \in J\,, \forall \,U \in \Sigma_D(X) \,\} \;.
\eeq
\item  
 the topological $k$-coalgebra $\sD_\acs(X,k) = \sC(X,k)'_\weak$  ($\sD^\circ_\acs(X,k) = \sC(X,k)'_\strong$) of   $k$-valued  measures with almost compact support 
on $X$   coincides, via the correspondence \eqref{intX},  with the $k$-module   
of
maps  
\beq
\mu : \Sigma(X) \longrightarrow k
\eeq
such that 
  \par \medskip
 {\bf (ACSMEAS)}
 \emph{For any disjoint subfamily $\{U_\alpha\}_{\alpha \in A}$ of  $\Sigma(X)$ such that 
 $U := \bigcup_{\alpha \in A} U_\alpha \in \Sigma(X)$, 
 the $A$-series $\sum_{\alpha \in A}\mu(U_\alpha)$
 converges unconditionally to $\mu (U)$ in $k$. } 
   \par  \medskip \noindent
 The topology of $\sD^\circ_\acs(X,k)$ is  the  naive $k$-canonical topology. The topology of  $\sD_\acs(X,k)$ 
may be described as follows. 
 For any $td$-partition $\sP \in \sQ(X)$  of $X$   and any $J \in \cP(k)$, we set 
 $$
 V_{\sP,J} := \{\mu \in \sD_\acs(X,k)\,|\, \mu(C) \in J\; ,\; \forall \; C \in \sP\,\} \;.
 $$
 Then, $V_{\sP,J}$ is a $k$-submodule of $\sD_\acs(X,k)$, and the family $\{V_{\sP,J}\}_{\sP, J}$, indexed by $\sQ(X) \times \cP(k)$, is a basis of open $k$-submodules of 
 $\sD_\acs(X,k)$.  
 \par
   The coproducts  of  \eqref{coprodcontpro} and   \eqref{coprodcontpro1} are given by formula 
  \eqref{coprmeas},  for any $U,V \in \Sigma(X)$.
  \item  The topological $k$-coalgebra $\sD(X,k)$ (resp. $\sD^\circ(X,k)$)   of   $k$-valued  measures on $X$
 coincides with the space of maps
 \beq
\mu: \cK(X) \longrightarrow k
 \eeq
 as in \eqref{intX}
   such that 
   \par \medskip
    {\bf (MEAS)}
 \emph{For any finite disjoint subfamily $\sF  \subset \cK(X)$ and for 
 $U :=  \bigcup_{C \in \sF} C$,  $\sum_{C \in \sF}\mu(C)=\mu(U)$.}
    \par \medskip \noindent
  The coproduct  of  \eqref{coprodcontpro} 
  is given by formula 
  \eqref{coprmeas},  for any $U,V \in \cK(X)$.    A basis of open $k$-modules for the $k$-linear topology of $\sD(X,k)$ is the family $\{ U_{\sF;I}\}_{\sF;I}$,  indexed by finite families $\sF \subset  \cK(X)$ and $I \in \cP(k)$, where
 \beq 
 U_{\sF,I} := \{ \mu \in \sD(X,k)\,|\, \mu(C) \in I\;,\;\forall \; C \in \sF \,\}\;.
 \eeq
 The topology of  $\sD^\circ(X,k)$ is the naive canonical topology. 
%\item
%  The topology of  $\sD_\acs(X,k)$ and of $\sD_\pro(X,k)$ (resp. of $\sD_\unif((X,\Theta_\sJ),k)$) has a basis of open submodules 
%  $\sU_J$ (resp. $\sU_{D,J}$)
%indexed by   ideals $J \in \cP(k)$ (resp. and  discrete quotient   $\pi_D: X \to D$ in $\sJ$)
%where 
%$$
%\sU_J = \{ \mu \in  \sD_\acs(X,k) \;\mbox{or}\; \sD_\pro(X,k) \,|\, \mu(U) \in J\,, \forall \,U \in \Sigma(X) \,\}
%$$
%(resp. 
%$$
%\sU_{D,J} = \{ \mu \in  \sD_\unif(X,k)\,|\, \mu(U) \in J\,, \forall \,U \in \Sigma_D(X) \,\}  \;\;\mbox{)} \;.
%$$
\item The first morphism of the sequence \eqref{uniffact2}   
is induced by the  restriction   of $\mu \in \sD_\acs(X,k)$, $\mu:\Sigma(X) \longrightarrow k$,  to $\Sigma_\sJ(X)$. The second morphism takes $\mu \in \sD_\unif((X,\Theta_\sJ),k)$ to the measure
\beq
\what{\mu} : \cK(X) \longrightarrow k \;\;,\;\; C \longmapsto \int_X \chi_C^{(X)} \mu \;,
\eeq
using the fact (see $\mathit 3$ of Lemma~\ref{contcompl}) that $\chi_C^{(X)} \in \sC_\unif((X,\Theta_\sJ),k)$.  
% $\mu:\Sigma_\sJ(X) \longrightarrow k$, to $\cK(X)$. 
%$U \in \Sigma(X)$, by writing $U$ as a disjoint union 
%$U = \bigcup_{\alpha \in A} U_\alpha$ of elements of $\Sigma_\sJ(X)$, and setting 
% \beq \label{measax} 
%\mu(U) := \sum_{\alpha \in A}\mu(U_\alpha) \;.
%\eeq
\een
 \end{prop}
 We omit the proof, since it is quite standard 
 \begin{prop} \label{almcompactsupp} Let 
 $$Res: \sD_\acs(X,k) \longrightarrow \sD(X,k)$$
be the composite morphism in the sequence \eqref{uniffact2}. 
 The set-theoretic image of $Res$ is the set of measures $\mu \in \sD(X,k)$ such that 
\par \medskip
{\bf(ACSMEASbis)} \emph{For any $I \in \cP(k)$, there exists  $C_I \in \cK(X)$ such that, for any $V \in \cK(X)$, $V \cap C_I =\emptyset$ implies
  $\mu(V) \in I$.} 
\end{prop}
\begin{proof} Assume $\mu \in \sD_\acs(X,k)$ but $Res(\mu)$ does not satisfy {\bf(ACSMEASbis)} for some $I \in \cP(k)$. Then 
for any $td$-partition $\sP$  
of $X$, $\mu(C) \notin I$, for an infinite subset of $C \in \sP$. So, there is a sequence $\{C_i\}_{i\in \N}$ in $\cK(X)$ such that 
\ben 
\item $\mu(C_i) \notin I$, for any $i \in \N$;
\item $U:= \bigcup_{i\in \N} C_i \in \Sigma(X)$; 
\item
$\sum_{i \in \N} \mu(C_i)$ does not converge unconditionally. 
\een
This is absurd, since $\mu \in \sD_\acs(X,k)$. \par \smallskip
%for any compact 
%  open $Z  \subset X$, $X - Z$ is clopen in $X$ so that it is a $td$-space.  We may find $V_1 \in \Sigma (X - Z)$ such that  $\mu(V_1) \notin I$. But $V_1$ is itself a $td$-space, hence a disjoint union of compact open subsets $\{C_\alpha\}_{\alpha \in A}$, one of which is such that $\mu(C_\alpha) \notin I$.
%So, we may  assume that $V_1$ itself is a compact open subset of $X$, and then   $Z_1 := Z \overset{\cdot}{\cup} V_1$ is a compact open subset of $X$ and we may repeat for $Z_1$ the argument just used for $Z$. We determine in this way a sequence of disjoint compact open subsets $V_1,V_2,\dots$ such that $\mu(V_i) \notin I$, for $i=1,2,\dots$. So, $\sum_{i\in \N} \mu(V_i)$ does not converge unconditionally. 
%Conversely, let $\mu \in \sD_\pro(X,k)$ and assume $\mu$ does not satisfy 
% {\bf (ACSMEAS)}. We show that $\mu$ cannot satisfy 
%  {\bf(ACSMEASbis)} either.  
%So, let $\{U_\alpha\}_{\alpha \in A}$ be any disjoint subfamily  of  $\Sigma(X)$ and assume 
% the $A$-series $\sum_{\alpha \in A}\mu(U_\alpha)$ does not 
% converge  unconditionally in $k$.  This means that there exists $I \in \cP(k)$ and an injective map $\N \to A$, $i \mapsto \alpha_i$ such that $\mu(U_{\alpha_i}) \notin I$, for any $i \in \N$. Let $Z$ be any compact open in $X$. Then there exists $i_0$ such that  $U_{\alpha_i} \subset X-Z$ for $i \geq i_0$. Since every $U_{\alpha_i}$ is a disjoint union of open compact subsets of $X$, for one such subset $C_i$, we must have $\mu(C_i) \notin I$. This contradicts  {\bf(ACSMEASbis)} for $\mu$. 
 Conversely, let $\mu \in \sD(X,k)$ and assume $\mu$   satisfies 
  {\bf(ACSMEASbis)}. We extend $\mu$ to $\ol{\mu}: \Sigma(X) \longrightarrow k$ as follows. For any $U \in \Sigma(X)$, $U$ is a $td$-space. Therefore $U$ admits a $td$-partition $\sP$. The  $\P$-series $\sum_{C \in \sP} \mu(C)$
  is  unconditionally convergent. We then set 
  $$
  \ol{\mu}(U) := \sum_{C \in \sP} \mu(C) \;.
  $$
We check that, if $\sQ$ is another $td$-partition of $U$,  
\beq \label{eqrefin}
\sum_{C \in \sP} \mu(C) = \sum_{C' \in \sQ} \mu(C') \;.
\eeq
 In fact,  the open compact subsets $C \cap C'$ form an open compact partition $\sP \cap \sQ$ of $U$ which refines the previous ones. We may then assume that $\sP$ is a refinement of $\sQ$, in which case  formula \eqref{eqrefin} is clear. This gives a well-defined extension $\ol{\mu}$ of $\mu$ to $\Sigma(X)$.  Finally, $\ol{\mu}$ satisfies   {\bf(ACSMEASbis)} by construction.
\end{proof}

  \begin{rmk} \label{counterunif} In contrast to point $\mathit 1$ of Proposition~\ref{measinter}, $k[\cK(X)]$ (resp. $k[\Sigma_\sJ(X)]$) is not in general dense  via $U \longmapsto \chi^{(X)}_U$  in $\sC_\acs(X,k)$
 (resp. in $\sC_\unif(X,k)$).  \par
In fact, let us assume that $X = D$  is discrete (and endowed  with the discrete uniformity) and that $k$ is discrete. Then 
$$\sC_\unif(D,k)=\sC_\acs(D,k) = (\prod_{y \in D} k \chi_y^{(D)})^\dis$$
and the density statement would imply the equality 
$$k[\fP(D)] = (\prod_{y \in D} k \chi_y^{(D)})^\dis
$$
which is clearly false if $D$ is infinite. 
\end{rmk}
\end{section} 
\begin{section}{Commutative $td$-groups} \label{td-groups}
\begin{subsection}{$td$-groups}
We specialize here our previous results  to group objects of the category of $td$-spaces. Second countable such groups are already of high interest \cite{wesolek1}, \cite{wesolek2}.  
A proof of the next two lemmas can easily derived from \cite[\S 2]{ngo} 
and goes back to David van Dantzig around 1930. 
%One should however take into account our more restrictive definition of a $td$-space (\cf Remark~\ref{diffNGO}). 
\begin{lemma} \label{lczdgp} Let $G = (G,\cdot)$ be a  topological group.
The following   conditions on $G$ are equivalent. 
\ben
\item $G$ is a  $td$-space;
\item $G$ admits a basis of neighborhoods of $1_G$ consisting of compact open subgroups;
\item $G$ admits an open profinite subgroup. 
%\item $G$ is the projective limit of a cofiltered projective system of discrete  groups  and surjective homomorphisms with finite kernel.
\een
\end{lemma}
\begin{proof} We  observe that  (2) $\Rightarrow$ (1) because if $H$ is a compact open subgroup of $G$,  then $H$ is a Stone space and  the partition of $G$ in left $H$-cosets shows that $G$ is a sum of Stone spaces. So, Theorem~\ref{STSspaces} shows that $G$ is a $td$-space.  Now, clearly (3) $\Rightarrow$ (2).  To prove  (1) $\Rightarrow$ (3) (which is the only non-trivial part) we may follow the proof of  \cite[\S 2]{ngo}.    
\end{proof}
\begin{defn}\label{tdgp} A \emph{$td$-group} is a topological group $G$ satisfying the  equivalent conditions of Lemma~\ref{lczdgp}.
\end{defn}
\begin{lemma} \label{lczdgp2} A compact $td$-group is a profinite group, and conversely.
\end{lemma} 
 Part $\mathit 2$ of Lemma~\ref{lczdgp} indicates that we may restrict to  topological groups $G$ endowed with a \emph{group topology}, namely for which
the family $\cP(G)$ of open subgroups is a basis of neighborhoods of $1_G$. Then, the \emph{left uniformity} on $G$ is the uniformity  for which a basis of uniform covers is given by partitions 
$$G/P = \{gP\,|\, g \in G\,\}
$$  in left cosets  for any given open subgroup $P \in \cP(G)$; left multiplications are uniform automorphisms of the 
\emph{left uniform group}
$(G, \{G/P\}_{P \in \cP(G)})$.  
Similarly for the \emph{right uniformity} on $G$, partitions $P\backslash G$, and right  multiplications. The map $x \mapsto x^{-1}$ is a uniform antiisomorphism  
$$(G, \{P\backslash G\}_{P \in \cP(G)}) \iso (G, \{G/P\}_{P \in \cP(G)})
\;.
$$ 
Our category of topological groups admits limits and colimits.   
A topological group $G$ is separated if and only if $\bigcap_{H \in \cP(G)} H = \{1_G\}$. The left uniform group $G$ is complete  if and only if, as a uniform space,  
$G = \limit_{H \in \cP(G)} G/H$, where the set of left cosets $G/H$ is given the discrete  uniformity. Equivalently, the right uniform group $G$ is complete  if and only if $G = \limit_{H \in \cP(G)}  H \backslash G$, where  $H \backslash G$ is given the discrete  uniformity. So $G$ is complete for the left uniformity if and only if it is so for the right one. We then say that the topological group $G$ is \emph{complete}; this implies that $G$ is   separated.  By default, we will equip a topological group with its left uniformity. 
By Lemma~\ref{sts complete}  
 any $td$-group  is complete. 
 \end{subsection}
\begin{subsection}{Hopf algebra structures on commutative $td$-groups}
\emph{From now on we will  deal exclusively with commutative $td$-groups.}  
On such a group the left and right uniformities coincide, and will be understood. We will use the name of \emph{ball of radius $H$} for any coset $gH$, for $g \in G$ and $H \in \cP(G)$.  
\begin{prop} \label{convprod} 
Let $G$  be a commutative  $td$-group.  
\ben
\item 
 $\sC(G,k)$ is 
  a Hopf algebra  object of $\cCLMu_k$. 
   The weak dual 
$\sD_\acs(G,k)$ of $\sC(G,k)$
is a Hopf algebra  object of $\cCLMu_k$. For $\mu,\nu \in  \sD_\acs(G,k)$, and $U \in \Sigma(G)$,
\beq \label{prodmeas1}
\mu_\sD(\mu \otimes \nu)(U) = (\mu \times \nu)(\{(g,h) \in G \times G\,|\,gh \in U\,\})\;.
\eeq
The weak dual of the Hopf $k$-algebra object $\sD_\acs(G,k)$ of $\cCLMu_k$ is the Hopf $k$-algebra object $\sC(G,k)$ of $\cCLMu_k$. The strong dual  of $\sD_\acs(G,k)$ is the Hopf $k$-algebra object $\sC^\circ(G,k) = \sC(G,k)^\can$ of $\cCLMcan_k$. 
 \item $\sD_\unif (G,k)$ (resp. $\sD^\circ_\unif (G,k)$) is a Hopf $k$-algebra object of  $\cCLMu_k$ (resp. of $\cCLMcan_k$).
 The $k$-algebra object $\cC_\unif(G,k)$ of $\cCLMcan_k$ is the strong dual of both $\sD_\unif(G,k)$ and $\sD^\circ_\unif (G,k)$.  It is a Hopf $k$-algebra object of $\cCLMcan_k$.  

\item We have a  weak  perfect duality pairing  of Hopf $k$-algebra objects of $\cCLMu_k$
$$
\circ : \sC(G,k) \times \sD_\acs(G,k) \longrightarrow k \;,
$$
%$$
%\circ : \sC_\acs(G,k) \times \sD(G,k) \longrightarrow k \;,
%$$
 a left-strong right-weak perfect  duality pairing  of Hopf $k$-algebra objects of $\cCLMu_k$
$$
\circ : \sC_\unif(G,k) \times \sD_\unif(G,k) \longrightarrow k \;,
$$
and a strong perfect  duality pairing  of Hopf $k$-algebra objects of $\cCLMcan_k$
$$
\circ : \sC_\unif(G,k) \times \sD^\circ_\unif(G,k) \longrightarrow k \;.
$$
Any of the previous pairings is  denoted by 
$$
(f,\mu) \longmapsto  f \circ \mu = \int_G f(x)  \mu (x)  \;,
$$
and called \emph{integration pairing} over $G$. The product of measures $\mu_G(\mu \otimes \nu) =:\mu \nu$ is called \emph{convolution} and may be described 
as 
\beqa \begin{split}
 \int_G f(x)  d (\mu  \nu) (x) = (\mu \wt^\un_k \nu) &\circ \P_G(f) = \int_{G \times G} \P_G(f)(x,y) \; d (\mu \times \nu)(x,y)  = \\
 &\int_{G \times G} f(x y) d\mu (x) d\nu (y)
 \; . \end{split}
\eeqa
\een
\end{prop}
\begin{proof} $\mathit 1.$ The map $\Delta_G^\ast: f(x) \longmapsto f(x y)$ induces, via the $\cCLMu_k$-isomorphism \eqref{prodfctsmor}, a coproduct  
$$
\P_\sC: \sC (G,k) \map{\Delta_G^\ast}  \sC (G \times G,k) \map{(\cL^{(G,G)})^{-1}} \sC (G,k) \wt^\un_k \sC (G,k)
$$
which, together with the counit map 
$$
\veps_\sC : \sC (G,k) \longrightarrow k  
$$
such that $\veps_\sC(f) = f(1_G)$, and the inversion $\rho_\sC(f)(x) = f(x^{-1})$, makes 
$$\sC(G,k) = (\sC(G,k), \mu_\sC,1_\sC,\P_\sC,\veps_\sC,\rho_\sC)$$
  a Hopf algebra  object of $\cCLMu_k$. By weak transposition we obtain, via the $\cCLMcan_k$-isomorphism \eqref{LDacs2}
$$
 \cL^{(G,G)}_\sD  :  \sD_\acs(G \times G,k) \iso\sD_\acs(G,k) \wt^\un_k  \sD_\acs(G,k) \;,
$$
a product map  
$$
\mu_\sD: \sD_\acs (G,k) \wt^\un_k \sD_\acs(G,k) \map{(\cL_\sD^{(G,G)})^{-1}} \sD_\acs (G \times G,k)  \map{(\Delta_G^\ast)^\tu}  \sD_\acs (G,k) \;,
$$
which, together with the  unit map  
$$
1_\sD :  k \longrightarrow  \sD_\acs(G,k) \;,\; a \longmapsto a \delta_{1_G}\;, 
$$ 
and the inversion $\rho_\sD(\mu)(U) = \mu(U^{-1})$, where for any $U \in \Sigma(G)$, $U^{-1} = \{g^{-1}\}_{g\in U}$, makes 
$$\sD_\acs(G,k) = (\sD_\acs(G,k), \mu_\sD,1_\sD,\P_\sD,\veps_\sD, \rho_\sD)$$   
  a Hopf algebra  object of $\cCLMu_k$. \par
  Formula \eqref{prodmeas1} is easily checked. 
The rest of comma $\mathit 1$ is clear.
\par \medskip $\mathit 2$.  
 We know from $\mathit 2$ of Proposition~\ref{Hopfstruct} that $\sD_\unif (G,k)$ (resp.  $\sD^\circ_\unif (G,k)$) is a counital $k$-coalgebra object of $\cCLMu_k$ (resp. of $\cCLMcan_k$). We show it carries a $k$-algebra structure, as well. For any $H \in \cK(G)$, $G/H$ is a discrete group and the 
construction of the previous point gives to $\sD_\unif(G/H,k) = \sD_\acs(G/H,k)$  the structure of Hopf $k$-algebra object of  $\cCLMcan_k$. In particular, there are product morphisms
$$\mu_{\sD,H}: \sD_\unif (G/H,k) \wt^\un_k \sD_\unif (G/H,k)  \map{} \sD_\unif (G/H,k) \;.
 $$
By \eqref{measunifdef} 
$$
  \sD_\unif (G,k) = \limit_{H \in \cK(X)}  \sD_\acs(G/H,k)   
$$
(resp. by \eqref{measunifodef} 
 $$ \sD^\circ_\unif (G,k) = \limit^\can_{H \in \cK(X)}  \sD_\acs(G/H,k)   \in \cCLMcan_k \; \mbox{)}\;.
  $$
  This shows that 
  $$
   \sD_\unif (G,k)'_\strong =  \sD^\circ_\unif (G,k)'_\strong =  \sC_\unif(G,k)
  $$
  so that, by \eqref{tensdual?} for $? =\strong$,

From  \eqref{tensfunct2} we deduce the existence of a natural morphism
  $$  \sD_\unif (G,k) \wt^\un_k \sD_\unif (G,k) \map{} \limit_{H \in \cK(X)}  ( \sD_\unif (G/H,k) \wt^\un_k \sD_\unif (G/H,k) )\;.
  $$
  The product morphisms $\mu_{\sD,H}$ induce
% $$\mu_\sD: \sD_\unif (G/H,k) \wt^\un_k \sD_\unif (G/H,k)  \map{} \sD_\unif (G/H,k)
% $$ 
 a morphism 
 $$
 \limit_{H \in \cK(X)}  ( \sD_\unif (G/H,k) \wt^\un_k \sD_\unif (G/H,k) ) \map{}  \limit_{H \in \cK(X)}    \sD_\unif (G/H,k)
 $$
 hence all in all  a morphism
 $$
 \mu_\sD: \sD_\unif (G,k) \wt^\un_k \sD_\unif (G,k) \map{}  \sD_\unif (G,k) \;.
 $$
 We conclude that $\sD_\unif (G,k)$ is a $k$-algebra object of $\cCLMu_k$.

Then, using \eqref{prodmeasunif}, we conclude that $\sD_\unif (G,k)$ (resp. $\sD^\circ_\unif (G,k)$) is a Hopf $k$-algebra object of  $\cCLMu_k$ (resp. of $\cCLMcan_k$).
Now, $\cC_\unif(G,k)$ is the strong dual of $\sD_\unif(G,k)$, and by \eqref{tensdual?} for $? =\strong$, it is a Hopf $k$-algebra object of $\cCLMcan_k$, too.   
\end{proof} 
\end{subsection}

\begin{subsection}{Invariant endomorphisms} \label{invsect} 
 We recall in this section some properties of duality between algebra and coalgebra objects in a monoidal category of (topological)  $k$-modules. We will not list precisely the assumptions needed and will leave to the reader the task of verifying that our statements hold true in the applications we present. 
\par \medskip
 Let $(R,\P_R,\veps_R)$ be any cocommutative $k$-coalgebra. An element $h \in \End_k(R)$  is \emph{invariant} if 
 \beq
 \P_R h = (\id_R \otimes h) \P_R = (h \otimes \id_R) \P_R \;.
 \eeq
 In particular, by composition to the left with $\id_R \otimes \veps_R$ or with $\veps_R \otimes \id_R$,
 \beq \label{dualop}
 h = (\id_R \otimes \veps_R h) \P_R =  (\veps_R h \otimes \id_R ) \P_R \;.
 \eeq
 If $h,h' \in \End_k(R)$ are invariant, so are their composition products $hh'$  and $h'h$; moreover, $h = \id_R$ is invariant.  So, the subset $\Inv_k(R) \subset \End_k(R)$ is a  $k$-subalgebra for the composition product.
Notice that in \eqref{dualop} $\veps_R h \in D = \hom_k(R,k)$, the dual $k$-algebra of $R$. Conversely, for any $d \in D$, the formula 
 \beq \label{dualop2}
 \ol{d} = (\id_R \otimes d) \P_R =  (d \otimes \id_R ) \P_R \;, 
 \eeq
 defines an element $\ol{d} \in \Inv_k(R)$. We conclude with the main result of this introductory section.
\begin{thm} \label{invariant} Let $R$ be a cocommutative $k$-coalgebra and let $D$ be the dual (commutative) $k$-algebra. The map 
\begin{align*}
D &\longrightarrow \Inv_k(R) \\
d & \longmapsto \ol{d} 
\end{align*}
where $\ol{d}$ is defined by \eqref{dualop2}, is a homomorphism of $k$-algebras, for the multiplication in $\Inv_k(R)$ induced by the composition product of $\End_k(R)$.
\end{thm}
We can apply Theorem~\ref{invariant} to any 
  integration pairing 
$$
\circ : \sC  \times \sD  \longrightarrow k \;,
$$
appearing in Proposition~\ref{convprod}. 
We obtain continuous morphisms
\beq \label{invariant1} \begin{split}
  \sC \longrightarrow \Inv_k (\sD)    \;\; &\mbox{(\,resp.} \;\;  \sD  \longrightarrow \Inv (\sC)  \; \mbox{)} \;,
\\
f \longmapsto \ol{f}     \;\; &\mbox{(\,resp.}\;\; \mu  \longmapsto \ol{\mu} \; \mbox{)}\;,
\end{split}
\eeq
where $\Inv_k$ denotes the group of continuous invariant $k$-linear automorphisms of $\sD$
 characterized  by 
\beq \label{invariant2}
\int_G h(x) d (\ol{f}(\mu))(x) = \int_G h(x) f(x) d \mu (x) \;\; \mbox{(\,resp.}\;\; 
\int_G  (\ol{\mu} (f))(x) d \nu(x) =  \int_G f(x) d(\mu \nu)(x) \; \mbox{)} \;,
\eeq
for any $h,f \in \sC$ and $\mu,\nu \in \sD$. 
\par \medskip
\begin{cor} \label{deltaprod}
For any $g_1,g_2 \in G$  we have the equality of Dirac masses in $\sD_\acs(G,k)$ 
\beq \label{deltahom}
 \delta_{g_1}  \delta_{g_2}  = \delta_{g_1g_2} \;.
\eeq
\end{cor}
\begin{proof}
To verify \eqref{deltahom} we consider the duality pairing 
$$
\circ : \sC(G,k) \times \sD_\acs(G,k) \longrightarrow k \;,
$$
and observe that the action $\ol{\delta_g}$ on $f \in \sC(G,k)$ is given by 
$$
(\ol{\delta_g}(f))(x) = f(gx) \;.
$$
Therefore, 
$$
 \ol{\delta_{g_1g_2}}(f)(x) = f(g_1g_2x) =  \ol{\delta_{g_1}} ( \ol{\delta_{g_2}} (f)) (x)
$$
so that 
$$\ol{\delta_{g_1g_2}} =  \ol{\delta_{g_1}} \ol{\delta_{g_2}} $$
in $\Inv_k(\sC(G,k))$, and therefore, by Theorem~\ref{invariant}, we obtain \eqref{deltahom}. 
\end{proof}
We now embed the group $k$-algebra $k[G]$ in $\sD_\acs(G,k)$ via $g \longmapsto \delta_g$  
and  sequence \eqref{uniffact2} may be completed into 
\beq \label{uniffactgroup}  
 k[G] \longrightarrow  \sD_\acs(G,k) 
 \longrightarrow   \sD_\unif(G,k)   \longrightarrow  \sD(G,k)
 \eeq
\begin{prop} \label{explunifmeas} \ben \item $\sD_\acs(G,k)$ is the completion of the ring $k[G]$ in 
the $k[G]$-linear topology of uniform convergence on balls of any radius;
% \ie in its naive $k$-canonical topology;
\item $\sD_\unif(G,k)$  is the completion of the ring $k[G]$ in the $k[G]$-linear topology  of uniform convergence on balls of the same radius;
\item $\sD(G,k)$  is the completion of $k[G]$ in the $k[G]$-linear topology  of simple convergence on balls.  
%\ie in the topology with basis of open ideals $\{J[H]\}_{J,H}$, for $J \in \cP(k)$ and $H \in \cP(G)$. 
\een
\end{prop}  
\begin{proof}
Clear. 
\end{proof}
%\begin{cor} \label{explunifmeascor}
%Assume $k$ is $I$-adic for a finitely generated ideal $I$ and $G$ is the underlying topological group of a finite free $k$-module $\bigoplus_{i=1}^n k g_i$ equipped with its naive 
%$k$-canonical topology. Let $x_i =g_i-1 \in k[G]$, for $i=1,\dots,n$, so that $ k[G] = k[x_1,\dots,x_n]$. Then 
%\ben \item $\sD_\acs(G,k) = k\{x_1,\dots,x_n\}$ equipped with its $I$-adic topology;
%\item $\sD_\unif(G,k)$  is the completion of the ring $k[x_1,\dots,x_n]$ in the $(I,x_1,\dots,x_n)$-adic topology;
%\item $\sD_\pro(G,k)$  is the completion of $k[G]$ in the $k[G]$-linear topology  of simple convergence on balls.   
%\een
%\end{cor}
In the next sections we will specialize the previous discussion to the case of $G$ = the profinite group $(\Z_p,+)$, equipped with the $p$-adic topology and $G$ = the $td$-group $(\Q_p,+)$ containing the previous profinite group $(\Z_p,+)$ as an open subgroup. The role of the ring $k$ will be played by the $p$-adic ring $\Z_p$ or by the discrete field $\F_p$.
\end{subsection}
\end{section}
\begin{section}{$p$-adic Fourier theory  on   $\Z_p$} \label{Integrations}
\begin{subsection}{Mahler-Amice theory on $\Z_p$}\label{MAss}
We describe here $\sC(\Z_p,k)$ and its weak dual $\sD(\Z_p,k)$, in case $k= \F_p$ or $=\Z_p$. 

\begin{notation} We use the notation $\Delta_a$ (resp. $\delta_a$) for the Dirac mass concentrated in $a \in \Q_p$ with values in $\Z_p$ (resp. in $\F_p$).
\end{notation} 

We proved in full generality  that
\ben
\item $\sC(\Z_p,k)$ = Hopf algebra over $k$   of continuous $k$-valued functions with the topology of uniform convergence on $\Z_p$ (if $k=\F_p$ this is the discrete topology; if $k=\Z_p$ it coincides with the $p$-adic topology);
\item $\sD(\Z_p,k)$ = the weak dual of $\sC(\Z_p,k)$, with weak perfect dual pairing 
\beq \label{mahleramice0}
(f,\mu) \longmapsto f \circ \mu = \int_{\Z_p} f(x) d\mu(x) \;.
\eeq
Then $\sD(\Z_p,k)$ is a Hopf algebra over $k$ and it 
%may be seen as  the set of $\Z_p$-valued functions $\mu$ on the family of ``balls'' $a +p^n\Z_p$ in $\Z_p$, for $a \in \Z$ and $n \in \N$ such that, if $a +p^n\Z_p$ is the disjoint union of $a_0 +p^{\Z_p$. 
 is equipped with the  \emph{natural topology}, \ie the  topology of uniform convergence on balls of equal radius. As a topological ring $\sD(\Z_p,k)$ is linearly topologized and  complete.
 \item $\sD^\circ(\Z_p,k)$ = the strong dual of $\sC(\Z_p,k)$, with strong perfect dual pairing \eqref{mahleramice0}
Then $\sD^\circ(\Z_p,k)$ is a Hopf algebra over $k$ and it 
%may be seen as  the set of $\Z_p$-valued functions $\mu$ on the family of ``balls'' $a +p^n\Z_p$ in $\Z_p$, for $a \in \Z$ and $n \in \N$ such that, if $a +p^n\Z_p$ is the disjoint union of $a_0 +p^{\Z_p$. 
 is equipped with the  \emph{topology of uniform convergence on all balls}. It coincides with $\sD(\Z_p,k)$, equipped  with the $k$-canonical topology.  
\een
The case $k=\Z_p$ is the classical case of Mahler and Amice, while the case $k=\F_p$ is simpler and follows easily from the classical theory. 
So, we now make   explicit, following Mahler and Amice, the structure of $\sC(\Z_p,k)$ and $\sD(\Z_p,k)$, for $k=\Z_p$. We will then point out the simplifications to apply when $k=\F_p$. 
\begin{prop} 
\label{mahleramice}
\ben
\item
For any $f \in \sC(\Z_p,\Z_p)$  and  $n \in \N$, we define $f^{[n]} \in \sC(\Z_p,\Z_p)$
recursively by  
\beq
\label{mahleramice0}
f^{[1]} (x) = f(x+1)-f(x)\;\;\mbox{and}\;\; f^{[n+1]}   = (f^{[n]})^{[1]}\;.
\eeq 
For any $f \in  \sC(\Z_p,\Z_p)$, we have
\beq \label{mahleramice11}
f(x) =  \sum_{n=0}^\infty f^{[n]}(0) {x \choose n}    \;,
\eeq
where the sequence $n \longmapsto f^{[n]}(0)$ converges to $0$. In particular, 
\beq \label{mahleramice112} \sC(\Z_p,\Z_p) = {\bigoplus}^\un_{n \in \N} \Z_p {x \choose n} = {\bigoplus}^\can_{n \in \N} \Z_p {x \choose n} \;,  
\eeq
where $(-)^\can$ means completion (of the algebraic direct sum) in the $p$-adic topology. 
\item The coproduct $\P_\sC$ of $\sC(\Z_p,\Z_p)$ is the continuous $\Z_p$-algebra morphism such that 
$$
\P_\sC  {x \choose n} = \sum_{i+j=n}  {x \wt 1 \choose i}  {1 \wt x \choose j} \in \sC(\Z_p,\Z_p) \wt_{\Z_p} \sC(\Z_p,\Z_p) \;,
$$
while the counit $\veps_\sC$ satisfies 
$$
\veps_\sC {x \choose n} = \delta_{n,0} \;.
$$
\item 
Let 
$$T = \Delta_1 - \Delta_0 \in \sD(\Z_p,\Z_p)\;.$$ 
Then
\beq \label{mahleramice1}
{x \choose i}\circ T^j = \int_{\Z_p} {x \choose i} dT^j (x)=  \delta_{i,j} \;,
\eeq
for any $i,j \in \N$.  In particular, $T^n \to 0$  in $\sD(\Z_p,\Z_p)$  as $n \to \infty$.
For any $f \in  \sC(\Z_p,\Z_p)$, the expansion \eqref{mahleramice11} coincides with 
\beq \label{mahleramice111}
f(x) =    \sum_{n=0}^\infty  \l(\int_{\Z_p} f(y) dT^n(y)\r) {x \choose n} 
\eeq
\item
For any  $\mu \in  \sD(\Z_p,\Z_p)$, 
\beq \label{mahleramice12}
\mu =  \sum_{n=0}^\infty T^n \int_{\Z_p} {x \choose n} d\mu(x)  \;.
\eeq 
\item 
 $\sD(\Z_p,\Z_p)$ is the completion of its subring $\Z[T]$ with the topology of uniform convergence 
 on balls of the same radius  in $\Z_p$. There is a canonical identification as rings  
\beq \label{Iwasawa1}
\sD(\Z_p,\Z_p) = \Z_p[[T]] \;.
\eeq 
%\item The coproduct $\P_\sD$ of $\sD(\Z_p,\Z_p)$ is the continuous $\Z_p$-algebra morphism such that 
%$$
%\P_\sD T = 1 \wt T + T \wt 1 + T \wt T \in \Z_p[[T]] \wt_{\Z_p} \Z_p[[T]]
%\;,
%$$
%while the counit $\veps_\sD$ satisfies 
%$$
%\veps_\sD T = 0 \;.
%$$
\een
\end{prop}  
\begin{proof} $\mathit 1.$ 
We adapt to the present situation the proof described by Colmez \cite[\S1.2]{colmez2}. 
Let $B_N$ be  the $p^N \times p^N$ lower-triangular unipotent matrix 
 $$B_N := (b_{n,a} = {n \choose a})_{n,a = 0,1,\dots,p^N-1} \in GL(p^N,\Z)\;.$$
 Let $f \in \sC(\Z_p,\Z_p)$. Then, for any $N =0,1,\dots$, we consider the column vectors
 $$F_N := (f(0),f(1),\dots,f(p^N-1))^\tu \;\;,\;\; C_N := (a^{(N)}_0,a^{(N)}_1,\dots,a^{(N)}_{p^N-1})^\tu \;,$$
 where $a^{(N)} \in \{0,1,\dots, p^N-1\}$  and $C_N$ is uniquely defined by the condition 
 $$
 B_N C_N \equiv F_N \mod p^{N+1} \;.
 $$
 Notice that, for any $N$,
 $$
 (a^{(N+1)}_0,a^{(N+1)}_1,\dots,a^{(N+1)}_{p^N-1}) \equiv (a^{(N)}_0,a^{(N)}_1,\dots,a^{(N)}_{p^N-1}) \mod p^{N+1}
 $$
% $$
%  (a^{(N+1)}_{p^N},a^{(N+1)}_{p^N +1},\dots,a^{(N+1)}_{p^{N+1}-1})
% $$
so that, for any $j =0,1,\dots$, the sequence $N \to a^{(N)}_j$ converges in $\Z_p$. We denote by $a_j$ its 
limit and observe that, for any $N =0,1,2,\dots$,
$$
 B_N (a_0,a_1,\dots,a_{p^N-1})^\tu \equiv F_N \mod p^{N+1} \;.
$$
We have obtained the \emph{Mahler expansion} 
\beq \label{mahlerexp}
f(x) = \sum_{j=0}^\infty a_j {x \choose j} \;,
\eeq
of $f$, which,  for any $x =0,1,2,\dots$,  is a finite sum. 
%For any $f \in \sC(\Z_p,\Z_p)$  and  $n \in \N$, we define $f^{[n]} \in \sC(\Z_p,\Z_p)$
%recursively by  
%$$f^{[1]} (x) = f(x+1)-f(x)\;\;\mbox{and}\;\; f^{[n+1]}   = (f^{[n]})^{[1]}\;.$$  
If $P_j(x)$, for $j=0,1,\dots$,  denotes the binomial function $x \longmapsto {x \choose j}$, 
we have 
\beq \label{ortogbinom} 
P_j^{[h]} = P_{j-h}\;\;\mbox{if}\;\; h \leq j \;,\;\;\mbox{while}\;\;   P_j^{[h]} = 0\;\;\mbox{if}\;\; h>j \;.
\eeq
Then, for $f \in \sC(\Z_p,\Z_p)$ as in \eqref{mahlerexp}
$$
f^{[h]}(x) = \sum_{j=0}^\infty a_{j+h}{x \choose j}
$$
and $f^{[h]}(0) = a_h$. Then, for any $f \in \sC(\Z_p,\Z_p)$,  \eqref{mahlerexp} 
%the Mahler expansion of any $f \in \sC(\Z_p,\Z_p)$ 
may be rewritten as \eqref{mahleramice11}.
%\beq \label{mahlerexp1}
%f(x) = \sum_{j=0}^\infty f^{[j]}(0) {x \choose j} \;. 
%\eeq
\begin{lemma} \label{infmahler} For any $f \in \sC(\Z_p,\Z_p)$ let
$$v_{\inf}(f) = \inf \{v(f(x))\,|\, x \in \Z_p\} \;.$$
Then there is $h \in \N$ such that 
$$v_{\inf}(f^{[p^h]}) \geq  v_{\inf}(f)  + 1\;.$$
\end{lemma}
\begin{proof}(of the Lemma) Exercise or see 
\cite[Lemme 1.14]{colmez2}.
\end{proof}
This concludes the proof of $\mathit 1.$
\par \medskip
$\mathit 2.$ 
follows from the formula
$$
  {x +y \choose n} = \sum_{i+j=n}  {x   \choose i}  {y \choose j}  \;.
$$
\par 
$\mathit 3.$ could be proven via Amice transform \cite[\S 2.3]{colmez2} but we prefer to base our proof on the previous section \ref{invsect}. We need to check that, for any $f \in \sC = \sC(\Z_p,\Z_p)$  and for any $n \in \N$, 
$$
 f^{[n]} = \ol{T^n}(f)  \;.
$$
This follows from Corollary~\ref{deltaprod}, since $T = \Delta_1-\Delta_0$ and $\mu \longmapsto \ol{\mu}$ is a morphism of $\Z_p$-algebras. So, 
%This may be proven by induction, using the definition \eqref{dualop2}. In fact, for $n=1$ and any $x \in \Z_p$,
%$$
%(\ol{T}(f))(x) = (\id_\sC \otimes (\Delta_1-\Delta_0))f(x \otimes 1 + 1 \otimes x) = f(x+1) - f(x) =  f^{[1]}(x)\;.
%$$
%$$\int_{\Z_p}  (\ol{T} (f))(x) d \Delta_0(x) =  \int_{\Z_p} f(x) d(T)(x) = \int_{\Z_p} f(x) d(\Delta_1-\Delta_0)(x) = f(1)-f(0)$$
$$
   f \circ T^n  =   f^{[n]}(0)\;.
$$
So, from \eqref{ortogbinom},  
\beq \label{mahleramice13}
{x \choose i}\circ T^j =  P_i^{[j]}(0) =  \delta_{i,j} \;,
\eeq
for any $i,j \in \N$.  
\par 
$\mathit 4.$ Formula \eqref{mahleramice12} follows from orthogonality of the subset $\{T^j\}_{j\in \N}$ of  $\sD(\Z_p,\Z_p)$
to the 
topological basis $\{{x \choose i}\}_{i\in \N}$ of $\sC(\Z_p,\Z_p)$.
\par 
$\mathit 5.$ 
%It follows from  \eqref{mahleramice112} and \eqref{mahleramice1}  that  \eqref{Iwasawa1} holds and that $T^n \to 0$ in the weak topology of the 
%dual of $\sC(\Z_p,\Z_p)$. 
Formula \eqref{mahleramice12} establishes a map $\sD(\Z_p,\Z_p) \map{} \Z_p[[T]]$. On the other hand, for any sequence $n \longmapsto a_n$ in $\Z_p$, $\sum_n a_nT^n$ is the continuous linear operator $\sC(\Z_p, \Z_p) \to \Z_p$, $P_h \longmapsto a_h$, $\forall h \in \N$.
The description of the topology of $\sD(\Z_p,\Z_p)$ as the topology of uniform convergence on balls of a given radius was already  in  $\mathit 2$ of Proposition~\ref{explunifmeas}.
% \par 
%$\mathit 6.$ is well-known.
\end{proof}
 A more explicit description of the topology of  $\sD(\Z_p,\Z_p)$ will be given in the next section. 
\end{subsection}
\begin{subsection}{The natural topology of $\sD(\Z_p,\Z_p)$.}
In this section we prove the two important facts:
\ben \item  as a topological ring, $\sD(\Z_p,\Z_p) = \Z_p[[T]]$ equipped with  the $(p,T)$-adic topology. 
\item the topology of $\sD(\Z_p,\Z_p) = \Z_p[[T]]$ is induced by  the valuation $w$
\beq \label{wdef}
w(\sum_n a_n T^n) = \min_n v_p(a_n) + n \;,
\eeq 
for any $\sum_n a_n T^n \in \Z_p[[T]]$.
\een
Notice that $2$ follows easily from  $1$~:
\begin{lemma} \label{torval} The topology induced by the valuation $w$ on $\Z_p[T]$ coincides with the $(p,T)$-adic one.
\end{lemma}
\begin{proof} 
We show that, for any $n \in \N$,
$$
(p,T)^n \subset \{ y \in \Z_p[T] \;:\; w(y) \geq n \;\}\ \subset (p,T)^n
$$
We have $w(p^i T^j) = i + j$, so that the first inclusion is clear. If now $w(a_i T^i) \geq n$ for some $i \in \N$, 
 we have $v(a_i) + i \geq n$   but then $a_i T^i \in 
(p,T)^n$.
\end{proof}
We now prove statement $1$ above. 
\begin{notation} \label{unconvballs}
For any $h,\ell \in \Z_{\geq 0}$   let
\beq \label{Iwasawa2}
\sU_{h,\ell} = \{ \mu \in \sD(\Z_p,\Z_p) \;|\; \mu (a + p^h \Z_p) \in p^\ell \Z_p\;,\;\forall \; a \in \Z_p\;\} \subset \sD(\Z_p,\Z_p)  \;,
\eeq 
which is easily seen to be an ideal of $\sD(\Z_p,\Z_p)$. The linear topology of $\sD(\Z_p,\Z_p)$ having $\{\sU_{h,\ell}\}_{h,\ell}$ as a basis of open ideals  is 
 the topology of \emph{uniform convergence on balls of the same radius}. Since locally constant functions are dense 
in $\sC(\Z_p,\Z_p)$, 
this is the \emph{weak topology} of $\sD(\Z_p,\Z_p)$, \ie  the topology of simple convergence on the continuous $\Z_p$-linear dual of $\sC(\Z_p,\Z_p)$. 
\end{notation}
\begin{prop}  \label{Iwasawa3} Let us identify $\sD(\Z_p,\Z_p)$ with $\Z_p[[T]]$ via  \eqref{Iwasawa1}. 
Then the topology on $\sD(\Z_p,\Z_p)$ of uniform convergence on balls of the same radius  coincides with the $(p,T)$-adic topology of  $\Z_p[[T]]$. 
More precisely we have~:
\beq  \label{Iwasawa311} \begin{split} 
(p,T)^{p^N} & \subseteq  \bigcap_{h +\ell = N+1} \sU_{h,\ell} = \\
(p^{N+1},p^NT,p^NT^2 \dots, p^NT^{p-2}&,p^NT^{p-1}, p^{N-1}T^p,\dots, p^{N-1}T^{p^2-1},   p^{N-2}T^{p^2},\dots \\ \dots, 
p^2 T^{p^{N-2}},\dots, p^2 T^{p^{N-1}-1},&p T^{p^{N-1}},\dots, p T^{p^N-1},T^{p^N}) \subseteq (p,T)^{N+1} 
\;.
\end{split}
\eeq
\end{prop} 
\begin{rmk} \label{Iwasawa3110} The equality in formula \eqref{Iwasawa311} may be more easily remembered as
$$\bigcap_{h +\ell = N + 1} \sU_{h,\ell +1} = p^N(p,T,   T^p/p,  T^{p^2}/p^2,\dots \dots,   T^{p^{N-1}}/p^{N-1}, T^{p^N}/p^N) \;.   
$$
\end{rmk} 
\begin{proof}   We first prove the  inclusion $\subseteq$ in \eqref{Iwasawa311},  \ie that
\beq  \label{Iwasawa31} (p,T)^{p^{h +\ell}} \subseteq \sU_{h,\ell +1}  
\;.
\eeq
\begin{lemma} \label{Iwasawa30} 
$$T^{p^{h +\ell}} \in \sU_{h,\ell +1}  \;.$$
\end{lemma}
\begin{proof}
It suffices to check that for $a \in \{0,1,\dots,p^h-1\}$ the value 
\beq  \label{Iwasawa32} 
\begin{split}
&T^{p^{h+\ell}} (a +p^h\Z_p) = (\Delta_1 - \Delta_0)^{p^{h+\ell}} (a +p^h\Z_p) = \\
(-1)^{p^{h+\ell}} \sum_{i=0}^{p^{h+\ell}} &(-1)^i {p^{h+\ell} \choose i} \Delta_i  (a +p^h\Z_p) = 
%\sum_{
%\substack{
%0 \leq i\leq p^{h+\ell -1} \\
%i \equiv a \mod p^h
%}} 
\pm \sum_{j=0}^{p^{\ell} } (-1)^{a+j p^h} {p^{h+\ell} \choose a+j p^h}  
%=\\
%\begin{cases} 
%0
%    & \text{if } a =0 \\
%&\\  
%(-1)^a {p^{h+\ell} \choose a}
%    & \text{if }  a \ne 0
%\end{cases} 
\end{split} 
\eeq
is $\equiv 0 \mod p^{\ell +1}$. 
To do so, we recall that 
$$
 {p^{n+1} \choose pj} \equiv {p^{n} \choose j} \mod p^{n+1} 
$$
\cite[5.11.3]{Dwork}, while, obviously,
$$
v_p(  {p^{n+1} \choose i} ) = n+1 \;,
$$
if $(p,i) = 1$. 
Then, if $a=0$, \eqref{Iwasawa32}  becomes
%\footnote{The proof for $p=2$ differs slightly from the case of $p$ odd.} 
$$\pm \sum_{j=0}^{p^{\ell } } (-1)^{j p^h} {p^{h+\ell} \choose j p^h}\equiv 
\pm \sum_{j=0}^{p^{\ell} } (-1)^{j p^h} {p^{\ell} \choose j } 
%= \begin{cases} 
% 0 \mod p^{\ell +1}
%    & \text{if } p \;\mbox{is odd} \\
%&\\  
%0 \mod 2^{\ell}
%    & \text{if }  p=2 
%\end{cases} 
\equiv  0 \mod p^{\ell +1} \;,
$$
if $p$ is odd. If $p=2$ the better estimate ``$\;\equiv  0 \mod 2^{2^{\ell}}\;$'' holds.   If $a \neq 0$, the same sum  becomes  
$$\pm \sum_{j=0}^{p^{\ell} -1} (-1)^{a+j p^h} {p^{h+\ell} \choose {a+ j p^h}}
  \equiv 0 \mod p^{\ell +1} \;.
$$
\end{proof}
In order to complete the proof of   \eqref{Iwasawa31} we pick a generator of $(p,T)^{p^{h +\ell}}$ of  the form $p^i T^{p^{h +\ell} -i}$. We may clearly ignore the terms with $i > \ell$. We need to show that for any $0 < i \leq \ell$
\beq \label{dadim}
v_p(T^{p^{h +\ell} -i} (a +p^h\Z_p)) \geq \ell +1 -i \;.
\eeq
Since $T(\Z_p) = 0$ we may assume $h >0$, so that $p^{h-1}p^\ell(p-1) \geq \ell$ and then 
$$
p^{h+\ell}-i \geq p^{h+\ell -1} \;,\; \forall\; 0 < i \leq \ell\;.
$$
By  Lemma~\ref{Iwasawa30},
$$
v_p(T^{p^{h +\ell -1}} (a +p^h\Z_p)) \geq \ell  \;,
$$
and \eqref{dadim} holds.
So, we proved   the  first inclusion $\subseteq$ in \eqref{Iwasawa311}.
\par  
Let $\mu = \sum_{n=0}^\infty  a_n T^n \in \Z_p[[T]]$ and assume $\mu \in \bigcap_{h +\ell = N} \sU_{h,\ell +1}$. Equivalently, we assume that 
$$
\sum_{n=0}^{p^N - 1}  a_n T^n  \in \bigcap_{h =0}^N \sU_{h,N-h +1}
$$
So, for any $h = 0,1,\dots,N$, 
\beq \label{indh}
\sum_{n=0}^{p^h - 1}  a_n (-1)^n {n \choose a}  \in  p^{N-h +1}\Z_p \;,\; \mbox{for}\; a =0,1,\dots,p^h-1 \;.
\eeq
To prove this statement 
we apply orthogonality of the binomial coefficient functions and reason by induction on $h$. For $h=0$, we get $a_0 \in p^{N+1}\Z_p$.  We proceed to the case  $h = 1$ where, by the lemma,  \eqref{indh} simplifies into
 $$
\sum_{n=1}^{p - 1}  a_n (-1)^n {n \choose a}  \in  p^N\Z_p \;,\; \mbox{for}\; a =1,\dots,p-1 \;,
$$
which implies that $a_1,a_2,\dots,a_{p-1} \in p^N\Z_p$. The next step, $h=2$,  simplifies by the lemma into 
 $$
\sum_{n=p}^{p^2 - 1}  a_n (-1)^n {n \choose a}  \in  p^{N-1}\Z_p \;,\; \mbox{for}\; a =p,\dots,p^2-1 \;,
$$
and therefore $a_p,a_{p+1},\dots,a_{p^2-1} \in p^{N-1}\Z_p$.
 Proceeding in this way, we find at the $h$-th step the equation
 $$
\sum_{n=p^{h-1}}^{p^h - 1}  a_n (-1)^n {n \choose a}  \in  p^{N-h +1}\Z_p \;,\; \mbox{for}\; a =p^{h-1},p^{h-1}+1,\dots,p^h-1 \;.
$$
 On the other hand, the $p^N \times p^N$ matrix 
 $({n \choose a})_{n,a = 0,1,\dots,p^N-1}$ is unipotent and contains the $N+1$ diagonal blocks 
 $({n \choose a})_{n,a = p^{h-1},\dots,p^h-1}$, for $h = 0,1,\dots,N$,  of unipotent $(p^h-p^{h-1}) \times (p^h-p^{h-1})$ matrices 
 \beq
\left( \begin{array}{cccccccc }
1&0&\dots&\dots&0&  0\\
{{p^{h-1} +1}\choose{p^{h-1}}}&1 & &\dots&\dots&  0  \\
{{p^{h-1} +2}\choose{p^{h-1}}}&{{p^{h-1} +2} \choose {p^{h-1} +1}}&\dots &\dots&\dots&  0 \\
\dots&\dots&\dots &0&\dots&  0 \\
\dots&{n \choose a}&\dots&1&0&  0\\
{{p^h - 2}\choose{p^{h-1}}}  & \dots&\dots& & 1& 0   \\
{{p^h - 1}\choose{p^{h-1}}} &{{p^h -1} \choose {p^{h-1} +1}}&{{p^h -1} \choose {p^{h-1} +2}}&\dots&
{{p^h -1} \choose {p^h -2}}& 1  \\
\end{array} \right)  \;.
\eeq
We conclude 
 that $a_i \in p^{N-h+1} \Z_p$, for $p^{h-1} \leq i \leq p^h-1$. So, $\mu$ belongs to the ideal
 $$
(p^{N+1},p^NT,\dots,p^NT^{p-1},p^{N-1}T^p,\dots, p^{N-1}T^{p^2-1}, p^{N-2}T^{p^2}, \dots, p T^{p^{N-1}},\dots, p T^{p^N-1},T^{p^N}) \;,
 $$
 as claimed.   
\end{proof}
\begin{rmk} \label{prodtop}
It follows that the weak topology of $\sD(\Z_p,\Z_p) = \Z_p[[T]]$ coincides with the product topology of the $p$-adic topology on every factor $\Z_p T^n$ of the representation \eqref{Iwasawa1}.
\end{rmk}
\begin{cor} \label{coprodD} The coproduct $\P_\sD$ of $\sD(\Z_p,\Z_p)$ is the continuous $\Z_p$-algebra morphism such that 
$$
\P_\sD T = 1 \wt T + T \wt 1 + T \wt T \in \Z_p[[T]] \wt_{\Z_p} \Z_p[[T]]
\;,
$$
while the counit $\veps_\sD$ satisfies 
$$
\veps_\sD T = 0 \;.
$$
\end{cor}
\begin{cor} \label{IwasawaFp} Let 
$$t = \delta_1 - \delta_0 \in \sD(\Z_p,\F_p)\;.$$ 
Then  
 $\sD(\Z_p,\F_p)$ is the completion of its subring $\F_p[t]$ with the topology of uniform convergence 
 on balls of the same radius  in $\Z_p$. 
We have 
\beq \label{IwasawaFp1} 
(t^{p^n}) = \{ \mu \in \F_p[t]\,|\, \mu (a+p^n\Z_p) = 0 \,,\, \forall \, a \in \Z_p\,\}
\eeq
so that, as topological rings, 
\beq \label{IwasawaFp10}
\sD(\Z_p,\F_p) = \F_p[[t]]  \;,
\eeq 
equipped with the $t$-adic topology. 
\end{cor}
\begin{defn}
We let $\sQ_0$ denote the completion of the field $\Q(T)$ in the valuation $w$~: it is a non-archimedean field extension of $\Q_p$, with residue field $\F_p$ and uniformizer $p$ or $T$, that contains $\sD(\Z_p,\Z_p)$ as a closed topological subring.
\end{defn}
We find it useful to consider the Berkovich analytic line $(\A^1_{\sQ_0},x)$, where the coordinate $x$ is a formal coordinate on the unit disc of  $(\A^1_{\Q_p},x)$, and the corresponding theory of $\sQ_0$-analytic functions. From this viewpoint, Proposition~\ref{mahleramice}  has the following (obvious) analytic interpretation
\begin{lemma} \label{Q0anal}
\ben \item  
Let  $\sR(\Z_p)$ be the ring of  $\sQ_0$-entire  functions $f$ such that $f(\Z_p) \subset \sD(\Z_p,\Z_p)$.
The character 
$$\Delta_{(-)}: (\Z_p,+) \longrightarrow \sD(\Z_p,\Z_p)\;\;,\;\; x \longmapsto \Delta_x
$$
extends analytically as 
$$
\Delta_x := \Delta_1^x = \sum_{n=0}^\infty {x \choose n} T^n   \in \sR(\Z_p) \;.
$$
\item So, for any $f \in  \sC(\Z_p,\Z_p)$, formula \eqref{mahleramice12} for $f(x) = f \circ \Delta_x$ 
is a ``Fourier expansion'' of $f$ w.r.t.  a basis $\{{x \choose n}\}_{n \in \N}$  of $p$-adically entire (polynomial!) functions ${x \choose n} \in \Q[x]$, 
such that, for any $n \in \N$,  
$${a \choose n} \in \Z_p \;\;,\;\;\mbox{for any}\;\; a \in \Z_p\;,$$
and
$$
 {x+y \choose n} = \sum_{i+j=n}  {x  \choose i}  {y \choose j} \;.
$$
\een
\end{lemma}
\end{subsection}
\end{section}
\begin{section}{The ring $\sD=\sD_\unif(\Q_p,\Z_p)$ of uniform measures on $\Q_p$}
\label{unifmeas}
\begin{subsection}{Spaces of functions and measures on $\Q_p$} \label{Intpairingss}
We review what the general theory of integration on $td$-spaces tells us in the case of $\Q_p$.  We have introduced
\ben
\item $\sC(\Q_p,\Z_p)$ (resp. $\sC(\Q_p,\F_p)$) = ring of $\Z_p$-valued (resp. $\F_p$-valued) continuous functions equipped  with the topology of uniform convergence on compact subsets of $\Q_p$;
\item $\sC_\acs(\Q_p,\Z_p)$ (resp. $\sC_\acs(\Q_p,\F_p)$) = (non-unital) ring of $\Z_p$-valued (resp. $\F_p$-valued) continuous functions ``with almost compact support'' \ie ``vanishing at $\infty$'' on $\Q_p$,  with the topology of uniform convergence on $\Q_p$, \ie with the $p$-adic (resp. discrete) topology;
 \item $\sC_\unif(\Q_p,\Z_p)$ (resp. $\sC_\unif(\Q_p,\F_p)$) = ring of $\Z_p$-valued (resp. $\F_p$-valued) uniformly continuous functions on $\Q_p$ with the topology of uniform convergence on $\Q_p$, \ie with the $p$-adic (resp. discrete) topology;
 \item $\sD_\acs(\Q_p,\Z_p)$ (resp. $\sD_\acs(\Q_p,\F_p)$) = weak dual of $\sC(\Q_p,\Z_p)$ (resp. $\sC(\Q_p,\F_p)$) 
(see \eqref {Dacsdef}). It is the $\Z_p$-module of maps $\Sigma (\Q_p) \map{\mu} \Z_p$ which satisfy the condition {\bf ACSMEAS} of  $\mathit 3$ of Proposition~\ref{measinter} equipped with the $\Z_p$-linear topology described there. Similarly for $\sD_\acs(\Q_p,\F_p)$. 
  \item  $\sD^\circ_\acs(\Q_p,\Z_p)$ (resp. $\sD^\circ_\acs(\Q_p,\F_p)$) = strong dual of $\sC(\Q_p,\Z_p)$ (resp. $\sC(\Q_p,\F_p)$), \ie the $p$-adic completion of $\bigcup_{n \in \N}\sD(p^{-n}\Z_p,\Z_p)$ (resp. $\bigcup_{n \in \N}\sD(p^{-n}\Z_p,\F_p)$ equipped with the discrete topology)  (see  \eqref{Dacsodef}); 
    \item $\sD(\Q_p,\Z_p)$ (resp. $\sD(\Q_p,\F_p)$) = weak dual of $\sC_\acs(\Q_p,\Z_p)$ (resp. $\sC_\acs(\Q_p,\F_p)$) = $\limit_{n \in \N} \sD (p^{-n}\Z_p,\Z_p)$ (see \eqref{dualdefn5});
    \item $\sD^\circ(\Q_p,\Z_p)$ (resp. $\sD^\circ(\Q_p,\F_p)$) = strong dual of $\sC_\acs(\Q_p,\Z_p)$ (resp. $\sC_\acs(\Q_p,\F_p)$) = $\limit^{\square,\un}_{n \in \N} \sD_\strong (p^{-n}\Z_p,\Z_p)$ (see \eqref{dualdefn5st});
  \item
  $\sD_\unif(\Q_p,\Z_p)$ (resp. $\sD_\unif(\Q_p,\F_p)$) (see \eqref{funcunifdual}) is the set of functions 
  $$
\mu: \Sigma_\unif(\Q_p) \longrightarrow \Z_p \;\;\mbox{(resp. $\longrightarrow \F_p$)}
$$
such that, for any disjoint family $\{U_\alpha\}_{\alpha \in A}$ of uniformly measurable subsets of a given  radius 
$$
\sum_{\alpha \in A} \mu(U_\alpha) 
$$
is an $A$-series unconditionally converging to $\mu(\bigcup_{\alpha \in A}  U_\alpha)$ (see $\mathit 2$ of Proposition~\ref{measinter}). It is equipped with its \emph{natural topology}, \ie the topology of uniform convergence on uniformly measurable subsets of the same radius. 
\een 
We consider the  sub-semigroup
\beq \label{Sdef} S := \Z[1/p] \cap \R_{\geq 0} \eeq
of $(\R,+)$. 
\par \medskip
The following result is a particular case of  Proposition~\ref{convprod}.  
\begin{prop}  \label{summary}
\ben
 \item $\sC(\Q_p,\Z_p)$ (resp. $\sC(\Q_p,\F_p)$) is the  weak dual of $\sD_\acs(\Q_p,\Z_p)$ (resp. $\sD_\acs(\Q_p,\F_p)$);
  \item $\sC_\acs(\Q_p,\Z_p)$ (resp. $\sC_\acs(\Q_p,\F_p)$) is the weak dual of  $\sD(\Q_p,\Z_p)$ (resp. $\sD(\Q_p,\F_p)$). 
  \item $\sC_\unif(\Q_p,\Z_p)$ (resp. $\sC_\unif(\Q_p,\F_p)$) is the  strong dual of $\sD_\unif(\Q_p,\Z_p)$ 
  (resp. $\sD_\unif(\Q_p,\F_p)$)\;.
\een
\end{prop}
%We define \emph{uniformly measurable subsets}  of $\Q_p$ the unions of balls of the same radius, called the \emph{radius} of the subset. A ($\Z_p$-valued)  \emph{uniform measure on $\Q_p$} is a function on the boolean algebra $\Sigma_\unif(\Q_p)$ of uniformly measurable subsets   of $\Q_p$
%$$
%\mu: \Sigma_\unif(\Q_p) \longrightarrow \Z_p
%$$
%such that, for any disjoint family $\{U_\alpha\}_{\alpha \in A}$ of uniformly measurable subsets of the same radius 
%$$
%\sum_{\alpha \in A} \mu(U_\alpha) 
%$$
%is an $A$-series unconditionally converging to $\mu(\bigcup_{\alpha \in A}  U_\alpha)$. Then
%\begin{defn} We let $\sD_\unif(\Q_p,\Z_p)$ (resp. $\sD_\unif(\Q_p,\F_p)$) be the $\Z_p$-module of $\Z_p$-valued (resp. $\F_p$-valued) uniform measures on $\Q_p$ equipped with the topology of uniform convergence on  uniformly measurable subsets of the same radius. 
%\end{defn}
\begin{prop} \label{unifconvballs}
 \ben
\item $\sD_\unif(\Q_p,\F_p)$ is the completion of the group-Hopf-algebra $\F_p[\Q_p]$ in the $(\delta_1-\delta_0)$-adic topology or equivalently of $\F_p[t^{1/p^\infty}]$ in the $t$-adic topology, where $t = \delta_1 -\delta_0$.
In particular, it is a perfect ring.
\item   $\sD_\unif(\Q_p,\Z_p)$ is also $p$-adically complete and  $\sD_\unif(\Q_p,\Z_p)/p\sD_\unif(\Q_p,\Z_p) = \sD_\unif(\Q_p,\F_p)$ is a perfect ring. Hence 
$\sD_\unif(\Q_p,\Z_p)$, as a ring equipped with the $p$-adic filtration, is a strict $p$-ring in the sense of \cite[Chap. II, \S 5]{serre}  and, as rings,  
\beq \label{wittmeas}
\sD_\unif(\Q_p,\Z_p) = {\rm W}(\sD_\unif(\Q_p,\F_p)) \;. 
\eeq
\item $\sD_\unif(\Q_p,\Z_p)$ = the completion of the group-algebra $\Z[\Q_p]$ in the $(p, \Delta_1-\Delta_0)$-adic topology.
\item  For $n =0,1,2, \dots$, let   
\beq \label{Tndef} T_n =\Delta_{1/p^n} -\Delta_0 \; ,  
\eeq
so that, in particular, $T_0 =T$. We set
\beq \label{Tdef}
\wtilde{T}_n = \lim_{m \to \infty} (\Delta_{1/p^{n+m}} -\Delta_0)^{p^m} = \lim_{m \to \infty} T_{n+m}^{p^m}= [\delta_{1/p^n}-\delta_0] =[t^{1/p^n}]\;,
\eeq  
all $p$-adically convergent expressions. 
For $n= 1,2,\dots$,  $\wtilde{T}_n$ is the only $p^n$-th root of  $\wtilde{T} = \wtilde{T}_0$ in $\sD_\unif(\Q_p,\Z_p)$
%$$ (\wtilde{T}_n)^{p^n} = \wtilde{T} = [t] = [\delta_1 -\delta_0]\;.$$
%So, for $\wtilde{T} = \wtilde{T}_0$, we have, for $n=0,1,2,\dots$,
%$$ (\wtilde{T}_n)^{p^n} = \wtilde{T} = [t] = [\delta_1 -\delta_0]\;.$$
and $\sD_\unif(\Q_p,\Z_p)$ = the completion  of its subring $\Z[\wtilde{T}^{1/p^\infty}]$ in the $(p,\wtilde{T})$-adic topology. 
\item Equality \eqref{wittmeas} also holds topologically in the two cases:
\ben
\item If $\sD_\unif(\Q_p,\Z_p)$ (resp. $\sD_\unif(\Q_p,\F_p)$) is equipped with the $p$-adic (resp. discrete)  topology, and the ring of Witt vectors is regarded as a $p$-adic ring;
\item If $\sD_\unif(\Q_p,\Z_p)$ (resp. $\sD_\unif(\Q_p,\F_p)$) is equipped with its natural topology, \ie  the topology of uniform convergence on uniformly measurable subsets of the same radius, and the ring of Witt vectors is equipped with the product topology. 
\een 
\een
\end{prop}
%\bmau
%Domando se per ogni $\mu \in \sD_\unif(\Q_p,\F_p)$, e per ogni insieme uniformemente misurabile $U \subset \Q_p$, $[\mu](U)=[\mu(U)]$. Per ogni $n\in \N$ sia $\mu_n \in \sD_\unif(\Q_p,\F_p)$ tale che $\mu_n^{p^n} = \mu$. Poi prendo $M_n \in \sD_\unif(\Q_p,\Z_p)$ tale che, per ogni $U$, $M_n(U) \mod p = \mu_n(U)$. Allora 
%$$
%M_n^{p^n}(U) \equiv [\mu](U) \mod p^n \;.
%$$
%Per esempio, $\mu =\delta_1-\delta_0$,  $\mu_n =\delta_{1/p^n}-\delta_0$, $M_n = \Delta_{1/p^n}-\Delta_0$. Allora $M^p_n = \Delta_{1/p^{n-1}}-\Delta_0 + p R_{n,p}$ ove $R_{n,p}$ ha supporto in $p^{-n}\Z_p$. Poi 
%$$M^{p^2}_n = \Delta_{1/p^{n-2}}-\Delta_0 + p R_{n-1,p} +pR_{n,p}(\Delta_{1/p^{n-1}}-\Delta_0) + p^2R_{n,p} \;.$$ 
%$$M^{p^3}_n = \Delta_{1/p^{n-3}}-\Delta_0 + p R_{n-2,p} +pR_{n-1,p}(\Delta_{1/p^{n-2}}-\Delta_0) + pR_{n,p}(\Delta_{1/p^{n-1}}-\Delta_0)(\Delta_{1/p^{n-2}}-\Delta_0)\;.$$
%In definitiva
%$$
%M^{p^n}_n = M_0 + p^n R_{n,p} + p^{n-1} R_{n,p} 
%$$ 
%\emau
\begin{proof} $\mathit 1.$ $\F_p[\Q_p]$ is the $\F_p$-vector space generated by the symbols $\delta_a$, for $a \in \Q_p$. The identification of $\delta_a$ with the $\F_p$-valued Dirac mass concentrated in $a \in \Q_p$  determines the identification of 
$\F_p[\Q_p]$ with a $\F_p$-sub-vector space of $\sD_\unif(\Q_p,\F_p)$. 
Then 
$\F_p[\Q_p]$ is also a $\F_p$-Hopf-sub-algebra of $\sD_\unif(\Q_p,\F_p)$, with product $\delta_a \delta_b = \delta_{a +b}$, for $a,b \in \Q_p$, so that $\delta_0$ is the identity. The coproduct and counit are given by 
$$
\P \delta_a = \delta_a \otimes \delta_a \;\;,\;\; \veps \delta_a = 1\;, \;\forall\; a \in \Q_p\;,
$$
respectively. Moreover, $\F_p[\Q_p]$ is a perfect ring of characteristic $p$ because, for any finite sum $\sum_{a \in \Q_p} c_a \delta_a$, $c_a \in \F_p$, 
$$
(\sum_a c_a \delta_a)^{1/p}  = \sum_a c_a \delta_{a/p}\;.
$$
As we explained in $\mathit 2$ of Proposition~\ref{explunifmeas}, $\sD_\unif(\Q_p,\F_p)$ is the completion of 
$\F_p[\Q_p]$ for the  $\F_p[\Q_p]$-linear topology  of uniform convergence on balls  of the same radius in $\Q_p$. 
Notice that, for any $i\in \Z$,  Frobenius of $\sD_\unif(\Q_p,\F_p)$ establishes   isomorphisms of linearly topologized  $\F_p$-algebras
$$\sD_\unif(p^{-i}\Z_p,\F_p) = \F_p[[t^{1/p^i}]] \iso \sD_\unif(\Z_p,\F_p) = \F_p[[t]] 
%\;\;,\;\; \F_p[[t^{1/p^i}]] \iso \F_p[[t]] 
\;,$$ 
all equipped with the $t$-adic topology. In fact, by Corollary~\ref{IwasawaFp} and in particular \eqref{IwasawaFp1}, for any $i,n$
\beq \label{piroot}
t^{p^n} \F_p[t^{1/p^i}] = \{ \mu \in \F_p[t^{1/p^i}] \,|\, \mu (a+p^n\Z_p) = 0 \,,\, \forall \, a \in \Q_p\,\} \;.
\eeq
By the general theory, 
 $$\sD_\unif(\Q_p,\F_p) = \colimit^\un_{i\in \N} \sD_\unif(p^{-i}\Z_p,\F_p) 
 %= \colimit^\un_{i\in \N} \F_p[[t^{1/p^i}]] 
 $$
is the completion of 
$$\F_p[\Q_p]
= \F_p[t^{1/p^\infty}]
$$
in the topology of uniform convergence on balls of given radius. But formula \eqref{piroot} shows that 
for any $n\in \Z$,
$$
t^{p^n} \F_p[t^{1/p^\infty}] = \{ \mu \in \F_p[t^{1/p^\infty}] \,|\, \mu (a+p^n\Z_p) = 0 \,,\, \forall \, a \in \Q_p\,\} \;,
$$ 
so that $\sD_\unif(\Q_p,\F_p)$ coincides with 
 the $t$-adic completion of   
$$\F_p[\Q_p]
= \F_p[t^{1/p^\infty}]
$$
as claimed. 
\par \smallskip 
$\mathit 2.$ By the general theory, $\sD_\unif(\Q_p,\Z_p)$ is complete in the topology of uniform convergence on uniformly measurable subsets of the same radius. So it is an object of $\cCLMu_{\Z_p}$. 
As recalled in Remark~\ref{8.3.3}, any object of $\cCLMu_{\Z_p}$ is complete in its $p$-adic topology. 
All assertions of this point are then well-known. 
\par \smallskip 
$\mathit 3.$ We use crucially, for $X =\Q_p$ and $k=\Z_p$, the double presentation  \eqref{surpriseacs} of $\sD_\acs(\Q_p,\Z_p)$ as a limit and as a colimit, and the uniformity $(\Q_p,\Theta_\sJ)$,  where $\sJ := (\pi_h)_{h \in \N}$ is described in Remark~\ref{noncofinal}.
 From Proposition~\ref{Iwasawa3} we know that $\sD(\Z_p,\Z_p)$ identifies with the completion of its subring $\Z[T] = \Z[\Delta_1]$, where $T =\Delta_1 -\Delta_0$, in the $(p,T)$-adic topology. It follows that $\sD(p^{-n}\Z_p,\Z_p)$ identifies with the completion of its subring $\Z[T_n] =\Z[\Delta_{1/p^n}]$, for $T_n$ as in \eqref{Tndef}. So,
 $$
\Z[T] \subset \Z[T_1] \subset \Z[T_2] \subset \dots \;.
$$
From the two relations ($p$ is assumed odd; the case $p=2$ is analogous)
$$
T_1^p = T - \sum_{i=1}^{p-1}  (-1)^i {p \choose i} \Delta_{i/p} \in (p,T) \Z[T_1]
\;\;,\;\;
T = T_1^p + \sum_{i=1}^{p-1} (-1)^i  {p \choose i} \Delta_{i/p}  \in (p,T_1^p) \Z[T_1]
$$
it follows that
$$
(p,T) \Z[T_1] = (p,T_1^p) \Z[T_1]  = (p,T_1)^p \Z[T_1] \;.
$$
%$$
%(p,T_1^p) \Z[T_1] = (p,T_1)^p \Z[T_1] \subset  (p,T) \Z[T_1] \subset (p,T_1) \Z[T_1]  \;.
%$$
In the ring $\Z[T_\infty] = \bigcup_{n \in \N} \Z[T_n]$ we obtain
$$
(p,T_0)  = (p,T_1^p) = (p,T_1)^p  
% \subset  (p,T_0)   \subset (p,T_1)  
\;.
$$
Hence
$$
(p,T_1)  = (p,T_2^p) = (p,T_2)^p   \;.
$$
%$$
%(p,T_2^{p^2}) = (p,T_2)^{p^2}   \subset  (p,T_1)^p   \subset (p,T_2)^p   
%$$
%and
%$$
%\dots  \subset (p,T_2)^{p^2}   \subset  (p,T_1)^p  \subset (p,T_0)   \subset (p, T_1) \subset (p,T_2) \subset \dots\;\;.
%$$
So, for any $n \in \N$, we have
\beq \label{pTntop}
(p,T) =   (p,T_n^{p^n})  = (p,T_n)^{p^n}  \;.
\eeq
We conclude that the $(p,T)$-adic topology and the $(p,T_n)$-adic topology of $\Z[T_\infty]$ coincide for any $n \in \N$. 
Both the $(p,T)$-adic topology of $\Z[T_n]$ and the natural topology of $\sD_\unif(p^{-n}\Z_p,\Z_p)$, \ie the topology of uniform convergence on balls of the same radius in $p^{-n}\Z_p$,  are linear topologies, hence 
$\sD_\unif(p^{-n}\Z_p,\Z_p)$ with its natural topology coincides with the completion of its subring $\Z[T_n]$ equipped with the $(p,T)$-adic topology. 
This remains true for  $\Z[T_\infty] \subset \sD_\unif(\Q_p,\Z_p)$ and proves $\mathit 3$.
\par \smallskip 
$\mathit 4.$ Clear.
\par \smallskip 
$\mathit 5.$ Part \emph{(a)} is well-known. For part \emph{(b)}, we observe that a basis of open  neighborhoods of 0 in the product topology of ${\rm W}(\sD_\unif(\Q_p,\F_p))$ consists of  the ideals
$$
(p^N,[t^{p^N}]) = \sum_{i=0}^N p^i  [t^{p^{N-i}} \sD_\unif(\Q_p,\F_p)]
$$
for $N =0,1,\dots$, and that this system is cofinal with $\{(p,[t])^N\}_{N =0,1,\dots}$.
\end{proof} 

\begin{defn} \label{Qdef} For $S$ as in \eqref{Sdef}, we define an $S$-valued valuation $w: \Z[\wtilde{T}^{1/p^\infty}] \to S \subset \R_{\geq 0}$ as 
\beq \label{measconv4} 
w(\sum_{q \in S} a_q \wtilde{T}^q) = \min_{q \in S} (v_p(a_q) + q )\;.
\eeq 
Obviously, $w$ extends to a $\Z[1/p]$-valued valuation $w: \Q(\wtilde{T}^{1/p^\infty}) \to \Z[1/p]$. We denote by $\sQ$ the completion of the valued field $(\Q(\wtilde{T}^{1/p^\infty}),w)$. 
\end{defn} 
Then $(\sQ,w)$ is a non-archimedean $\Z[1/p]$-valued  field extension of $\Q_p$ with uniformizer $\wtilde{T}$ or $p$ and residue field $\F_p$. 
In particular, Berkovich theory of $\sQ$-analytic spaces applies. 
\begin{lemma} \label{torval} The topology induced by the valuation $w$ on $\Z[\wtilde{T}^{1/p^\infty}]$ coincides with the $(p,\wtilde{T})$-adic one.  
\end{lemma}
\begin{proof} 
 It suffices to show that, for any $n = 3,4,\dots$,
$$
(p,\wtilde{T})^n \subset \{ y \in R \;:\; w(y) \geq n \;\}\ \subset (p,\wtilde{T})^{n-1} 
$$
We have $w(p^i \wtilde{T}^j) = i + j$, so that the first inclusion is clear. If now $w(a_q \wtilde{T}^q) \geq n$ for some $q \in S$, 
$v(a_q) + q \geq n$ implies $v(a_q)  + \lfloor q \rfloor  \geq n-1$ but then $p^{v(a_q)} \wtilde{T}^{\lfloor q \rfloor} \in 
(p,\wtilde{T})^{n-1} $.
\end{proof}

\begin{cor} \label{corq} $\sD_\unif(\Q_p,\Z_p)$ is the  $w$-completion of $\Z[\wtilde{T}^{1/p^\infty}]$. Then, 
every element in 
 $\sD_\unif(\Q_p,\Z_p)$ is the sum of a unique $S$-series $\sum_{q \in S} a_q \wtilde{T}^q$, with $a_q \in \Z_p$, such that the following condition holds
\beq \label{measconv} 
\mbox{For any  $C>0$, $v(a_q) +q  > C$ 
 for almost all $q \in S$.}
\eeq
Conversely, any $S$-series $\sum_{q \in S} a_q \wtilde{T}^q$ such that the map $q \mapsto a_q \in \Z_p$ satisfies \eqref{measconv}, converges in $\sD_\unif(\Q_p,\Z_p)$ along the filter of cofinite subsets of $S$.  
\par
The $(p,\wtilde{T})$-adic topology of $\sD_\unif(\Q_p,\Z_p)$ coincides with  the topology induced by the valuation
\beq \label{measconv4} 
w(\sum_{q \in S} a_q \wtilde{T}^q) = \min_{q \in S} (v(a_q) + q ) \in S\;,
\eeq
for any $\sum_{q \in S} a_q \wtilde{T}^q \in  \sD_\unif(\Q_p,\Z_p)$.  \end{cor}
\begin{defn} We will denote by  $\sQ\{x\}$ the $\sQ$-Fr\'echet algebra of entire functions on the Berkovich affine line $(\A^1_{\sQ},x) = (\A^1_{\Q_p},x) \, \wt \,\sQ$. In particular $\sD := \sD_\unif(\Q_p,\Z_p)$ is the closure of 
$\Z[\wtilde{T}^{1/p^\infty}]$ in $\sQ$. We let $\sD(\Q_p)$ be the ring of  $\sQ$-entire functions $f$ such that $f(\Q_p) \subset \sD$.
\end{defn}

 \end{subsection}

 \begin{subsection}{Change of variables in $\sD = \sD_\unif(\Q_p,\Z_p)$} \label{ChgVarss}
% We consider the completion  $\sD$ of the ring  
% $$ \Z_p[T^{1/p^\infty}]   $$
% in the $(p,T)$-adic topology.  
 We denote by $\sD$, for short, the ring $\sD_\unif(\Q_p,\Z_p)$ equipped with its natural topology, \ie the topology of uniform convergence on balls of the same radius. We proved that $\sD$ is the  completion of $\Z_p[\Q_p]$ in the $(p,T)$-adic topology, for $T=\Delta_1-\Delta_0$. 
We also denote by $\sD^\padic$ the ring  $\sD$ equipped with its $p$-adic topology. Notice that $\sD^\padic$ is  complete 
and that  it contains properly the $p$-adic completion 
of  $\Z_p[\Q_p]$ whose residue ring is $\F_p[t^{1/p^\infty}]$ equipped with the discrete topology.  We also set $\wtilde{\sD} = \sD/p\sD$ (topological quotient) which coincides, for $t=T \mod p$, with the $t$-adic completion of  
 $$\F_p[t^{1/p^\infty}] = \colimit (\F_p[t] \map{\varphi} \F_p[T] \map{\varphi} \dots) \;,$$ 
where $\varphi$ is Frobenius and is  therefore a perfect ring of characteristic $p$. 
Notice that  the topological quotient 
$$\sD^\padic/p\sD^\padic = \wtilde{\sD}^\dis\;,$$
which is the ring $\wtilde{\sD}$ equipped with the discrete topology and contains properly $\F_p[t^{1/p^\infty}]$. 
%Then $\F_p[t^{1/p^\infty}]$ equipped with the discrete topology, which is a (much smaller!) subring of  $\wtilde{\sD}^\dis$,  the ring $\wtilde{\sD}$ equipped with the discrete topology. 
  \par \medskip 
 The topological ring $\sD^\padic$ is a strict $p$-ring in the sense of \cite[Chap. II, \S 5]{serre} with residue ring $\wtilde{\sD}^\dis$. In particular, the map 
\beq
\theta: \vect(\wtilde{\sD}^\dis) \longrightarrow \sD^\padic
\eeq
of \emph{loc. cit.}  is an isomorphism of rings.  For any $x \in \wtilde{\sD}$ we simply denote $\theta ([x])$ by 
$[x]$. In the terminology of \emph{loc. cit.}, it is the \emph{multiplicative representative} of $x$ in $\sD$. 
\par \medskip
  The reason why we need the $(p,T)$-adic topology and not just the (finer) $p$-adic topology on 
$\Z[\Q_p]$ is that we need a change of variable which would not converge $p$-adically, but only converges $(p,T)$-adically, as we now explain. 
Let
$$ E(T) = \exp (\sum_{i=0}^\infty T^{p^i}/p^i) \in 1 + T\Z_{(p)}[[T]]$$
 (similar to $\exp(T)$) be the ($p$-adic) \emph{Artin-Hasse exponential} series and let 
 $$ L(T) \in T+ T^2\Z_{(p)}[[T]]
$$
(similar to $\log (1+T)$)
be its inverse in the sense that 
$$
E(L(T)) = 1+ T \; \;\mbox{and}\;\; L(E(T)-1) = T \;.
$$
Then $L$ is called the ($p$-adic) \emph{Artin-Hasse  logarithm}.  
The two Artin-Hasse series, the exponential $E(T)$ and the logarithm $L(T)$, establish an isomorphism between the formal open $\Z_p$-disc  centered at $0$ and the formal open $\Z_p$-disc centered at $1$. We denote by $\ol{E}(t) \in 1+t\F_p[[t]]$ and $\ol{L}(t) \in t+t^2\F_p[[t]]$ the reductions of $E$ and $L$ modulo $p\Z_{(p)}[[T]]$. 
\begin{prop} \label{repr1} \hfill
\ben  
%\item The expression $\lim_{n \to + \infty} (T^{1/p^n} +1)^{p^n}$ converges in $\sD^\circ$   to the element 
%$$\Delta_1 :=  [1+t] \in \sD^\circ = {\rm W}(\wtilde{\sD}^\circ)\;.$$
%For any $i = 1,2,\dots$, $\lim_{n \to +\infty} (T^{1/p^{n+i}} +1)^{p^n}$ converges in $\sD^\circ$ to a definite $p^i$-th root  $\Delta_{1/p^i} = [1+t^{1/p^i}]$ of $\Delta_1$. For any $i=0,1,2,\dots$, we set $\delta_{1/p^i} := \Delta_{1/p^i} + p\sD \in \wtilde{\sD}^\circ$, so that 
%$$
%\Delta_{1/p^i} = [\delta_{1/p^i}] \in \sD^\circ\;.
%$$
%\item We have in $\sD^\circ$
%$$T^{1/p^i} = [t^{1/p^i} ] = [\delta_{1/p^i} -1] = \lim_{n \to +\infty} (\Delta_{1/p^{n+i}} -1)^{p^n}\;.$$ 
\item  Let 
$$\ol{\mu}_\can := \ol{L}(t) \in t +t^2\F_p[[t]] \subset \sD_\unif(\Q_p,\F_p)\;.$$
\item  The expression
$$\lim_{n \to +\infty} L(\Delta_{1/p^n}-1)^{p^n} = \lim_{n \to +\infty} L(T_n)^{p^n}
$$ 
 converges  in $\sD_\unif(\Q_p,\Z_p)$ to the element 
$$\mu_\can := [\ol{\mu}_\can]
\;.$$ 
\item Conversely,   for any $i = 0,1,2,\dots$, 
$$\lim_{n \to +\infty} E(\mu_\can^{1/p^{n+i}})^{p^n} = \Delta_{1/p^i}  $$
for the $p$-adic topology of  $\sD_\unif(\Q_p,\Z_p)$.
\een
\end{prop}
\begin{proof} Obvious. 
\end{proof}
%\bmau We denote by $\sQ$ = the completion of $\Q(T^{1/p^\infty})$ in the valuation $w$ and let $\sD(\Q_p)$ be the ring of  $\sQ$-entire functions $f$ such that $f(\Q_p) \subset \sD$. \emau
 \end{subsection}
 \begin{subsection}{Measures on $\Q_p$ and Witt vectors}  

\begin{prop} \hfill
\ben  
\item The expression $\lim_{n \to + \infty} (\wtilde{T}^{1/p^n} +1)^{p^n}$ converges in $\sD^\padic$ to the element 
$$\Delta_1 :=  [1+t]\;.$$
For any $i = 1,2,\dots$, $\lim_{n \to +\infty} (\wtilde{T}^{1/p^{n+i}} +1)^{p^n}$ converges in $\sD^\padic$ to a definite $p^i$-th root  $\Delta_{1/p^i} = [1+t^{1/p^i}]$ of $\Delta_1$. For any $i=0,1,2,\dots$, we set $\delta_{1/p^i} := \Delta_{1/p^i} \mod p \in \wtilde{\sD}$, so that 
$$
\Delta_{1/p^i} = [\delta_{1/p^i}] \;.
$$
\item We have in $\sD^\padic$
$$\wtilde{T}^{1/p^i} = [t^{1/p^i} ] = [\delta_{1/p^i} -1] = \lim_{n \to +\infty} (\Delta_{1/p^{n+i}} -1)^{p^n}\;.$$ 
\item  Let 
$$\ol{\mu}_\can := \ol{L}(t) \in t +t^2\F_p[[t]]  \;.$$
\item  The two expressions 
$$\lim_{n \to +\infty} L(\Delta_{1/p^n}-1)^{p^n} = \lim_{n \to +\infty} L(\wtilde{T}^{1/p^n})^{p^n}
$$ 
both converge  in $\sD$ to the element 
$$\mu_\can := [\ol{\mu}_\can]
\;.$$ 
\item Conversely,   for any $i = 0,1,2,\dots$, 
$$\lim_{n \to +\infty} E(\mu_\can^{1/p^{n+i}})^{p^n} = \Delta_{1/p^i}  $$
in $\sD$.
\een
\end{prop}
\begin{proof} Obvious. 
\end{proof}
\begin{rmk} \label{pidef} The sum
\beq \label{pi} 
\pi(\wtilde{T}) = \sum_{i=-\infty}^\infty \wtilde{T}^{p^i}/p^i  
\eeq
converges in $\sD$.
\end{rmk}
\begin{rmk} \label{Deltachar} 
The maps
\beq
\begin{split}
\Q_p \map{} \sD \;\;,&\;\; s \longmapsto \Delta_s \\
 \Q_p \map{} \wtilde{\sD} \;\;,&\;\;  s \longmapsto \delta_s
 \end{split} 
 \eeq
are obviously continuous characters. In fact, for any $s_1,s_2 \in \Q_p$
$$
\Delta_{s_1 + s_2} = \Delta_{s_1}\Delta_{s_2} \;\;,\;\; \delta_{s_1 + s_2} = \delta_{s_1}\delta_{s_2} \;.
$$
Continuity follows from the fact that, for any $h=0,1, \dots$, $|s_1-s_2|_p \leq p^{-h}$ if and only if for any ball $B= a+p^h\Z_p$ of radius $\pi_h$,  $s_1$ and $s_2$ are both in  or both out  of $B$. 
\par \smallskip 
The previous characters are in fact locally $\sQ$-analytic in a neighborhood of $\Q_p$ in the analytic Berkovich $\sQ$-line.  For any  $s_0 \in p^{-n}\Z_p$ the formulas
\beq \begin{split}
\Delta_s = \Delta_{1/p^n}^{sp^n} = (1+\wtilde{T}^{1/p^n})^{sp^n} = \sum_{i=0}^\infty {sp^n \choose i} \wtilde{T}^{i/p^n}\\
\delta_s = \delta_{1/p^n}^{sp^n}  = (1+t^{1/p^n})^{sp^n}  = \sum_{i=0}^\infty {s p^n\choose i} t^{i/p^n}
 \end{split} 
 \eeq
 provide a $\sQ$-analytic continuation of $s \longmapsto \Delta_s$ and of $s \longmapsto \delta_s$ to a neighborhood of $s_0$ in the Berkovich $\sQ$-line. 
%the series $(1+t)^s = \sum_{i=0}^\infty {s \choose i} t^i$ (resp. $(1+\wtilde{T})^s = \sum_{i=0}^\infty {s \choose i} \wtilde{T}^i$) converges in $ \wtilde{\sD}$ (resp. in $\sD$) and $s \mapsto (1+t)^s$ (resp. $s \mapsto (1+\wtilde{T})^s$)  is a continuous function of $s$. Moreover, $s \mapsto (1+t)^s$ may be extended to a continuous character of $\Q_p$ via 
%$$(1+t)^{p^{-n}s} = (1+t^{p^{-n}})^{s} = \sum_{i=0}^\infty {s \choose i} (t^{p^{-n}})^i \;.$$
%  For any $s \in \Q_p$ we define $\Delta_1^s = [(1+t)^s]$ so that $s \longmapsto \Delta_1^s$ is a continuous character. 
\end{rmk}
\end{subsection}
\end{section}
 \begin{section}{Fourier theory for $\sC_\unif(\Q_p,\Z_p) = \sD_\unif(\Q_p,\Z_p)'_\strong$}
 \label{fourier}
\begin{subsection}{Uniformly continuous functions  $\Q_p \longrightarrow \Z_p$.}    \label{unifcontQps}
  Let $S := \Z[1/p] \cap \R_{\geq 0}$, be as before. For $q \in S$ the function ${x \choose q}^{(p)}  : \Q_p \longrightarrow \Z_p$ is defined as 
\beq \label{binQp}
{x \choose q}^{(p)} := \lim_{n \to +\infty} {{p^n x} \choose {p^n q}} \;,\;\forall \; x\in \Q_p\;.
\eeq
The convergence of \eqref{binQp} in $\Z_p$  follows from the congruence 
\beq \label{bald}
{{p\gamma} \choose {p\beta}} \equiv {{\gamma} \choose {\beta}} \mod p \gamma \Z_p \;,
\eeq
for $\beta, \gamma \in \N$ \cite[5.11.3]{Dwork}.

\begin{prop} The functions ${x \choose q}^{(p)}  : \Q_p \longrightarrow \Z_p$, for $q \in S$, are continuous. 
\end{prop}
\begin{proof} Let $q = p^{-m}Q$, with $m, Q \in \N$. 
Continuity of $x \longmapsto {x \choose q}$ may be checked  locally~:  for $x \in p^{-n}\Z_p$ 
$${x \choose q}^{(p)}   =  \lim_{i \to \infty} {p^{m+n+i} x \choose p^{n+i}Q} 
$$
obviously a sequence  of continuous functions of $x \in p^{-n}\Z_p$. Moreover, by \eqref{bald}, we have for $x \in p^{-n}\Z_p$
$$
{p^{m+n+i+1} x \choose p^{n+i+1}Q}  - {p^{m+n+i} x \choose p^{n+i}Q}   \in p^{m+i+1} \Z_p\;, 
$$
so that the sequence converges uniformly to ${x \choose q}^{(p)}$ on $p^{-n}\Z_p$. 
\end{proof}
\begin{lemma} For any $q \in S$ and $s \in \Q_p$
$$
v({s \choose q}^{(p)}  )   \geq (p-1) (v(s)-v(q)) \;.
$$
\end{lemma}
\begin{proof} If $v(s) \leq v(q)$ there is nothing to prove. So assume 
 $v(s) > v(q)$. We have 
 $$| {s \choose q}^{(p)}  | = | {p^N s \choose {p^N q}} | 
 $$
 for $N >> 0$. Let the $p$-adic digit expansions of $p^N s$, $p^N q$ and $p^N s - p^N q$ be
 $$
 p^Ns = n_{N+v(s)}p^{N+v(s)} + n_{N+v(s)+1}p^{N+v(s)+1} + \dots \;,
 $$
  $$
 p^Nq = a_{N+v(q)}p^{N+v(q)} + a_{N+v(q)+1}p^{N+v(q)+1} + \dots \;,
 $$
  $$
 p^Ns - p^Nq = b_{N+v(q)}p^{N+v(q)} + b_{N+v(q)+1}p^{N+v(q)+1} + \dots \;.
 $$
Therefore, if $s \in S$, 
$$
v({s \choose q}^{(p)}  ) = v({p^N s \choose {p^N q}} ) \geq (p-1) (v(s)-v(q)) \;.
$$
By continuity in $s$ this holds for any $s \in \Q_p$ with $v(s) >v(q)$. 
\end{proof}
 \begin{cor} \label{unifcont1} \ben \item For any $q \in S$ and any compact subgroup $C$ of $(\Q_p,+)$
 we have
\beq \label{unifcont1}
 {{x+y}\choose q}^{(p)}  = \sum_{q_1 + q_2 = q}  {x \choose q_1}^{(p)}   {y \choose q_2}^{(p)}  \;,
\eeq
for $x,y \in C$, where for $S_q = \{(q_1,q_2) \in S^2\;|\; q_1 + q_2 = q\;\}$ 
 the $S_q$-series in the statement  is unconditionally convergent for the topology of uniform convergence on $C \times C$. 
 \item In particular, the functions  
 ${x \choose q}^{(p)}  : \Q_p \longrightarrow \Z_p$, for $q \in S$, are uniformly continuous.
 \een
 \end{cor}
 \begin{proof} $\mathit 1.$ For a given $q \in S$, and any fixed $v \in \Z$, there are only a finite subset $F_v$ of $(q_1,q_2) \in S_q$ 
 such that $v(q_i) \geq v$, for   $i=1$ or  $i=2$. Now $v(x)$ has a minumum $M(C)$ for $x \in C$. Let us choose $v$ to be    $M(C) - h$, for $h =1,2,\dots$, $h>>0$. Then for $(q_1,q_2) \in S_q - F_v$ we have  $v(q_i) < M(C) -h$, for both $i=1$ and $i=2$. Then 
 $$ v({x \choose q_1}^{(p)}   {y \choose q_2}^{(p)} ) \geq (p-1)h \;.
 $$
 \par\noindent $\mathit 2.$ 
 %For $(q_1,q_2) \in S_q$,  
 For the given $q \in S$,  the subset $G_q$ of those $(q_1,q_2) \in S_q$  such that 
 $v(q_2) \geq v(q)$ is   finite.  Let $V(q) := \max_{(q_1,q_2) \in G_{v(q)}}  v(q_2)$. We choose $\veps_h \in \Q_p$ such that 
 $v(\veps_h) > V(q) + h$, for $h =1,2,\dots$. Then 
 $$
 v({x \choose q_1}^{(p)}   {\veps_h \choose q_2}^{(p)}  ) \geq (p-1)h \;,
 $$
independent of $x \in \Q_p$. We conclude that
$$
 v({{x +\veps_h}\choose q}^{(p)}  - {x\choose q}^{(p)} )  \geq (p-1)h
 $$
 independent of $x \in \Q_p$, hence the result. 
 \end{proof} 
 \end{subsection}
  \begin{subsection}{Orthogonality}
  \label{orthogonal}
 Let $\iota(p): \Q_p \longrightarrow \Q_p$ denote the map 
multiplication by $p$ in $\Q_p$. It induces natural Hopf-algebra morphisms 
$$\sC_\unif(\Q_p,\Z_p) \map{{\iota(p)}^\ast} \sC_\unif(\Q_p,\Z_p) \;\;,\;\; f \longmapsto {\iota(p)}^\ast f \;,$$
where $({\iota(p)}^\ast f )(x) = f(px)$, 
and 
$$\sD_\unif(\Q_p,\Z_p) \map{\iota(p)_\ast} \sD_\unif(\Q_p,\Z_p)\;\;,\;\; \sum_q a_q\Delta_q \longmapsto  \sum_q a_q\Delta_{pq}$$
such that 
$$
\int_{\Q_p} (\iota(p)^\ast f)(x) d \mu(x) = \int_{\Q_p} f(px) d \mu(x)  =   \int_{\Q_p}  f(x) d (\iota(p)_\ast \mu)(x) \;.
$$
Equivalently,
$$
\iota(p)_\ast(T_{n+1}) = \iota(p)_\ast(\Delta_{1/p^{n+1}} - \Delta_0) = \Delta_{1/p^n} - \Delta_0 = T_n \;.
$$
 The previous maps induce morphisms of weak perfect pairings
\beq \label{pushoutmonoou}
 \;\;\;  \begin{tikzcd}[column sep=0.6em, row sep=2.5em] 
\sC(p^{-n}\Z_p,k)  \arrow{d}{{\iota(p)}^\ast}  & \times
& \sD(p^{-n}\Z_p,k)  \arrow{rrrrrr}{\int_{p^{-n}\Z_p}} &&&&&& k \arrow{d}{\id_k}
\\ 
\sC(p^{-n-1}\Z_p,k)   & \times
& \sD(p^{-n-1}\Z_p,k)  \arrow{u}{\iota(p)_\ast}   \arrow{rrrrrr}{\int_{p^{-n-1}\Z_p}}&&&&&&k
\end{tikzcd}
\eeq
such that for any $f \in \sC(p^{-n}\Z_p,k)$ and $\mu \in  \sD(p^{-n-1}\Z_p,k)$  
\beq \label{intform}
\int_{p^{-n-1}\Z_p} (\iota(p)^\ast f)(x) d \mu(x) = \int_{p^{-n}\Z_p} f(x) d (\iota(p)_\ast \mu)(x) \;.
\eeq
\begin{thm} \label{orthrel}
For any $q,q' \in S = \Z[1/p]_{\geq 0}$, we have 
\beq  \label{orthrel1}
{x \choose q}^{(p)} \circ \wtilde{T}^{q'} = \int_{\Q_p} {x \choose q}^{(p)}  d\wtilde{T}^{q'}(x) = \delta_{q,q'} \;.
\eeq
\end{thm}
\begin{proof}
We have 
$$ \int_{\Q_p} {x \choose q}^{(p)}  d\wtilde{T}^{q'}(x) = \lim_{N \to +\infty}  \int_{\Q_p} {p^N x \choose p^N q}^{(p)}  d T_N^{p^Nq'}(x) = \delta_{p^Nq,p^Nq'} =  \delta_{q,q'} \;.
$$
\end{proof}
\end{subsection}
 \begin{subsection}{Fourier expansion in  $\sC_\unif(\Q_p,\Z_p)$} \label{fourierunif}
\begin{thm} \label{fourierunif0} \ben 
\item For any $\mu \in \sD_\unif(\Q_p,\Z_p)$ we have 
\beq \label{fourierunif03}
\mu  = \sum_{q \in S} a_q \wtilde{T}^{q} \;,
\eeq
where $a_q \in \Z_p \; \forall q \in S$, and, for any $C>0$, 
\beq \label{fourierunif031}
v(a_q) + q > C \;\;,\;\; \forall \; \forall \; q \in S \;.
\eeq
For any $q \in S$, we have 
\beq \label{fourierunif04}
a_q =  \int_{\Q_p}  {x \choose q}^{(p)} d\mu(x) \;.
\eeq
\item For any $f \in \sC_\unif(\Q_p,\Z_p)$
we have 
\beq \label{fourierunif01}
f(x) = \sum_{q \in S} b_q {x \choose q}^{(p)} \;,
\eeq
where $b_q \in \Z_p \; \forall q \in S$, and, for any $C>0$, 
\beq \label{fourierunif011}
v(b_q) - q > C \;\;,\;\; \forall \; \forall \; q \in S \;.
\eeq
For any $q \in S$, we have
\beq \label{fourierunif02}
b_q =  \int_{\Q_p} f(x)  d\wtilde{T}^{q}(x) \;.
\eeq
\een
\end{thm}
\begin{proof} Part $\mathit 1$ has already been proven. As for part $\mathit 2$,  we know  from the general theory that $\sC_\unif(\Q_p,\Z_p)$ is the strong dual of $\sD_\unif(\Q_p,\Z_p)$. For any element $f \in \sC_\unif(\Q_p,\Z_p)$, 
we may define $b_q$, for any $q \in S$, as in \eqref{fourierunif02}. Then, for any $\mu \in \sD_\unif(\Q_p,\Z_p)$ as in  \eqref{fourierunif03} 
$$
f \circ \mu = \sum_{q \in S} a_q (f \circ \wtilde{T}^{q}) =  \sum_{q \in S} a_q b_q \;,
$$
so that the $S$-series on the \emph{r.h.s.} must  converge unconditionally in $\Z_p$. 
Therefore, for any $C>0$
$$
v(a_q) + v(b_q)   > C \;\;,\;\; \forall \; \forall \; q \in S \;.
$$
Since $q \longmapsto a_q$ is any $S$-sequence subject to \eqref{fourierunif031}, the $S$-sequence 
$q \longmapsto b_q$ must satisfy \eqref{fourierunif011}. Now, for any $\nu = \sum_{q \in S} c_q \wtilde{T}^{q} \in \sD_\unif(\Q_p,\Z_p)$, the $S$-sequence 
$q \longmapsto c_q$ satisfies the estimate \eqref{fourierunif031} so that 
$$\sum_{q \in S} b_q {x \choose q}^{(p)} \circ \nu = \sum_{q \in S} b_qc_q$$
is an $S$-series unconditionally convergent. This shows that $\sum_{q \in S} b_q {x \choose q}^{(p)} \in \sC_\unif(\Q_p,\Z_p)$ and that it equals the original $f$. 
\end{proof}
\end{subsection}
 \begin{subsection}{Final remarks} \label{final}
%\begin{section}{Uniform $\Z_p$-valued measures on $\Q_p$} 
%\begin{defn} \label{unifmeas}  A subset $U \subset \Q_p$ is \emph{uniformly measurable} (\emph{of radius $\rho$}) if it is a union of balls of the same radius. We denote by $\Sigma_\unif(\Q_p)$ the set of uniformly measurable subsets of $\Q_p$ (of any radius).
%A \emph{uniform $\Z_p$-valued measure on $\Q_p$} is  a map
%$$
%\mu: \Sigma_\unif(\Q_p) \longrightarrow \Z_p
%$$
%such that for any disjoint family $\{U_\alpha\}_{\alpha\in A}$ of uniformly measurable subsets of $\Q_p$ of the same radius, 
%$$
%\sum_{\alpha \in A} \mu(U_\alpha)
%$$
%is an unconditionally convergent $A$-series and its sum is $\mu(\bigcup_{\alpha \in A} U_\alpha)$. We denote by $\sD_\unif(\Q_p,\Z_p)$ the $\Z_p$-module of uniform $\Z_p$-valued measures on $\Q_p$ equipped with the topology of uniform convergence on uniformly measurable subsets of $\Q_p$ of the same radius. 
%\end{defn} 
  We now  pick a system $(\gamma_i)_{i \in \N}$ of $p^n$-roots of $\gamma_0 \in 1 +p\Z_p$, $\gamma_{i+1}^p =\gamma_i$, for $i=0,1,\dots$, and set 
$$ \gamma := (\gamma_0,\gamma_1,\gamma_2,\dots) = 1 +\varpi \in 1+\C_p^{\flat \circ \circ} \;.$$
%so that  $c(\gamma_0)^\sharp = \gamma_0$. 
We write explicitly
$$(\gamma_0,\gamma_1,\gamma_2,\dots)  = (1,1,1,\dots) + (\varpi_0,\varpi_1,\varpi_2,\dots)
$$
with  $\varpi_{i+1}^p = \varpi_i \in \C_p^{\circ \circ}$.  
Notice that for any $s \in \Q_p$, $\gamma^s$ makes sense in $1+\C_p^{\flat \circ \circ}$ as  the expression 
$$(\gamma_n,\gamma_{n+1}, \dots)^{p^ns} =  ((1,1,1,\dots) + (\varpi_n,\varpi_{n+1},\dots))^{p^ns} =  
\sum_{i=0}^\infty {{p^ns} \choose i}
(\varpi_n,\varpi_{n+1},\dots)^i \;,
$$
for any $n \in \N$ such that $p^n s \in \Z_p$.
The next proposition shows that the additive character 
\beq \label{addchar}
\begin{split} \Q_p &\longrightarrow 1+\C_p^{\circ \circ}\\
s &\longmapsto (\gamma^s)^\sharp
\end{split}
\eeq
is uniformly continuous. 
\begin{prop} \label{powers} Let notation be as above.   
We have for any $s \in \Q_p$ 
\beq \label{binQp2}
(\gamma^s)^\sharp = \lim_{n \to \infty} (1 + \varpi_n)^{p^ns} = \lim_{n \to \infty} \sum_{i=0}^\infty {{p^ns} \choose i}  \varpi_n^i = \sum_{q \in S} {s \choose q}^{(p)}  (\varpi^q)^\sharp \;, 
\eeq 
where the $S$-series converge  unconditionally in $\sC_\unif(\Q_p,\C_p^\circ) = \sC_\unif(\Q_p,\Z_p) \wt \, \C_p^\circ$.
\end{prop}
%\begin{rmk} The above proposition justifies our notion of convergence for an $S$-series. 
%\end{rmk}
\begin{proof} We only need to show that the $S$-series in  \eqref{binQp2} converges. 
\par
In fact, for $q>>0$, 
$|(\varpi^q)^\sharp|$
 is very small. On the other hand, assume  $s \in S$; if $v(q) < v(s)$, we have 
 $$| {s \choose q}^{(p)} | = | {p^N s \choose {p^N q}} | 
 $$
 for $N >> 0$. Therefore the $p$-adic digit expansions of $p^N s$, $p^N q$ and $p^N s - p^N q$ are  
 $$
 p^Ns = n_{N+v(s)}p^{N+v(s)} + n_{N+v(s)+1}p^{N+v(s)+1} + \dots \;,
 $$
  $$
 p^Nq = a_{N+v(q)}p^{N+v(q)} + a_{N+v(q)+1}p^{N+v(q)+1} + \dots \;,
 $$
  $$
 p^Ns - p^Nq = b_{N+v(q)}p^{N+v(q)} + b_{N+v(q)+1}p^{N+v(q)+1} + \dots \;.
 $$
Therefore, if $s \in S$, 
$$
v({s \choose q}^{(p)} ) = v({p^N s \choose {p^N q}} ) \geq (p-1) (v(s)-v(q)) \;.
$$
So, let $C \subset \Q_p$ be a compact subset. Then $v(x)$ has a minumum $M(C)$ for $x \in C$. 
We conclude that, for any $h =1,2,\dots$, if $v(q) \leq M(C) - h$ then 
\beq \label{smallvalq}
v({s \choose q}^{(p)} ) \geq (p-1)h \;,\; \forall \; s\in C \cap S\;.
\eeq
By continuity of $s \longmapsto {s \choose q}^{(p)}$ the previous inequality holds for any $s \in C$. 
\par\bigskip
So the $S$-series in \eqref{binQp2} converges along the filter of cofinite subsets of $S$. 
\end{proof}
 \end{subsection}
 \end{section}

\end{document}